\documentclass[10pt]{amsart}

\usepackage[utf8]{inputenc}
\usepackage[T1]{fontenc}
\usepackage[french,english]{babel}

\usepackage{lmodern}
\usepackage{newtxtext}

\usepackage{graphicx}	

\usepackage[a4paper,plainpages=false,colorlinks,linkcolor=bleuFonce,citecolor=rougeFonce,urlcolor=vertFonce,breaklinks]{hyperref}

\usepackage{placeins}
\usepackage{float} 

\usepackage{tikz}
\usetikzlibrary{patterns}

\usepackage[normalem]{ulem}

\usepackage{color}
\definecolor{vertFonce}	{rgb}{0,0.5,0}
\definecolor{numLignes}	{rgb}{0.17,0.57,0.7}	
\definecolor{gris}		{rgb}{0.5,0.5,0.5}
\definecolor{grisFonce}	{rgb}{0.2,0.2,0.2}
\definecolor{orange}	{rgb}{1,0.65,0.31}		
\definecolor{orangeFonce}{rgb}{1,0.4,0}
\definecolor{bleuFonce}	{rgb}{0,0,0.4}
\definecolor{rougeFonce}{rgb}{0.3,0,0}
\definecolor{rougeWord}	{rgb}{0.5,0,0}
\definecolor{vertClair}	{rgb}{0.8,1,0.8}
\definecolor{rougeClair}{rgb}{1,0.5,0.5}
\definecolor{violet}	{rgb}{0.5,0,0.5}

\usepackage{pict2e}
\usepackage{multido}

\usepackage{amsfonts,amssymb,amsthm,amsmath} 
\usepackage{dsfont}		
\usepackage{mathrsfs}
\usepackage{yfonts}
\usepackage{cancel}
 
\newtheorem{lem}{Lemma}[section]
\newtheorem{theorem}{Theorem}[section]
\newtheorem{cor}{Corollary}[section]
\newtheorem{prop}{Proposition}[section]

\newtheorem{remark}{Remark}[section]

\newcommand{\step}[1]	{\paragraph{\itshape\bfseries #1.}}

\usepackage{stmaryrd}
\SetSymbolFont{stmry}{bold}{U}{stmry}{m}{n}
\newcommand		{\subsetArrow}	{\mathrel{\ooalign{$\subset$\cr%
\hidewidth\raise-.087ex\hbox{$_\shortrightarrow\mkern-1.5mu$}\cr}}}
\newcommand		{\subsetarrow}	{\mathrel{\ooalign{$\subset$\cr%
\hidewidth\raise-1.45ex\hbox{$\vec{}\mkern6mu$}\cr}}}

\newcommand{\qqquad}	{\qquad\quad}
\newcommand{\qqqquad}	{\qquad\qquad}

\newcommand		{\N}		{\mathbb N}			
\newcommand		{\RR}		{\mathbb R}			
\newcommand		{\R}		{\RR}

\newcommand		{\Rd}		{\R^3}
\newcommand		{\Rdd}		{\R^6}

\newcommand		{\CC}		{\mathbb C}			
\newcommand		{\cH}		{\mathcal H}		
\newcommand		{\cD}		{\mathcal D}		
\newcommand		{\cN}		{\mathcal N}		
\renewcommand	{\L}		{\mathcal L}		
\newcommand		{\cW}		{\mathcal W}		

\newcommand		{\cG}		{\mathcal G}

\newcommand		{\cL}		{\mathcal L}		

\newcommand		\sfA		{\mathsf A}
\newcommand		\sfB		{\mathsf B}
\newcommand		\sfC		{\mathsf C}
\newcommand		\sfD		{\mathsf D}
\newcommand		\sfG		{\mathsf G}
\newcommand		\sfH		{\mathsf H}			
\newcommand		\sfI		{\mathsf I}
\newcommand		\sfJ		{\mathsf J}
\newcommand		\sfQ		{\mathsf Q}
\newcommand		\sfL		{\mathsf L}			
\newcommand		\sfP		{\mathsf P}			
\newcommand		\sfR		{\mathsf R}			
\newcommand		\sfT		{\mathsf T}			
\newcommand		\sfU		{\mathsf U}
\newcommand		\sfV		{\mathsf V}
\newcommand		\sfX		{\mathsf X}			

\newcommand		{\lt}			{\left}				%
\newcommand		{\rt}			{\right}			%
\renewcommand	{\(}			{\lt(}
\renewcommand	{\)}			{\rt)}
\newcommand		{\bangle}[1]	{\lt\langle #1\rt\rangle}
\newcommand		{\weight}[1]	{\bangle{#1}}	

\newcommand		{\inprod}[2]	{\bangle{#1, #2}}
\newcommand		{\com}[1]		{\lt[{#1}\rt]}		

\newcommand		{\n}[1]			{\lt\lvert #1 \rt\rvert}
		
\newcommand		{\nrm}[1]		{\lt\lVert #1\rt\rVert}
\newcommand		{\Nrm}[2]		{\nrm{#1}_{#2}}	
				
\newcommand		{\Lp}[2]		{\Nrm{#1}{#2}}

\newcommand		{\indic}	{\mathds{1}}		

\renewcommand		{\d}		{\mathop{}\!\mathrm{d}}		
\newcommand			{\bd}		{\partial}			
\newcommand			{\dpt}		{\partial_t}

\newcommand			{\dt}		{\frac{\d}{\d t}}	
\newcommand			{\ddt}[1]	{\frac{\d #1}{\d t}}

\newcommand			{\grad}		{\nabla}
\newcommand			{\lapl}		{\Delta}
\newcommand			{\Dx}		{\nabla_x}
\newcommand			{\Dv}		{\nabla_\xi}
\newcommand			{\conj}[1]	{\overline{#1}}		

\DeclareMathOperator{\cF}		{\mathcal{F}}		
\DeclareMathOperator{\im}		{Im}				
\DeclareMathOperator{\tr}		{Tr}				
\DeclareMathOperator{\diag}		{diag}

\newcommand		{\F}[1]			{\cF\!\( #1 \)}		
\renewcommand	{\Im}[1]		{\im\!\( #1 \)}		
\newcommand		{\Tr}[1]		{\tr\!\( #1 \)}		
\newcommand		{\Diag}[1]		{\diag\!\( #1 \)}

\newcommand		{\intd}			{\int_{\Rd}}
\newcommand		{\intdd}		{\int_{\Rdd}}
\newcommand		{\iintd}		{\iint_{\Rdd}}

\newcommand		{\jj}			{\mathrm{j}}	
\newcommand		{\init}			{\mathrm{in}}

\newcommand		{\fb}			{\mathfrak b}
\newcommand		{\eps}			{\varepsilon}

\newcommand		{\cC}			{\mathcal{C}}

\usepackage{braket}
\newcommand		{\Inprod}[2]	{\Braket{#1 | #2}}

\renewcommand	{\r}		{\op}				
\newcommand		{\op}		{\boldsymbol{\rho}}	

\newcommand		{\opm}		{\boldsymbol{m}}	
\newcommand		{\opw}		{\boldsymbol{w}}	
\newcommand		{\opup}		{\boldsymbol{\upsilon}}
\newcommand		{\opmu}		{\boldsymbol{\mu}}	
\newcommand		{\diagopmu}	{\rho_{\opmu}}
\newcommand		{\opnu}		{\boldsymbol{\nu}}	
\newcommand		{\tildop}		{\,\tilde{\!\op}}	

\newcommand		{\opp}		{\boldsymbol{p}}

\newcommand		{\opN}		{\boldsymbol{\rho}_N}
\newcommand		{\opNr}		{\boldsymbol{\rho}_{N:1}}
\newcommand		{\opupN}	{\opup_N}
\newcommand		{\opNop}	{\boldsymbol{\rho}_{N,\op}}	
\newcommand		{\opNf}		{\boldsymbol{\rho}_{N,f}}	
\newcommand		{\opupNop}	{\opup_{N,\op}}

\newcommand		{\Dh}		{\boldsymbol{\nabla}}	
\newcommand		{\DDh}		{\boldsymbol{\Delta}}	
\newcommand		{\Dhx}[1]	{\Dh_{\!x} #1}			
\newcommand		{\Dhv}[1]	{\Dh_{\!\xi} #1}		
\newcommand		{\Dhxj}[1]	{\Dh_{\!x_\jj} #1}		
\newcommand		{\Dhvj}[1]	{\Dh_{\!\xi_\jj} #1}	

\DeclareMathOperator{\dG}	{\d\Gamma}			
\newcommand		{\dGl}		{\dG_l}				
\newcommand		{\dGr}		{\dG_r}				
\newcommand		{\dGplr}	{\dG^+_{l,r}}		
\newcommand		{\dGprl}	{\dG^+_{r,l}}		
\newcommand		{\dGmrl}	{\dG^+_{r,l}}		
\newcommand		{\DG}[1]	{\dG\!\(#1\)}		
\newcommand		{\DGl}[1]	{\dGl\!\(#1\)}		
\newcommand		{\DGr}[1]	{\dGr\!\(#1\)}		
\newcommand		{\DGplr}[1]	{\dGplr\!\(#1\)}	
\newcommand		{\DGprl}[1]	{\dGprl\!\(#1\)}	
\newcommand		{\DGmrl}[1]	{\dGmrl\!\(#1\)}	

\newcommand		{\h}		{\mathfrak{h}}		

\DeclareMathOperator{\adj}	{ad}
\newcommand		{\ad}[1]	{\adj_{#1}}
\newcommand		{\Ad}[2]	{\ad{#1}\!\( #2 \)}

\title[\textsc{From Schrödinger to Vlasov}]{\Large From Many-body Quantum Dynamics\\ to the Hartree--Fock and Vlasov Equations\\ with Singular Potentials}
\author[\textsc{J. Chong}]{\large\textsc{Jacky J. Chong}}
\author[\textsc{L. Lafleche}]{\large\textsc{Laurent Lafleche}}
\address[J. Chong, L. Lafleche]{Department of Mathematics, The University of Texas at Austin, Austin, TX 78712, USA, {\tt jwchong@math.utexas.edu}, {\tt lafleche@math.utexas.edu}}

\address[L. Lafleche]{Institut Camille Jordan, UMR 5208 CNRS \& Université Claude Bernard Lyon 1, France, {\tt lafleche@math.univ-lyon1.fr}}

\author[\textsc{C. Saffirio}]{\large\textsc{Chiara Saffirio}}
\address[C. Saffirio]{Department of Mathematics and Computer Science, University of Basel, 4051 Basel, Switzerland, {\tt chiara.saffirio@unibas.ch}}


\begin{document}

\begin{abstract}
	We obtain the combined mean-field and semiclassical limit from the $N$-body Schrödinger equation for fermions interacting via singular potentials. To obtain the result, we first prove the uniformity in Planck's constant $h$ propagation of regularity for solutions to the Hartree--Fock equation with singular pair interaction potentials of the form $\pm |x-y|^{-a}$, including the Coulomb and gravitational interactions.
		
	\noindent In the context of mixed states, we use these regularity properties to obtain quantitative estimates on the distance between solutions to the Schrödinger equation and solutions to the Hartree--Fock and Vlasov equations in Schatten norms. For $a\in(0,1/2)$, we obtain local-in-time results when $N^{-1/2} \ll h \leq N^{-1/3}$. In particular, it leads to the derivation of the Vlasov equation with singular potentials. For $a\in[1/2,1]$, our results hold only on a small time scale, or with an $N$-dependent cutoff.
\end{abstract}

\begingroup
\def\uppercasenonmath#1{} 
\let\MakeUppercase\relax 
\maketitle
\thispagestyle{empty} 
\endgroup

\bigskip

\renewcommand{\contentsname}{\centerline{Table of Contents}}
\setcounter{tocdepth}{2}	
\tableofcontents

 
\part{\large Introduction}

{\section{Background}}

\subsection{The Equations}

	We consider a system of $N$ identical fermions with unit mass interacting through a pair potential $K(x-y)$. The state of the system is described by an $N$-body  anti-symmetric wave function $\psi_N = \psi_N(t, x_1, x_2, \dots, x_N)$ belonging to the Hilbert space of square-integrable complex-valued functions $\h=L^2(\R^{3N},\CC)$, with evolution given by the $N$-body Schrödinger equation
	\begin{equation}\label{eq:Schrodinger}
		i\hbar\,\dpt \psi_N = \sum_{k=1}^N-\frac{\hbar^2}{2}\,\Delta_{x_k} \psi_N + \sum_{1\leq k<l\leq N} K(x_k-x_l)\,\psi_N,
	\end{equation}
	where $h$ is the Planck constant and $\hbar = \frac{h}{2\,\pi}$ is the reduced Planck constant. In applications, one is typically interested in systems where the number of particles $N$ is large, thus making the microscopic description given by the solution to Equation~\eqref{eq:Schrodinger} not suitable for studies. In fact, the high dimensionality of the problem presents a formidable barrier for understanding qualitative behaviors of the many-body dynamics from the wave function at the microscopic scale. Instead, one could consider the problem at a macroscopic scale and look at the classical phase space distributions of particles $f=f(t,x,\xi)$, where $(x,\xi)\in\Rd\times\Rd$ are the spatial and momentum variables. In particular, we consider scales where the dynamics of a large number of interacting particles can be approximated by the Vlasov equation
	\begin{equation}\label{eq:Vlasov}
		\partial_t f + \xi\cdot\Dx f + E_f\cdot\Dv f = 0,
	\end{equation}
	where $E_f = -\nabla V_f$ is the force field corresponding to the mean-field potential
	\begin{equation*}
		V_f(x) = (K * \rho_f)(x) = \intd K(x-y)\,\rho_f(y)\d y
	\end{equation*}
	and $\rho_f$ is the spatial distribution of particles defined by
	\begin{equation}\label{eq:spatial_density}
		\rho_f(x) = \intd f(x,\xi)\d \xi.
	\end{equation}

	To explore the connection between the microscopic and the  macroscopic scales of the system,  we consider an intermediate mean-field quantum equation. Roughly speaking, we approximate the many-body effects exerted by the system on each particle by an
	effective interaction potential obtained by averaging the pair potential
	$K$ with the underlying spatial density of the system. To draw a parallel with classical mechanics, one could consider the mean-field equation called the Hartree equation which is the quantum analogue of the Vlasov equation. 
	More precisely, let us take a positive self-adjoint trace class operator $\op$ acting on $L^2(\Rd,\CC)$, which can be seen as a positive linear convex combination  of projections onto one-particle wave functions. We use the same notation to denote both the operator $\op$ and its integral kernel $\op(x,y)$. Here, $\op$ plays the role of the quantum one-particle phase space distribution of particles. Moreover, the effective one-particle Hamiltonian is given by $H = -\frac{\hbar^2}{2}\, \Delta + V_{\op}$, called the Hartree Hamiltonian, where $V_{\op}$ is the mean-field potential $V_{\op} = K*\rho(x)$ and $\rho(x)$ is the quantum spatial distribution of particles defined by 
	\begin{equation}\label{eq:spatial_density_quantum}
		\rho(x) = \Diag{\op}(x) := h^3\, \op(x,x).
	\end{equation}
	With these notations, the Hartree equation reads 
	\begin{equation*}
		i\hbar\,\partial_t \op = \com{H,\op},
	\end{equation*}
	where $\com{A, B} := AB-BA$ is the commutator of the operators $A$ and $B$. If the particles obey the Fermi statistics, a more accurate description of their evolution is given by the Hartree--Fock equation
	\begin{align}\label{eq:Hartree--Fock}
		i\hbar\,\partial_t \op &= \com{H_{\op},\op}, & H_{\op} &= -\frac{\hbar^2}{2}\, \Delta + V_{\op} - h^3\, \sfX_{\op}\,,
	\end{align}
	where the exchange term $\sfX_{\op}$ is the operator with integral kernel
	\begin{equation}\label{eq:exchange_operator}
		\sfX_{\op}(x,y) = K(x-y)\, \op(x,y).
	\end{equation}
	
\subsection{Mean-field and Semiclassical Scalings}
	
	Our goal in this paper is to study simultaneously the mean-field limit, corresponding to the approximations made when the number of particles $N$ is large and each pair interaction is weak, and the semiclassical limit, corresponding to a change of scales where the Planck constant $h$ becomes negligible. Let us elaborate more on the two scalings.   
	
	To understand of the dynamics generated by the many-body Schrödinger equation~\eqref{eq:Schrodinger} at different scales, it is convenient to recast the equation in its dimensionless form. Suppose $L$ is some characteristic length of the problem and $T$ is some characteristic timescale, then we define the following dimensionless variables 
	\begin{equation*}
		\widetilde x := x/L \quad \text{ and } \quad \widetilde T := t/T.
	\end{equation*}
	We also recast the interaction potential in its dimensionless form via the change of scale
	\begin{equation*}
		\widetilde K(\widetilde x):= \frac{N\, T^2}{m\, L^2}K( x)=\frac{N\, T^2}{m\, L^2}K(L\, \widetilde x)
	\end{equation*}
	where $m$ denotes the mass which we have set to one. If we define the rescaled dimensionless parameter 
	\begin{equation*}
		\widetilde \hbar = \frac{\hbar\, T}{m\, L^2}
	\end{equation*}
	and the new rescaled wave function 
	\begin{equation*}
		\widetilde \psi_N(\widetilde t, \widetilde x_1,\ldots, \widetilde x_N) := L^{dN/2}\,\psi_N(t, x_1, \ldots, x_N),
	\end{equation*}
	then multiplying Equation~\eqref{eq:Schrodinger} by $\frac{T^2}{m\,L^2}$ yields the dimensionless equation 
	\begin{equation*}
		i\,\widetilde{\hbar}\,\partial_{\,\widetilde t\,} \widetilde{\psi}_N = \sum_{k=1}^N -\frac{\widetilde{\hbar}^2}{2}\,\Delta_{\widetilde x_k} \widetilde{\psi}_N + \frac{1}{N}\,\sum_{1\leq k<l\leq N} \widetilde{K}(\widetilde x_k-\widetilde x_l) \, \widetilde{\psi}_N.
	\end{equation*}
	Moreover, in the case of an homogeneous interaction of the form $K(x) = \kappa \n{x}^{-a}$ for some parameter $\kappa\in\R$, this gives $\widetilde{K}(\widetilde x) = \widetilde{\kappa}\, \n{\widetilde x}^{-a}$ where $\widetilde{\kappa} = \kappa\,N\,T^2/(m\,L^{2+a})$. From here, we consider the case of space-time scales where $\widetilde \kappa$ is of order 1 and we simply set $\widetilde \kappa = 1$. This provides a $N^{-1}$ prefactor in front of the interaction potential which is usually referred to as the mean-field scaling. In this class of scales, the dimensionless parameter $\widetilde \hbar$ is of the order $L^a/(\kappa NT)$. 
	Furthermore, we shall refer to the scale where $\widetilde \hbar$ becomes negligible  as the semiclassical regime. For convenience, let us express $L$ and $T$ in terms of the parameters $N, \kappa$ and $\widetilde\hbar$ as follows 
		\begin{equation}\label{eq:characteristic_length_scale}
			L =\, \(\frac{\hbar^2}{m\,\kappa\, N\,\widetilde\hbar^2}\)^\frac{1}{2-a},\quad\quad\quad
			T =\, \frac{\widetilde \hbar \, m}{\hbar}\(\frac{\hbar^2}{m\,\kappa\, N\,\widetilde\hbar^2}\)^\frac{2}{2-a} .
		\end{equation}
	From here, we impose the condition that $N\,\widetilde\hbar^2\ll \kappa^{-1}$ to guarantee that $L\gg 1$. In particular, we could set $\kappa = N^{-1}$.

	While $\hbar$ and $N$ can a priori be considered as independent parameters, certain constraints arise when dealing with fermions. In Section~\ref{sec:constraints_scaling}, it is explained that the Pauli principle imposes a limitation on $N\,h^3$, which must remain bounded to ensure convergence of the one-particle density operator to a nonzero function on the phase space. This is in contrast to bosonic systems, which are system of particles that have permutation symmetry as opposed the anti-permutation symmetry of fermions, where the Pauli principle does not apply.
	
	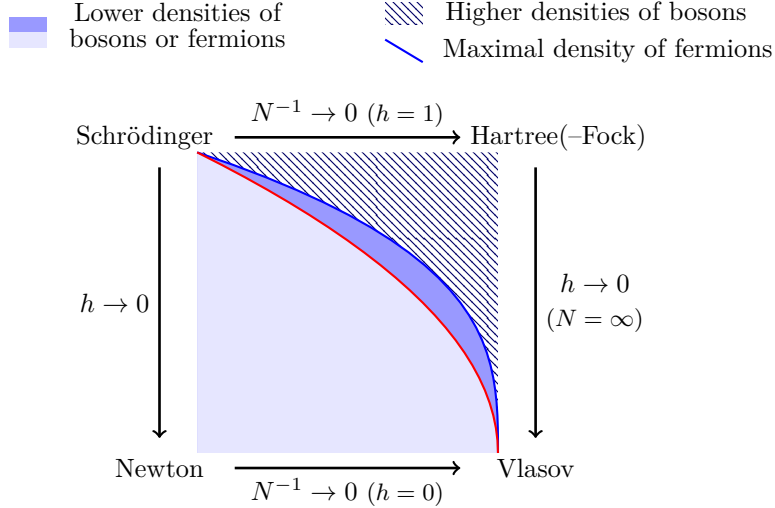
\begin{figure}[ht]\centering
		\begin{tikzpicture}
			%
			\fill[fill=blue!40!white]
			(-4.5,5.8)--(-4.5,5.6)--(-4,5.6)--(-4,5.8) node[midway, right=5]{\shortstack{Lower densities of\\bosons or fermions}};
			\fill[fill=blue!10!white]
			(-4.5,5.6)--(-4.5,5.4)--(-4,5.4)--(-4,5.6);
			
			\fill[pattern=north west lines,pattern color = bleuFonce]
			(0.5,6.0)--(0.5,5.7)--(1,5.7)--(1,6.0) node[midway, right=5]{Higher densities of bosons};
			\draw[line width = 0.8, blue]
			(1,5.2)--(0.5,5.5) node[midway, black, right=11]{Maximal density of fermions};
			\node[align=left] at (-2.7,4.2) {Schrödinger};
			\node[] at (2.8,4.2) {Hartree(--Fock)};
			\node[] at (-2.5,-0.2) {Newton};
			\node[] at (2.5,-0.2) {Vlasov};
			%
			%
			\fill[pattern=north west lines,pattern color = bleuFonce] 
			(2,0)--(2,4)--(-2,4)
			-- plot [domain=-2:2, samples=120] (\x, {4*((2-\x)/4)^(1/3)})
			-- cycle;
			\fill[fill=blue!40!white]
			(2,0)--(-2,0)--(-2,2)
			-- plot [domain=-2:2, samples=100] (\x, {4*((2-\x)/4)^(1/3)})
			-- cycle;
			\fill[fill=blue!10!white]
			(2,0)--(-2,0)--(-2,2)
			-- plot [domain=-2:2, samples=100] (\x, {4*((2-\x)/4)^(1/2)})
			-- cycle;
			%
			%
			\draw[->,line width = 1](-1.5,4.2)--(1.5,4.2) node[midway, above]{$N^{-1}\to 0$ \small$\(h=1\)$};
			\draw[<-,line width = 1](-2.5,0.2)--(-2.5,3.8) node[midway,left]{$h\to 0$};
			\draw[<-,line width = 1](2.5,0.2)--(2.5,3.8) node[midway,right]{\shortstack{$h\to 0$\\\\\small$\(N=\infty\)$}};
			\draw[->,line width = 1](-1.5,-0.2)--(1.5,-0.2) node[midway, below]{$N^{-1}\to 0$ \small$\(h=0\)$};
			%
			%
			\draw [domain=0:4, line width=0.8,blue] plot ({-(\x/4)^3*4+2}, {\x});
			\draw [domain=0:4, dashed, line width=0.8, red] plot ({-(\x/4)^2*4+2}, {\x});
		\end{tikzpicture}
		\caption{The different scalings for the combined mean-field and semiclassical limits. The dashed (red) curve corresponds to the equation $h = N^{-1/2}$ and the continuous (blue) curve to $h = N^{-1/3}$.}\label{fig:scaling}
	\end{figure}
	
	In addition, observe that the particle density, defined as the number of particles per unit of characteristic volume, scales as $N\,L^{-3}$, where $N$ is the total number of particles. By using the scaling given by Formula~\eqref{eq:characteristic_length_scale}, we can express the density as follows
	\begin{equation*}
		\frac{N}{L^3}\simeq N^{\frac{5-a}{2-a}}\, h^{\frac{6}{2-a}}\,\kappa^{\frac{3}{2-a}}.
	\end{equation*}
	This explains why the region of Figure~\ref{fig:scaling} closer to the Hartree--Fock equation corner corresponds to relatively higher densities, while the region below corresponds to relatively lower densities. Moreover, we should note that in this work, we consider $h$ to satisfy the constraint $N^{-\frac12}\ll h \leq CN^{-\frac13}$, which corresponds to the dark shaded region in the figure.  It should be noted that the constraint $N^{-\frac12} \ll h$, meaning $N \, h^2 \to \infty$, could be technical and it arises in the proof of the main result (See Proposition~\ref{prop:propagator_bound}).
	
	With a little abuse of notation and language, we shall drop the tildes and study the equation
	\begin{align}\label{eq:Schrodinger_rescaled}
		i\hbar\,\dpt \psi_N &= H_N\,\psi_N, & H_N &= \sum_{k=1}^N -\frac{\hbar^2}{2}\Delta_{x_k} + \frac{1}{N}\,\sum_{1\leq k<l\leq N} K(x_k-x_l)
	\end{align}
	 where $N$ is large and $\hbar$ is small, and with
	\begin{equation}\label{eq:force}
		K(x) = \frac{\kappa}{\n{x}^a},
	\end{equation}
	where $\kappa\in\R$ is of order $1$ and $a\in(0,1]$.
	
	More precisely, we study the time evolution of $N$-body fermionic mixed states, which are self-adjoint, positive trace class operators of rank larger than one. By the spectral theorem, they can be expressed in the following way
	\begin{equation}\label{eq:mixed-state-rapresentation}
		\op_N = \sum_{j=1}^\infty \lambda_j \ket{\psi_j}\!\!\bra{\psi_j} \quad \text{ with }\quad \lambda_j \ge 0
	\end{equation}
	where $\set{\psi_j}_{j\in \N} \subset \h^{\otimes N}$ is an orthonormal set of anti-symmetric wave functions. The operator $\op_N$ is called a pure state provided it is a rank one projection, that is, $\op_N=\ket{\psi_N}\!\!\bra{\psi_N}$. The time evolution equation for density operators is given by the Liouville--von Neumann equation
	\begin{equation}\label{eq:Liouville}
		i\hbar\,\dpt \op_{N} = \com{H_N,\op_{N}}
	\end{equation}
	where the Hamiltonian $H_N$ is given in Equation~\eqref{eq:Schrodinger_rescaled}, which is the quantum analogue of the classical Liouville equation, equivalent to the $N$-body Newton laws.

\subsection{State of the Art}

	Both the problems of the mean-field limits and the semiclassical limits are well-known questions that are largely addressed in the literature. However, the derivation of the Vlasov--Poisson equation, i.e. the case of the Coulomb and gravitational potentials, remains an open problem, both in the case of quantum mechanics and in the case of classical Newton's laws.
	
	\subsubsection{The Classical Mean-Field Limit} In the context of classical mechanics, the problem of justifying the Vlasov equation~\eqref{eq:Vlasov} starting from the dynamics of $N$-particles obeying Newton's laws  was first proven for twice differentiable potentials in the pioneering works by Neunzert and Wick~\cite{neunzert_theoretische_1972}, Braun and Hepp~\cite{braun_vlasov_1977}, and then by Dobrushin~\cite{dobrushin_vlasov_1979} using the Wasserstein--Monge--Kantorovich distance (see also~\cite{spohn_large_1991} for an introduction to the topic). The class of potentials was then extended to less regular potentials but still locally H\"older continuous by Hauray and Jabin \cite{hauray_$n$-particles_2007, hauray_particle_2015}, which was later improved by Jabin and Wang using entropy methods in \cite{jabin_mean_2016}, where the potential is only required to be bounded.
	
	From another point of view, it was also proved in \cite{hauray_particle_2015, boers_mean_2015} that it is possible to obtain the mean-field limit for potentials with a vanishing cutoff, converging to potentials almost as singular as the Coulomb potential when $N\to\infty$. This is in particular interesting from a numerical point of view. These results were then improved by Lazarovici \cite{lazarovici_vlasovpoisson_2016}, allowing the cutoff potential to converge to the Coulomb potential, and by Lazarovici and Pickl \cite{lazarovici_mean_2017}, with a $N$-dependent cutoff of the order of the inter-particle distance.
	
	\subsubsection{Combined Mean-Field and Semiclassical Limits} The first rigorous derivation of the Vlasov equation~\eqref{eq:Vlasov} from the $N$-body Schrödinger equation~\eqref{eq:Schrodinger} was proved by Narnhofer and Sewel \cite{narnhofer_vlasov_1981} in the case of smooth potentials, with $\hbar = N^{-1/3}$. Subsequently, the restriction on the potential was substantially relaxed by Spohn~\cite{spohn_vlasov_1981} to twice differentiable potentials. For the same kind of potentials, a more explicit rate of convergence without assuming $\hbar = N^{-1/3}$ was later obtained by Graffi, Martinez, and Pulvirenti~\cite{graffi_mean-field_2003} in the case of weak convergence, and more recently by Golse and Paul~\cite{golse_schrodinger_2017} in the quantum Wasserstein metrics, and by Chen, Lee and Liew for fermions~\cite{chen_combined_2021} {in the scaling $\hbar = N^{-1/3}$}. 
	
	\subsubsection{Quantum Mean-Field Limit} It is also possible to first look at the mean-field limit with $\hbar=1$, i.e. without taking the semiclassical limit, leading to the Hartree and the Hartree--Fock equations. In this case, the situation is better understood, even for the Coulomb and gravitational potentials. For bosons, weak convergence was proved in \cite{bardos_weak_2000, erdos_derivation_2001, bardos_derivation_2002}, and explicit rates in stronger norms were obtained in \cite{rodnianski_quantum_2009, grillakis_second-order_2010, pickl_simple_2011, chen_rate_2011, kuz_rate_2015, mitrouskas_bogoliubov_2019, chen_rate_2018, paul_size_2019}. For fermions, weak convergence was proved in \cite{bardos_mean_2003} for bounded potentials, and estimates in trace norm and singular potentials such as the Coulomb potential were obtained in \cite{frohlich_microscopic_2011, bach_kinetic_2016, petrat_new_2016, petrat_hartree_2017}.
	
	Some of these results have been extended by taking into account the semiclassical parameter $\hbar$. For fermions, taking $\hbar=N^{-1/3}$, convergence of the Husimi transform has been proven in \cite{elgart_nonlinear_2004} for analytic interactions and short times. Schatten norms estimates have been obtained in  \cite{benedikter_mean-field_2014, benedikter_mean-field_2016, petrat_new_2016} for at least twice differentiable potentials. Assuming a certain semiclassical structure on the solution of the Hartree equation, a result was obtained in the case of pure states and singular potentials in \cite{porta_mean_2017, saffirio_mean-field_2018}. 
	
	For bosons, results were obtained for at least twice differentiable potentials in \cite{golse_mean_2016, golse_derivation_2018, golse_empirical_2019}. 
	
	\subsubsection{Semiclassical Limit} Another possible direction is to look only at the semiclassical limit $\hbar\to 0$, either for the number of particles $N$ fixed or in the mean-field regime. This last case corresponds to going from the Hartree or the Hartree--Fock equation to the Vlasov equation. In the case of the Hartree equation, this was proved in \cite{lions_sur_1993, markowich_classical_1993} in weak topology, but including singular potentials such as the Coulomb interaction (see also \cite{figalli_semiclassical_2012} for the case of quantum Liouville dynamics). Explicit rates in stronger norms were then obtained in \cite{athanassoulis_semiclassical_2011, amour_semiclassical_2013, benedikter_hartree_2016, golse_schrodinger_2017} for at least twice differentiable potentials, and then in \cite{lafleche_propagation_2019, saffirio_semiclassical_2019, saffirio_hartree_2020, lafleche_global_2021, lafleche_strong_2021} for singular interactions.
	
	 To our knowledge our work is the first one addressing mixed states (see \eqref{eq:mixed-state-rapresentation}) in the case of singular interactions of the form \eqref{eq:force} and proving in this context the approximation of the mean-field dynamics with the Hartree--Fock equation on time scales of order one when $a\in(0,1/2)$ and to time scales of order $\sqrt{\hbar}$ when $a\in[1/2,1]$. See also Remark~\ref{remark:time-validity}.

\subsection{Constraints on the Scalings}\label{sec:constraints_scaling}
    The Fourier transform is defined by
    \begin{equation}
        \widehat g(\xi) = \intd e^{-2\pi i\,x\cdot \xi}\, g(x)\d x
    \end{equation}
    for $g \in L^2(\R^3)$. We also adopt the following conventions for the time-dependent operator solution $\op = \op(t)$ to the Hartree--Fock equation~\eqref{eq:Hartree--Fock}
	\begin{equation}\label{eq:scaling_op}
		\begin{split}
			\Nrm{\op}{\infty} &= \cC_\infty, \vspace{3pt}
			\\
			\Tr{\op} &= h^{-3}.
		\end{split}
	\end{equation}
	for some constant $\cC_\infty$. In particular, $\intd \rho(x)\d x = h^3 \Tr{\op} = 1$. For such an operator, we define its Wigner transform by
	\begin{equation*}
		f_{\op}(x,\xi) := \intd e^{-i\,y\cdot\xi/\hbar} \,\op(x+\tfrac{y}{2},x-\tfrac{y}{2})\d y,
	\end{equation*}
	so that it is a function of the phase space with mass $\iint f_{\op}\d x\d\xi = h^3\Tr{\op} = 1$.
	It is well known that, under some regularity assumptions, the Wigner transform of solutions $\op$ to the Hartree--Fock equation~\eqref{eq:Hartree--Fock} converge to solutions of the Vlasov equation~\eqref{eq:Vlasov} in the semiclassical limit $h\to 0$ (see e.g. \cite{lions_sur_1993}). We refer to \cite{lions_sur_1993} for a listing of the properties of the Wigner transform. One of them is the fact that
	\begin{equation}\label{eq:L2_norm}
		\Nrm{f_{\op}}{L^2(\Rdd)} = h^{\frac{3}{2}} \Nrm{\op}{2},
	\end{equation}
	where we denote by 
	\begin{equation}\label{eq:Schatten}
		\Nrm{\op}{p} = \(\Tr{\n{\op}^p}\)^\frac{1}{p}
	\end{equation}
	the Schatten norm of order $p$. Here, the absolute value of an operator $A$ is defined by $\n{A} = \sqrt{A^*A}$. Since we want to address the case when $f_{\op}$ converges in $L^2(\Rdd)$ to a solution $f$ of the Vlasov equation, this implies that $h^{\frac{3}{2}} \Nrm{\op}{2} \underset{h\to 0}{\rightarrow} \Nrm{f}{L^2(\Rdd)}$, so that $\Nrm{\op}{2}$ is of size $h^{-3/2}$.
	
	For an $N$-particle density operator $\op_{N}$, we look at its corresponding one-particle reduced density operator $\op_{N:1}$ defined as the partial trace of $\op_{N}$ with respect to the variables $2$ to $N$, that is, 
    \begin{equation*}
        \op_{N:1} = \tr_{2, \ldots, N}(\op_{N}).
    \end{equation*}
    Since we also want the corresponding Wigner transform $f_{N:1}$ of the operator $\op_{N:1}$ of the $N$-particle density operator to converge to $f$, we have as well
	\begin{equation}\label{eq:CV_of_L2_norm}
		\Nrm{f_{N:1}}{L^2(\Rdd)} = h^{\frac{3}{2}} \Nrm{\op_{N:1}}{2} {\longrightarrow} \Nrm{f}{L^2(\Rdd)}\quad \mbox{as } N\to\infty\ \mbox{and } h\to 0.
	\end{equation}
	However, in the case of fermions, we also know that (see for instance~\cite[Equation 12.5.12]{lieb_analysis_2001}, or \cite[Theorem 8.4]{solovej_many_2014})
	\begin{equation}\label{eq:fermionic_inequality}
		0 \leq \op_{N:1} \leq \frac{\Tr{\op_{N:1}}}{N} = \frac{h^{-3}}{N}.
	\end{equation}
	Therefore, by bounding the square of the Hilbert--Schmidt norm by the product of the trace norm and the operator norm, we deduce that
	\begin{equation}\label{eq:fermionic_inequality2}
		\Nrm{\op_{N:1}}{2}^2 \leq \Nrm{\op_{N:1}}{1} \Nrm{\op_{N:1}}{\infty} \leq \frac{h^{-6}}{N}.
	\end{equation}
	Combining Inequality~\eqref{eq:fermionic_inequality2} and Formula~\eqref{eq:L2_norm} with $\op = \op_{N:1}$, we obtain the bound
	\begin{equation}
		h \leq \cC_2^{-2/3} \, N^{-1/3},
	\end{equation}
	where $\cC_2 = \Nrm{f_{N:1}}{L^2(\Rdd)}$ converges to $\Nrm{f}{L^2(\Rdd)}$, which remains of order $1$. Hence, we are mainly interested in the case when $N\,h^3$ is bounded above by a constant independent of $N$ and $h$.  In particular, the case when $N\,h^3$ is of order one is called the critical scaling regime. This corresponds to the blue line in Figure~\ref{fig:scaling}.
	
	Notice that our analysis still makes sense if $N\,h^3 \to \infty$. However, in this situation, even though the solution to the $N$-body Schrödinger equation and the solution to the Hartree--Fock equation are close, they will not converge in the semiclassical limit to a nontrivial solution of the Vlasov equation, but to zero.

\section{Functional Spaces}

\subsection{Semiclassical Spaces}
	
	Since we want to look at the convergence in the semiclassical limit $\hbar\to 0$ towards probability distributions of the phase space, we define the semiclassical versions of the Lebesgue norms of the phase space as the following scaled Schatten norms 
	\begin{equation}\label{eq:def_norm}
		\Nrm{\op}{\L^p} = h^{\frac{3}{p}} \Nrm{\op}{p} = h^{\frac{3}{p}} \Tr{\n{\op}^p}^\frac{1}{p}.
	\end{equation}
	More generally, given any positive operator $m$, we define the corresponding weighted spaces by the norm $\Nrm{\op}{\L^p(m)} = \Nrm{\op\,m}{\L^p}$.  With this choice of scaling of the norm, notice that for any operator $\op\geq 0$ satisfying the scaling assumptions~\eqref{eq:scaling_op}, one obtains
	\begin{equation*}
		\Nrm{\op}{\L^1} = 1, \qqquad \Nrm{\op}{\L^2} = \Nrm{f_{\op}}{L^2(\Rdd)},
		\qqquad
		\Nrm{\op}{\L^\infty} = \cC_\infty.
	\end{equation*}
	One useful property of the norm defined by~\eqref{eq:def_norm} is that it is compatible with taking powers of the operator, in the sense that for any $c>0$, $\Nrm{\op^c}{\L^p} = \Nrm{\op}{\L^{pc}}^c$. In particular, in the rest of the paper we will often work with the operator $\sqrt{\op}$, which satisfies, as one would expect, $\Nrm{\sqrt{\op}}{\L^2}= 1$ and $\Nrm{\sqrt{\op}}{\L^\infty} = \sqrt{\cC_\infty}$. 
    
    The fact that these norms are good analogues of the classical Lebesgue norms can be better understood in light of particular examples. One class of examples  is when the density operator has the form $f(x)\,g(\opp)$, where $\opp = -i\hbar\nabla$ is the momentum operator. Then the Kato--Seiler--Simon inequality \cite[Theorem~4.1]{simon_trace_2005} reads
	\begin{equation}\label{eq:Kato_Seiler_Simon}
		\Nrm{f(x)\,g(\opp)}{\L^p} \leq \Nrm{f}{L^p}\Nrm{g}{L^p} \ \text{ if } p\in[2,\infty),
	\end{equation}
	with equality when $p=2$, and where $L^p = L^p(\Rd)$. It is the analogue of the identity $\Nrm{f(x)\,g(\xi)}{L^p_{x,\xi}} = \Nrm{f}{L^p}\Nrm{g}{L^p}$. Another class of examples is the class of Toeplitz operators, namely when $\op$ is an averaging of coherent states, as presented in Remark~\ref{remark:on_initial_data_of_Hartree}.
	
	We also want to consider the semiclassical version of Sobolev spaces of the phase space. Thus, as in \cite{lafleche_strong_2021}, we introduce the following operators
	\begin{equation}
		\Dhx \op := \com{\nabla,\op} \quad \text{ and } \quad
		\Dhv \op := \com{\frac{x}{i\hbar},\op},
	\end{equation}
	which can be seen as an application of the correspondence principle of quantum mechanics. More precisely, one can observe that these operators correspond to the gradients of the Wigner transform, since
	\begin{equation*}
		f_{\Dhx \op} = \Dx f_{\op}\quad \text{ and } \quad f_{\Dhv \op} = \Dv f_{\op}.
	\end{equation*}
	In the rest of the paper, we shall refer to $\Dhx\op$ and $\Dhv\op$ as the first-order quantum gradients or, simply, the quantum gradients.

	We define the semiclassical analogues of the weighted kinetic homogeneous Sobolev norms by
	\begin{align*}
		\Nrm{\op}{\dot{\cW}^{1,p}(m_n)}^p &:= \sum_{\jj=1}^3 \Nrm{\Dhvj{\op}}{\L^p(m_n)}^p + \Nrm{\Dhxj{\op}}{\L^p(m_n)}^p,
		\\
		\Nrm{\op}{\dot{\cW}^{1,\infty}(m_n)} &:= \sup_{\jj\in \{1,2,3\}} \( \Nrm{\Dhvj{\op}}{\L^\infty(m_n)}, \Nrm{\Dhxj{\op}}{\L^\infty(m_n)}\),
	\end{align*}
	and consider the particular case of the weight defined for $n\in \N$ by
	\begin{equation}\label{eq:weight}
		m_n := 1 + \n{\opp}^n.
	\end{equation}
	where $\opp = -i\hbar\nabla$ so $\n{\opp}^2 = -\hbar^2\Delta$. We also define the inhomogeneous version by
	\begin{equation}
		\Nrm{\op}{\cW^{1,p}(m_n)}^p := \Nrm{\op}{\L^p(m_n)}^p + \Nrm{\op}{\dot{\cW}^{1,p}(m_n)}^p,
	\end{equation}
	with the usual modification when $p=\infty$. In particular, for $p=2$, we have that $\Nrm{\op}{\cW^{1,2}} = \Nrm{f_{\op}}{H^1(\Rdd)}$.
	
\subsection{Fermionic Fock Space}
	
	Let $\h^{\wedge N} := \h\wedge \cdots \wedge \h$ denote the $n$-fold anti-symmetric tensor product of $\h=L^2(\R^3, \CC)$. We define the fermionic (anti-symmetric) Fock space over $\h$ to be the closure of
	\begin{equation}
		 \cF(\h)= \cF := \CC\oplus \bigoplus_{n=1}^\infty \h^{\wedge n}
	\end{equation}
	with respect to the norm induced by the endowed inner product
	\begin{equation}
		\Inprod{\psi}{\varphi}_{\cF} = \conj{\psi^{(0)}}\, \varphi^{(0)} + \sum_{n\geq 1} \int_{\R^{3n}} \conj{\psi^{(n)}(\underline{x}_n)}\,\varphi^{(n)}(\underline{x}_n) \d x_1\cdots \d x_n,
	\end{equation}
	for any pair of vectors $\psi = (\psi^{(0)}, \psi^{(1)}, \ldots)$ and $\varphi = (\varphi^{(0)}, \varphi^{(1)}, \ldots)$ in $\cF$ where $\underline{x}_k = (x_1,\ldots, x_k)\in \RR^{3k}$. For simplicity of notation, we will also denote the closure by $\cF$.  The vacuum, defined by the vector $$\Omega_{\cF} = (1, 0, \ldots)\in\cF,$$
	describes the state with no particles.
	We define the number of particles operator by
	\begin{equation}\label{eq:def_number_operator}
		\cN \psi = \(n\,\psi^{(n)}\)_{n\in\N}
	\end{equation}
	whose meaning can be interpreted as counting the number of particles in each sector of $\cF$. A class of operators on $\cF$ that is important to our studies is the class of mixed states on $\cF$, which are high rank density matrices on $\cF$. More specifically, we are interested in operators of the form 
	\begin{equation}\label{eq:diagonalisation_op_N}
		\opN := \sum_{\jj\in\N} \lambda_\jj \ket{\psi_\jj}\!\!\bra{\psi_\jj}, 
	\end{equation}
	for some orthonormal set $\psi_\jj$ of vectors of $\cF$ with the normalization
	\begin{equation}
		\Tr{\opN} = \sum_\jj \lambda_\jj = h^{-3} \ \ \text{ and } \ \ h^3 \Tr{\cN\,\opN} = N.
	\end{equation}
	Here, $N$ is the mean number of particles. Moreover, for each $(n,m)\in\N^2$, we define the operator $\opN^{(n,m)}$ as the operator with the integral kernel 
	\begin{equation}
		\opN^{(n,m)}(\underline{x}_{n},\underline{y}_{m}) = \sum_{\jj\in\N} \lambda_\jj \,\psi_\jj^{(n)}(\underline{x}_{n})\,\conj{\psi_\jj^{(m)}(\underline{y}_{m})}.
	\end{equation}
	As in the case of the one-particle operator given in Equation \eqref{eq:def_norm}, we define the Fock space semiclassical Schatten norms by
	\begin{equation}
		\Nrm{\opN}{\L^p(\cF)} := h^\frac{3}{p} \tr_{\cF}\(\n{\opN}^p\)^\frac{1}{p},
	\end{equation}
	so that $\Nrm{\opN}{\L^1(\cF)} = 1$ and $\Nrm{\cN\opN}{\L^1(\cF)} = N$. We also define the one-particle reduced density matrix, i.e. the analogue of the classical one-particle marginal, by
	\begin{equation*}
		\opNr := \sum_{n\in\N} \frac{n}{N}\,\tr_{2..n}\!\(\opN^{(n,n)}\),
	\end{equation*}
	where $\tr_{2..n}$ indicates the partial trace with respect to all variables except the first.

\section{Main Results}

\subsection{Propagation of Regularity}

	Our first result gives the local-in-time and uniform in $\hbar$  propagation of the regularity of the solution to the Hartree--Fock equation~\eqref{eq:Hartree--Fock}. Let us notice that there are no constraints on the scaling here since we are only considering the mean-field equation. Moreover, this result also holds uniformly in $\hbar$ in the case of the Coulomb potential. 
    
    Recall that we work exclusively with the singular interaction potential $K(x) = \kappa\, |\cdot|^{-a}$ for $0<a\le 1$. We define the parameter
	\begin{equation}
		\fb := \frac{3}{a+1}
	\end{equation}
	which corresponds to the integrability of the force field since $\nabla K \in L^{\fb,\infty}$.

	\begin{theorem}[Propagation of regularity]\label{thm:regularity}
		Let $a\in (0,1]$, $m_n=1+\n{\opp}^n$ with $n\in 2\N$ satisfying $n\geq 6$ and $\op$ be a solution to the Hartree--Fock equation~\eqref{eq:Hartree--Fock} with initial condition $\op^\init\in \L^\infty(m_n)$ satisfying \eqref{eq:scaling_op} and such that
		\begin{equation}\label{eq:regularity0_0}
			\op^\init \in \cW^{1,2}(m_n)\cap \cW^{1,4}(m_{n-2}).
		\end{equation}
		Then there exists $T>0$ such that
		\begin{equation}
			\op \in L^\infty([0,T], \cW^{1,2}(m_n)\cap \cW^{1,4}(m_{n-2})),
		\end{equation}
		uniformly in $\hbar \in (0, 1)$.
	\end{theorem}

	\begin{remark}
		In the case when $a\in (0, \tfrac12)$, we further extend the local-in-time and uniform-in-$\hbar$ propagation of regularity result of Theorem \ref{thm:regularity} to a global-in-time result  in \cite{chong_global--time_2022}. 
	\end{remark}
	
	\begin{remark}[On the initial data of the Hartree equation]\label{remark:on_initial_data_of_Hartree}
		Define $\varphi(x) = e^{-\pi\n{x}^2/2}$ and $\varphi_{x,\xi}(y) := \frac{1}{h^{9/4}} \varphi\!\(\frac{y-x}{\sqrt{h}}\) e^{i\,y\cdot\xi/\hbar}$. Then one can define an approximation of the Dirac delta on the phase space by $\op_{x,\xi} := \ket{\varphi_{x,\xi}}\!\!\bra{\varphi_{x,\xi}}$. Now for any $g : \Rdd \to \R$ such that $g \in W^{1,\infty}(1+\n{\xi}^n) \cap W^{1,2}(1+\n{\xi}^n) \cap L^2$, one can define the averaging of coherent states, also called Toeplitz operator (see e.g.~\cite{golse_mean_2016, golse_schrodinger_2017}) or Wick quantization (see e.g.~\cite{lerner_fefferman-phong_2007}), as the operator
		\begin{equation*}
			\tildop_g := \iintd g(x,\xi)\,\op_{x,\xi} \d x\d \xi.
		\end{equation*}
		This defines a positive compact operator such that
		\begin{align*}
			\Nrm{\tildop_g}{\L^\infty} &\le \Nrm{g}{L^\infty(\Rdd)},
			&
			\Nrm{\tildop_g}{\L^2} &\le \Nrm{g}{L^2(\Rdd)},
		\end{align*}
		and more generally, as proved for example in~\cite{lions_sur_1993}, such that for any convex function $\Phi$ such that $\Phi(0) = 0$, it holds
		\begin{equation*}
			h^3 \Tr{\Phi\!\(\tildop_g\)} \leq \iintd \Phi(g) \d x\d\xi.
		\end{equation*}
		In particular, in Theorem~\ref{thm:regularity}, we can take $\op^\init = \tildop_g$ with $\Nrm{g}{L^2(\Rdd)} = 1$ and $\Nrm{g}{L^\infty(\Rdd)} = \cC_\infty^{1/2}$, and then $\op^\init$ satisfies the assumptions~\eqref{eq:scaling_op}.
		
		However, we can consider more general operators than simply the averaging of coherent states. Given a function $g$ on the phase space, one can perform the inverse of the Wigner transform, called the Weyl quantization, to define the operator $\op_g$ as the operator with integral kernel
		\begin{equation}
			\op_g(x,y) = \intd e^{-2i\pi \(y-x\)\cdot\xi}\, g(\tfrac{x+y}{2},h\xi) \d \xi.
		\end{equation}
		This operator satisfies the hypotheses of the initial condition of Theorem~\ref{thm:regularity} if $g$ is sufficiently smooth and decays at infinity, as proved for example in \cite[Section~3]{lafleche_strong_2021}.
	\end{remark}
	
\subsection{Mean-Field and Semiclassical Limits}
	
	To state our mean-field results, we assume there exists a constant $C>0$ independent of $N$ and $\hbar$ such that
	\begin{equation}\label{eq:scaling}
		 N^{-\frac{1}{2}} \ll h \leq C N^{-\frac{1}{3}},
	\end{equation}
	where $a\ll b$ means that $\frac{a}{b}\to 0$ as $N\rightarrow \infty$. We also assume that the constant $\cC_\infty$ defined in Formula~\eqref{eq:scaling_op} is independent of $N$ and $\hbar$, and satisfies
	\begin{equation}\label{eq:scaling_op_strict}
		\cC_\infty < (N\,h^3)^{-1}.
	\end{equation}
	We define the following trace class norm over the Fock space weighted by the number operator
	\begin{equation}
		\Nrm{\opN}{\L^1_k(\cF)} := \Nrm{\(\cN+N\)^k\opN}{\L^1(\cF)}.
	\end{equation}
	In what follows, for technical reasons related to the well-posedness of the auxiliary dynamics given in Appendix~\ref{sec:auxiliary_dynamics_appendix}, we will assume that the initial quantum spatial distribution of particles~\eqref{eq:spatial_density_quantum} satisfies
	\begin{equation*}
		\intd \rho^\init(x) \(1+\n{x}\)^3\d x \leq C, 
	\end{equation*}
	where $C$ may depend on $h$.
	
	\begin{theorem}[Mean-field limit]\label{thm:mean_field}
		Let $a\in(0,\frac{1}{2})$ and assume conditions~\eqref{eq:scaling} and~\eqref{eq:scaling_op_strict} are satisfied. Let $n\in 2\N$ satisfying $n \geq 6$. Let $\op$ be a solution to the Hartree--Fock equation~\eqref{eq:Hartree--Fock} with initial condition $\op^\init\in \L^\infty(m_n)$ satisfying \eqref{eq:scaling_op} and such that
		\begin{subequations}
			\begin{align}\label{eq:regularity_op}
				\op^\init &\in \cW^{2,2}(m_n)\cap \cW^{2,4}(m_{n-2})
				\\\label{eq:regularity_sqrt}
				\sqrt{\op^\init} &\in \cW^{1,2}(m_n)\cap \cW^{1,q}(m_{n-2})
			\end{align}
		\end{subequations}
		with $q\in [\tfrac{6}{1-2a},\infty]$. Then, there exist $T>0$, $\opNop^\init \in\L^1(\cF)$, $\lambda>0$ and $C > 0$ such that for any $\opN$ solution of  the second quantized version of~\eqref{eq:Liouville} (see Equation~\eqref{eq:density_matrix_Cauchy_prblm_on_Fock} below) with initial condition $\opN^\init\in\L^1(\cF)$ commuting with $\cN$, for any $t\in[0,T]$ and $p \in [1, \infty]$,
		\begin{equation*}
			\Nrm{\opNr - \op}{\L^p} \leq \frac{C\,e^{\lambda\,t}}{\min(N^{1/2},N\,h^{3/p'})} \(1+ \Nrm{\opN^\init - \opNop^\init}{\L^1_k(\cF)}\)
		\end{equation*}
		for any $k \geq \frac{1}{2\,p} + \frac{3}{2}\lceil\frac{\ln N}{p\ln\(N h^2\)}\rceil$ where $p'=p/(p-1)$ is the H\"older conjugate of $p$.
	\end{theorem}
	
	\begin{remark}
		The $N$-body operator $\opNop^\init$ is explicitly created from $\op^\init$ via the Bogoliubov transformation (see Equation \eqref{eq:reference_state} in Section~\ref{sec:Bogoliubov}). In particular, one should note that $\opNop^\init$ is constructed so that its one-particle reduced density matrix coincides with the initial data $\op^{\init}$ of the Hartree--Fock equation.
	\end{remark}
	
	\begin{remark}
		In the particular case when $h = N^{-\frac{1}{3}}$, $\Nrm{\opN^\init-\opNop^\init}{\L^1(\cF)} \leq C N^{-4}$, and $\Nrm{\cN^4\(\opN-\opNop\)}{\L^1(\cF)} \leq C$, then for any $t\in[0,T]$, one obtains
		\begin{equation*}
			\Nrm{f_{N:1} - f_{\op}}{L^2} = \Nrm{\opNr - \op}{\L^2} \leq \frac{C_T}{N^{1/2}}.
		\end{equation*}
		where $f_{N:1}$ denotes the Wigner transform of $\opNr$.
	\end{remark}
	
	One can combine the above theorem with the result proved in \cite{lafleche_strong_2021} by two of the authors to obtain an estimate directly between the solution of the $N$-body Schrödinger equation~\eqref{eq:Schrodinger} and the Vlasov equation~\eqref{eq:Vlasov}. To simplify, we restrict our attention to the case when $p\leq 2$.
	
	\begin{theorem}[Combined mean-field and semiclassical limits]
		Assume the conditions of Theorem~\ref{thm:mean_field} and that $f$ is a positive solution of the Vlasov equation~\eqref{eq:Vlasov} with initial condition satisfying
		\begin{equation*}
			\(1+\n{x}^8 +\n{\xi}^8\) \nabla_x^{\ell_0}\nabla_{\xi}^\ell f^\init \in L^\infty(\Rdd)\cap L^2(\Rdd) \ \text{ where } \ \ell_0+\ell \le 9.
		\end{equation*}
        Moreover, assume $\opN^\init \in \cL^1(\cF)$ is such that $[\mathcal{N},\opN^\init]=0$. Then, for any $p\in[1,2]$, there exists $T>0, C_T>0$ and an operator $\opNf^\init\in\L^1(\cF)$ such that for any solution $\opN$ to the second quantized version of~\eqref{eq:Liouville} (see Equation~\eqref{eq:density_matrix_Cauchy_prblm_on_Fock} below) with initial condition $\opN^\init$, the estimate 
		\begin{equation*}
			\Nrm{\opNr - \op_f}{\L^p} \leq C_T\(\frac{1}{N\,h^{3/p'}}+h\) \(1 + \Nrm{\opN^\init - \opNf^\init}{\L^1_k(\cF)}\),
		\end{equation*}
        holds for any $t\in[0,T]$ and any $k \geq \frac{1}{2\,p} + \frac{3}{2}\lceil\frac{\ln N}{p\ln\(N h^2\)}\rceil$.
	\end{theorem}
	
	\begin{remark}
		In particular, when $\Nrm{\opN^\init - \opNf}{\L^1_k(\cF)} \leq C$ and $p=2$, then, by Identity~\eqref{eq:L2_norm}, we obtain again a $L^2$ convergence result with the quantitative bound
		\begin{equation*}
			\Nrm{f_{N:1} - f}{L^2(\Rdd)} \leq C_T\(\frac{1}{N\,h^{3/2}}+h\),
		\end{equation*}
		where $f_{N:1}$ is the Wigner transform of $\opNr$.
		
		In our result, it is interesting to notice that the semiclassical error $h$ is larger than the mean-field error when $N \gg h^{-5/2}$, and smaller when $N \ll h^{-5/2}$. When the two are of the same order, one obtains an error of order $h = N^{-2/5}$, which is the best of the rates in term of the number of particles, while the rate is of order $h=N^{-1/3}$ in the critical scaling. However, we do not claim our results yield the optimal rates. 
	\end{remark}
		
	\begin{remark}\label{remark:time-validity}
		In the particular case of the Coulomb potential, we can still obtain an estimate for small times or with a $N$ dependent cutoff (see Theorem~\ref{thm:mean_field_2} and Remark~\ref{remark:cutoff}). Our results are summarized in the following table.
		\vspace{5pt}
		\begin{center}
			\textbf{Time of Validity}\\ \vspace{1em}
			\begin{tabular}{|c|c|c|}
				\hline
				 & $a\in(0,1/2)$ & $a\in [1/2,1]$ \\
				\hline
				Semiclassical Regularity & $t<T$ & $t<T$ \\
				\hline
				Mean-field & $t<T$ & $t\ll h^{a-1/2}$ or cutoff \\
				\hline
				Mean-field + semiclassical & $t<T$ & $t\ll h^{a-1/2}$ or cutoff \\
				\hline
			\end{tabular} 
		\end{center}
	\end{remark}

\section{The Strategy and the General Result}

	\subsection{Second Quantization}
	The method of second quantization provides a mathematical framework for studying the notion of quantum fluctuations. The goal of this section is to recast the original Cauchy problem \eqref{eq:Liouville} with a mixed state initial data on $\h^{\wedge N}$ to a problem on the Fock space $\cF$. We briefly present the method of second quantization and state the corresponding Hamiltonian evolution problem on $\cF$. We refer the interested reader to \cite{berezin_method_1966, reed_functional_1980, folland_harmonic_1989, derezinski_mathematics_2013, baez_introduction_2014} for a more complete presentation.
	
	For every $f \in \h$, we define the associated creation operator $a^\ast(f)$ and its adjoint the annihilation operator $a(f)$ on $\cF$ by their actions on the $n$-sector of $\cF$ as follows
	\begin{align*}
		(a^*(f)\,\psi)^{(n)}(\underline{x}_n) &:= \frac{1}{\sqrt{n}}\sum^n_{j=1}(-1)^{j-1}f(x_j)\, \psi^{(n-1)}(\underline{x}_{n\backslash\, j})
		\\
		(a(f)\,\psi)^{(n)}(\underline{x}_n) &:= \sqrt{n+1}\intd \conj{f(x)} \, \psi^{(n+1)}(x, \underline{x}_n) \d x,
	\end{align*}
	where $ \underline{x}_{n\backslash\, j}:= (x_1, \ldots, \cancel{x}_j, \ldots, x_n)$. Moreover, the action of the annihilation operator on the vacuum of $\cF$ is defined to be $a(f)\,\Omega_{\cF} = 0$. Then, we extend the operators linearly to the whole $\cF$. It can easily be checked that the collection of creation and annihilation operators on $\cF$ satisfies the canonical anti-commutation relations (CAR)
	\begin{equation}\label{CAR}
		\com{a(f), a^*(g)}_+ = \inprod{f}{g}_\h, \quad \com{a(f), a(g)}_+ = \com{a^*(f), a^*(g)}_+ = 0
	\end{equation}
	for all $f, g \in \h$ where $\com{A, B}_+ = AB + BA$ is the anti-commutator of the operators $A, B$. Moreover, from relation \eqref{CAR}, we have the identity
	\begin{equation}
		\Nrm{a(f)\psi}{\cF}^2+\Nrm{a^*(f)\psi}{\cF}^2 = \Nrm{f}{\h}^2\Nrm{\psi}{\cF}^2 \ \ \implies \ \ \Nrm{a^\sharp(f)}{\infty}= \Nrm{f}{\h}
	\end{equation} 
	for all $f \in \h$ where $a^\sharp$ is either $a^*$ or $a$. Thus, both the creation and annihilation operators are bounded operators on $\cF$. 

	At times, it is more convenient to deal with creation and annihilation operators at a given position, say $x$, as opposed to $a^*(f)$ and $a(f)$. Thus, it is useful to introduce, at least formally, the fermionic creation and annihilation operator-valued distributions at the position $x$, denoted respectively by $a^*_x$ and $a_x$, as follows
	\begin{subequations}
		\begin{align}
			(a_x^* \psi)^{(n)}(\underline{x}_n) &= \frac{1}{\sqrt{n}}\,\sum_{j=1}^n \(-1\)^{j-1}\delta(x-x_j)\,\psi^{(n-1)}(\underline{x}_{n\backslash j}),\\
			(a_x \psi)^{(n)}(\underline{x}_n) &= \sqrt{n+1}\, \psi^{(n+1)}(x,\underline{x}_n).
		\end{align}
	\end{subequations}
	It is also straightforward to check that $a^*_x$ and $a_x$ satisfy the anti-commutation relations
	\begin{equation}
		\com{a_x,a_y^*}_+ = \delta(x-y), \qquad \com{a_x,a_y}_+ = \com{a_x^*,a_y^*}_+ = 0,
	\end{equation}
	and that the creation and annihilation operators can be rewritten as follows
	\begin{equation}\label{eq:op-val-distr}
		a^*(f) = \intd f(x)\, a_x^* \d x, \qquad a(f) = \intd \conj{f(x)}\,a_x\d x.
	\end{equation}

	To every observable $O$ on $\h$ corresponds an induced linear operator $\dG(O):\cF\rightarrow \cF$ called the {\it second quantization} of $O$ on $\cF$, defined as 
	\begin{equation}
		\DG{O} = 0\oplus\bigoplus_{n=1}^\infty \dG_n(O) \qquad 
	\end{equation}
	where $\dG_n(O)$ is the $n$-particle operator
	\begin{equation}
		\dG_n(O) = \sum_{j=1}^n 1_{\h^{j-1}}\otimes O \otimes 1_{\h^{n-j}}. 
	\end{equation}
	An important example of second quantized operator is the number operator which is simply the second quantization of the identity operator. Another relevant class of operators are the trace class operators. It is straightforward to check that the second quantization of trace class operators on $\h$ are also trace class operators on $\cF$.
	
	If the observable $O$ has the distributional kernel $O(x, y)$, then we could rewrite $\dG(O)$ in terms of the operator-valued distributions $a^*_x$ and $a_x$ as follows
	\begin{equation}
		\dG(O) = \intdd O(x, y)\,a_x^*\, a_y \d x\d y.
	\end{equation}
	In particular, the number operator can be rewritten as
	\begin{equation}
		\cN = \intd a_x^*\, a_x \d x.
	\end{equation}

	\subsection{State Purification and Time Evolution} 
	
	We define the Fock space Hamiltonian by
	\begin{equation}\label{eq:Fock_Hamiltonian}
		\sfH_N = \intd a_x^* \(-\tfrac{\hbar^2}{2}\lapl_x\) a_x \d x+ \frac{1}{2N}\intdd K(x-y)\,a^*_x \,a^*_y \,a_y\, a_x\d x\d y.
	\end{equation}
	By direct computation, we see that $\sfH_N$ commutes with the number operator, which implies that the expectation of the number of particles is conserved under the Hamiltonian dynamics. 
	Moreover, its action on the $n$-sector is given for any $\psi\in\cF$ by 
	\begin{equation}
		(\sfH_N\psi)^{(n)}= \sfH_N^{(n)}\psi^{(n)} =\sum^n_{k=1}-\frac{\hbar^2}{2}\lapl_{x_k}\psi^{(n)}+\frac{1}{N}\sum^n_{k<l} K(x_l-x_k)\,\psi^{(n)},
	\end{equation}
	which, on the $N$-sector of $\cF$, coincides with the mean-field Hamiltonian defined in Equation~\eqref{eq:Schrodinger_rescaled}. We consider the Cauchy problem 
	\begin{equation}\label{eq:density_matrix_Cauchy_prblm_on_Fock}
		i\hbar\,\dpt \opN = \com{\sfH_N, \opN}, \ \text{ with }\ \opN(t=0) = \op_N^\init = \sum_\jj \lambda_\jj \ket{\psi_\jj}\!\!\bra{\psi_j}
	\end{equation}
	where the data are defined as in \eqref{eq:diagonalisation_op_N}. Following the idea of \cite{benedikter_mean-field_2016}, we reformulate Equation~\eqref{eq:density_matrix_Cauchy_prblm_on_Fock} as an evolution problem of a pure state\footnote{Here, we make the identification of $\ket{\Psi}\!\!\bra{\Psi}$ with $\Psi \in \cG$. In other words, pure state density matrices are simply vectors.} in the fermionic Fock space 
	\begin{equation}
		\cG := \cF(\h\oplus\h)
	\end{equation} 
	which hereinafter will be referred to as the double Fock space. This procedure is commonly known as purification of mixed states. For completeness, we devote the remainder of this section to review the state purification process. 

	For any operator $\opN$ as defined in  Equation \eqref{eq:diagonalisation_op_N} and any orthonormal basis $\phi_{\jj}$ of $\cF$, we construct the following Hilbert--Schmidt operator on $\cF$
	\begin{equation}
		\opupN := \sum_{\jj\in\N} \eps_\jj \ket{\psi_{\jj}}\!\!\bra{\phi_{\jj}},
	\end{equation}
	where $\n{\eps_\jj}^2 = \lambda_\jj$. Then, we see that $\opN = \n{\opupN}^2$, which is called the Schmidt decomposition of $\opN$. In particular, its scaled Hilbert--Schmidt norm, defined by $\Nrm{\opup}{\L^2(\cF)}^2 = h^3 \tr(\n{\opup}^2)$, is 
	\begin{equation}
		\Nrm{\opupN}{\L^2(\cF)}^2 = \Nrm{\opN}{\L^1(\cF)} = 1.
	\end{equation}
	It is important to observe that the decomposition is not unique. In fact, we will need to make a definite choice later. 

	Recall that the space of Hilbert--Schmidt operators $\L^2(\cF)$ is isomorphic to the tensor product $\cF(\h)\otimes\cF(\h)$, as Hilbert spaces, via the linear mapping $\sfJ_h=\sfJ$ that maps $\ket{\psi}\!\!\bra{\phi}\mapsto h^{-\frac{3}{2}}\,\overline\phi\otimes\psi$. One can then associate to $\opupN$ an element of $\cF(\h)\otimes\cF(\h)$ as follows
	\begin{equation}\label{eq:def_tensor_prod_vector}
		\sfJ\opupN = h^{-\frac{3}{2}}\sum_{\jj\in\N} \eps_\jj \, \conj{\phi_{\jj}}\otimes\psi_{\jj} .
	\end{equation}
	 Furthermore, we can associate to every element \eqref{eq:def_tensor_prod_vector} a vector in the double Fock space $\cG$ via the isomorphism $\sfU : \cF\otimes\cF \to \cG$ defined by the following: for $F\in \h^{\wedge n}$ and $G \in \h^{\wedge m}$
	\begin{equation}
		\sfU\(F\otimes G\) = \sqrt{\frac{\(n+m\)!}{n!\,m!}} \(J_l^{\otimes n}F\)\otimes_a \(J_r^{\otimes m}G\),
	\end{equation}
	where $J_l, J_r:\h\rightarrow \h\oplus\h$ are respectively the canonical embedding of $\h$ into the \textit{left} and the \textit{right} coordinate of $\h\oplus\h$, and $\otimes_a$ is the anti-symmetric tensor product. Then extend the mapping linearly  to the entire $\cF\otimes\cF$. The unitary map $\sfU$ is known as the exponential law for Fock spaces and it satisfies the following properties (see \cite[Theorem 3.43]{derezinski_mathematics_2013} or \cite[Chapter 3]{baez_introduction_2014})
	\begin{subequations}
		\begin{align}
			\Omega_\cG &= \sfU(\Omega_{\cF} \otimes \Omega_{\cF}),
			\\
			a_l^\sharp(f) := a^\sharp(f\oplus 0) &= \sfU \(a^\sharp(f)\otimes 1\) \sfU^*,
			\\
			a_r^\sharp(f) := a^\sharp(0\oplus f) &= \sfU \((-1)^\cN \otimes a^\sharp(f)\) \sfU^*,
		\end{align}
	\end{subequations}
	with $a^\sharp$ is either $a$ or $a^*$ and $f \in \h$. The presence of the operator $(-1)^\cN$ ensures that the operators satisfy the CAR. It can also be readily checked that $a^\sharp_l(f)$ anti-commutes with $a^\sharp_r(g)$ for all $f, g \in \h$. 

	Just like in the case of $\cF$, it is useful to define the left and right creation and annihilation operator-valued distributions at the position $x$ by 
	\begin{align*}
		(a_{x, l}^*\Psi)^{(n, m)}(\underline x_n, \underline y_m) &:= \frac{1}{\sqrt{n}}\,\sum_{j=1}^n \(-1\)^{j-1}\delta(x-x_j)\,\Psi^{(n-1, m)}(\underline{x}_{n\backslash j}, \underline y_m),
		\\
		(a_{x, l}\Psi)^{(n, m)}(\underline x_n, \underline y_m) &:= \sqrt{n+1}\,\Psi^{(n+1, m)}(x, \underline x_n, \underline y_m),
		\\
		(a_{x, r}^*\Psi)^{(n, m)}(\underline x_n, \underline y_m) &:= \frac{1}{\sqrt{m}}\,\sum_{j=1}^m \(-1\)^{n+j-1}\delta(x-y_j)\,\Psi^{(n, m-1)}(\underline x_n, \underline{y}_{m\backslash j}),
		\\
		(a_{x, r}\Psi)^{(n, m)}(\underline x_n, \underline y_m) &:= \(-1\)^n\sqrt{m+1}\,\Psi^{(n, m+1)}(\underline x_n, x, \underline y_m).
	\end{align*}
	This allows us to express $a^\sharp_\sigma (f)$ for $f \in \h$ in terms of operator-valued distributions, that is, 
	\begin{equation}
		a_\sigma(f) = \intd \overline{f(x)} \, a_{x, \sigma} \d x \ \ \text{ and } \ \ a_\sigma^*(f) = \intd f(x)\, a^*_{x, \sigma} \d x
	\end{equation}
	where $\sigma \in \{l, r\}$. It is again straightforward to check the CAR relations: $\com{a_{x, \sigma}, a^*_{y, \sigma}}_+ = \delta(x-y)$ and $[a^\sharp_{x, \sigma}, a^\sharp_{y, \sigma'}]_+ = 0$ where $\sigma, \sigma' \in \{l, r\}$. 

	For every observable $O$ on $\h$, we can define the left or right induced linear operator $\DGl{O}, \DGr{O} : \cG\rightarrow \cG$ 
	\begin{align*}
		\DGl{O} &:= \DG{O\oplus 0} = \sfU\(\dG(O)\otimes 1\)\sfU^* =\intdd O(x, y)\,a^*_{x, l}\, a_{y, l} \d x\d y,
		\\
		\DGr{O} &:= \DG{0\oplus O} = \sfU\(1\otimes\dG(O)\)\sfU^* =\intdd O(x, y)\,a^*_{x, r}\, a_{y, r} \d x\d y.
	\end{align*}
	The number operator on $\cG$ is defined by 
	\begin{equation}
		\cN = \cN_l + \cN_r = \sfU \(\cN\otimes 1 + 1\otimes\cN\) \sfU^*.
	\end{equation}
	We shall denote by
	\begin{equation}\label{eq:def_I_G}
		\sfI_\cG := \sfU\sfJ
	\end{equation}
	the transformation from $\L^2(\cF)$ to $\cG$ mapping density operators to vectors of the double Fock space. Then for an operator $\opupN \in \L^2(\cF)$, the action of the operator $\cN$ in $\cG$ becomes $\cN \,\sfI_\cG(\opupN) = \sfI_\cG(\cN\opupN + \opupN\cN)$.
	
	With the above purification process, we can recast our Cauchy problem for mixed states to a Cauchy problem for pure states defined on the double Fock space~$\cG$. Recall the solution to the Cauchy problem~\eqref{eq:density_matrix_Cauchy_prblm_on_Fock} in the Schr\"odinger picture is given by 
	\begin{equation}
		\opN = e^{-i(t/\hbar)\, \sfH_N} \opN^\init \, e^{i(t/\hbar)\, \sfH_N}.
	\end{equation}
	We define the time evolution of $\opupN$ with initial state $\opupN^\init$ by
	\begin{equation}
		\opupN = e^{-i(t/\hbar)\, \sfH_N} \opupN^\init \, e^{i(t/\hbar)\, \sfH_N}.
	\end{equation}
	Then $\opN = \n{\opupN}^2$ solves Equation~\eqref{eq:Liouville} with initial data $\opN^\init = \n{\opupN^\init}^2$. In the double Fock space $\cG$, this corresponds to say that the evolution is given by $\Phi = \Phi(t)$ with
	\begin{equation}
		\Phi := \sfI_\cG(\opupN) = e^{-i(t/\hbar)\, \sfL_N} \sfI_\cG(\opupN^\init) = e^{-i(t/\hbar)\, \sfL_N}\Phi^\init
	\end{equation}
	where the Liouvillian $\sfL_N$ is defined by $\sfL_N = \sfU\(\sfH_N\otimes 1 - 1 \otimes \sfH_N\)\sfU^*$. In particular, for any observable $\mathsf{O}$ of $\cF$, we have the relation
	\begin{equation}\label{eq:mixed-state-pure-state-trace}
		\tr_{\cF}\(\mathsf{O}\opN\) = \Inprod{\Phi}{\(\mathsf{O}\otimes 1\)\Phi}_{\cG} = \tr_{\cG}\((\mathsf{O}\otimes 1)\ket{\Phi}\!\bra{\Phi}\),
	\end{equation}
	which allows us to compute the mean value of the observable $\mathsf{O}$ with respect to the mixed state $\op_N$ in terms of the purified state $\Phi$. In particular, we could express the one-particle reduced density matrix of $\op_N$ in terms of $\Phi$, that is, the integral kernel of $\op_{N:1}$ is given by 
	\begin{equation}
		\op_{N:1}(x, y)=\frac{1}{N\,h^3}\Inprod{\Phi}{a_{x, l}^*\, a_{y, l}\,\Phi}. 
	\end{equation}
	Notice that we are using the normalization $\Tr{\opNr} = h^{-3}$.

\subsection{Bogoliubov Transformation and Quasi-Free States}\label{sec:Bogoliubov}
	In general, we do not know if the evolution of the Cauchy problem \eqref{eq:density_matrix_Cauchy_prblm_on_Fock} can be well-approximated by its mean-field dynamics. Therefore, it is natural to restrict our studies to a subclass of initial data. As stated in \cite{benedikter_mean-field_2016}, equilibrium states at finite positive temperature are believed to be well-approximated by mixed quasi-free states. In the particular case of non interacting fermions at positive temperature, equilibrium states are exactly described by mixed quasi-free states (see~\cite{bratteli_operator_1981}). 
	Furthermore, mixed quasi-free states have the important property that they can be represented by the action of a Bogoliubov transformation on the vacuum of the double Fock space $\cG$, which is a key object in our study of the mean-field limit.

	In this section, we give a brief overview of rudimentary facts about Bogoliubov transformation in the framework of the double Fock space $\cG$ and construct a class of quasi-free states exhibiting the structure of pure states in $\cG$, with average number of particles $N$ and pairing density equal to zero. We follow closely the presentation given in~\cite{solovej_many_2014}.
	
	\subsubsection{Bogoliubov Transformation} For the pairs $f=f_1\oplus f_2, g= g_1\oplus g_2 \in \h\oplus \h$, we define the corresponding field operators by
	\begin{align*}
		A(f, g) &:= a(f)+a^*(\overline g)= a_l(f_1)+a_r(f_2)+a^*_l(\conj{g_1})+a^*_r(\conj{g_2})
		\\
		A^*(f, g) &:= (A(f, g))^* = a_l(\conj{g_1})+a_r(\conj{g_2})+a^*_l(f_1)+a^*_r(f_2).
	\end{align*}
	Notice the field operator $A(f, g)$ and its adjoint satisfy the relation 
	\begin{equation}\label{eq:field_op_conj_relation}
		A^*(f, g) = A(C(f, g))
	\end{equation}
	for all $f, g \in \h\oplus \h$ where $C:(\h\oplus\h)\,\oplus\,(\h\oplus\h)\to (\h\oplus\h)\,\oplus\,(\h\oplus\h)$ is the anti-linear map defined by $C(f_1\oplus f_2,g_1\oplus g_2)=(\overline{g_1}\oplus \overline{g_2},\overline{f_1}\oplus \overline{f_2})$. 
	We can also readily check the collection of field operators satisfy the anti-commutation relations:
	\begin{equation}\label{eq:field_op_anticommutation_relation}
		\com{A(f, g), A^*(h, k)}_+ = \Inprod{(f, g)}{(h, k)}_{(\h\oplus \h)\oplus (\h\oplus \h)}, \quad \com{A^\sharp(f, g), A^\sharp(h, k)}_+ = 0
	\end{equation}
	where $A^\sharp = A$ or $A^*$ and $f, g, h, k \in \h\oplus \h$. 

	A linear isomorphism
	$\nu:(\h\oplus\h)\,\oplus\,(\h\oplus\h)\,\rightarrow\,(\h\oplus\h)\,\oplus\,(\h\oplus\h)$ is called a Bogoliubov (canonical) transformation of $(\h\oplus\h)\,\oplus\,(\h\oplus\h)$ provided it preserves the anti-commutation relations \eqref{eq:field_op_anticommutation_relation}, that is, we have that 
	\begin{equation}\label{eq:field_op_anticommutation_relation_preserved}
		\com{A(\nu(f, g)), A^*(\nu(h, k))}_+ = \Inprod{(f, g)}{(h, k)}_{(\h\oplus \h)\oplus (\h\oplus \h)}
	\end{equation}
	for all $f, g, h, k \in \h\oplus \h$ and, likewise, for the other relations. Hence, it follows from~\eqref{eq:field_op_conj_relation} and~\eqref{eq:field_op_anticommutation_relation_preserved} that $\nu$ is a Bogoliubov transformation provided it satisfies the conditions
	\begin{equation}\label{eq:Bogolibov-properties}
	\nu\, C = C\,\nu\,, \qquad \text{ and } \qquad \nu^*\nu= \nu \nu^* = I\,,
	\end{equation}
	where $I$ is the corresponding identity map. 
	
	It is more convenient to express conditions \eqref{eq:Bogolibov-properties} as follows: $\nu$ is a Bogoliubov transformation on $(\h\oplus\h)\,\oplus\,(\h\oplus\h)$ if there exist operators $U$, $V: \h\oplus\h\rightarrow \h\oplus\h$ satisfying the properties
	\begin{equation}\label{eq:Bogoliubov-cond}
		U^*U+V^*V=I\qquad \text{ and }\qquad U^*\overline{V}+V^*\overline{U}=0
	\end{equation}
	such that $\nu$ has the form 
	\begin{equation}\label{eq:Bogoliubov-expres}
		\nu = \(\begin{array}{ll} U & \overline{V} \\ V & \overline{U} \end{array}\).
	\end{equation}
	Moreover, we say that the Bogoliubov transformation $\nu$ is (unitarily) implementable on $\cG$ if there exists a unitary map $\sfR_{\nu}:\cG\mapsto\cG$ such that
	\begin{equation}\label{eq:implementation_of_nu}
		\sfR^*_\nu\,A(f,g)\, \sfR_\nu = A(\nu(f,g))
	\end{equation}
	for all $f, g \in \h\oplus \h$. A necessary and sufficient condition for the transformation $\nu$ to be implementable is given by Shale and Stinespring in \cite{shale_spinor_1965}: $\nu$ is implementable if and only if $V$ is a Hilbert--Schmidt operator. In particular, if $\Tr{V^*V}$ is finite, then $\nu$ is an implementable Bogoliubov transformation. It is common to refer to $\sfR_\nu$ as the Bogoliubov transformation on $\cG$.

	\subsubsection{Quasi-Free States} A fermionic state $\opN$ on $\cF$ is said to be quasi-free provided it possesses the following factorization properties 
	\begin{subequations}\label{eq:def_quasifree_conditions}
		\begin{align}
			&\tr_{\cF}\(a^{\sharp_1}(f_1)\cdots a^{\sharp_{2n+1}}\!\(f_{2n+1}\) \opN\) =0,
			\\
			&\tr_{\cF}\(a^{\sharp_1}(f_1)\cdots a^{\sharp_{2n}}\!\(f_{2n}\) \opN\)
			\\\nonumber
			&\qquad= \sum_{\sigma} (-1)^\sigma \prod_{j=1}^n\tr_{\cF}\(a^{\sharp_{\sigma(2j-1)}}(f_{\sigma(2j-1)})a^{\sharp_{\sigma(2j)}}(f_{\sigma(2j)}) \opN\), 
		\end{align}
	\end{subequations}
	where $f_k \in \h$ and the sum is over all permutations $\sigma$ of $\set{1,\dots,2n}$ satisfying  
	\begin{equation*}
		\forall j\in\set{1,\dots, n}, \sigma(2j-1) < \sigma(2j), \text{ and }
		\sigma(2j-1) < \sigma(2j+1) \text{ if } j<n.
	\end{equation*}
	In short, a state is said to be quasi-free if the higher-order reduced density matrices of $\opN$ are completely determined by the generalized one-particle reduced density matrix. We could also express Conditions \eqref{eq:def_quasifree_conditions} in terms of the purified state $\Phi$. This means that any quasi-free mixed state can be viewed as the partial trace of a quasi-free pure state. Moreover, using the fact that pure quasi-free states are completely characterized by their generalized one-particle reduced density matrix, it can be shown that a pure quasi-free state $\Phi$ on $\cG$ can be written as $\Phi = \sfR_{\nu}\Omega$ for some Bogoliubov transformation $\sfR_{\nu}$.

	Let us now construct the Bogoliubov transformation and its corresponding class of quasi-free states that we will study in Part~\ref{part:mean-field} of the paper. Let $\omega$ be a one-particle density operator on $\h$ satisfying the properties: $0\leq \omega \leq 1$ and $\Tr{\omega}=N$. Define the map
	$\nu : (\h\oplus\h)\oplus(\h\oplus\h)\mapsto(\h\oplus\h)\oplus(\h\oplus\h)$
	given by Equation \eqref{eq:Bogoliubov-expres} with $U$ and $V$ having the explicit forms
	\begin{equation}
		U = \(\begin{array}{ll}
		u & 0 \\ 0 & \overline{u}
		\end{array}\)
		\quad\quad\mbox{and}\quad\quad
		V =\(\begin{array}{ll}
		 0 & \overline{v} \\ -v & 0
		\end{array}\)
	\end{equation}
	with 
	\begin{equation}
		u := \sqrt{1-\omega} \quad\quad\mbox{and}\quad\quad v := \sqrt{\omega}\,.
	\end{equation}
	Notice $U$ and $V$ satisfy Equation~\eqref{eq:Bogoliubov-cond} which means $\nu$ is a Bogoliubov transformation. Furthermore, $V$ is a Hilbert--Schmidt operator. Indeed, since $\Tr{V^*V}=2\Tr{\omega} = 2\,N$ is clearly finite, then, by the Shale--Stinespring condition \cite{shale_spinor_1965}, $\nu$ is implementable.
	Hence, there exists a unitary map $\sfR_{\nu}:\cG\mapsto\cG$ implementing $\nu$. Consequentially, Equation \eqref{eq:implementation_of_nu} yields the relations
	\begin{align*}
		\sfR^*_{\nu}a_{x, l}\sfR_{\nu} =&\ a_l(u_x)-a^*_r(\conj{v_x}),\\
		\sfR^*_{\nu}a_{x, r}\sfR_{\nu} =&\ a_r(\conj{u_x})+a^*_l(v_x).
	\end{align*}
	where we used the notation $u_x(y)= u(y, x), v_x(y)= v(y, x)$. 
	
	Let us now use the Bogoliubov transformation to represent quasi-free mixed states. The construction we present here is an example of the well-known Araki--Wyss representation, see \cite{araki_representations_1964, araki_quasifree_1971, derezinski_mathematics_2013}. More precisely, we are interested in constructing a quasi-free mixed state with one-particle reduced density $\op$ on the double Fock space $\cG$. To this end, we define $\sfR_{\op}$ as the Bogoliubov transform with
	\begin{equation*}
		\omega = N\,h^3 \op
	\end{equation*}
	and let the unitary map $\sfR_{\op}$ act on the vacuum $\Omega_\cG$, i.e.
	\begin{equation}\label{eq:quasi-free-mixed}
		\Phi_{\op} := \sfR_{\op}\,\Omega_{\cG}\in\cG\,.
	\end{equation}
	We can now compute the integral kernel of the one-particle reduced density matrix associated with 
	the state $\Phi_{\op}$:
	\begin{equation*}
		\begin{split}
		\op_{N:1}(x,y) &= \frac{1}{N\,h^3} \Inprod{\Phi_{\op}}{a_{l,y}^* \, a_{l,x} \Phi_{\op}} = \frac{1}{N\,h^3} \Inprod{\Omega_\cG}{\sfR^*_{\op}a_{l,y}^*\sfR_{\op}\sfR^*_{\op}a_{l,x}\sfR_{\op}\Omega_{\cG}}
		\\
		&= \frac{1}{N\,h^3} \Inprod{\Omega_\cG}{a_{l}(\overline{v_y})\,a^*_r(\overline{v_x})\,\Omega_{\cG}} =\frac{1}{N\,h^3} \, (v^*v)(x,y)=\op(x,y).
		\end{split}
	\end{equation*}
	Therefore, the one-particle reduced density matrix associated with $\Phi_{\op}$ corresponds to the operator $\op$. Furthermore, the off-diagonal term associated with the state $\Phi_{\op}$, referred to as pairing density, is zero. Indeed, since
	\begin{equation*}
		\begin{split}
		\alpha_{\Phi_{\op}}(x,y)&:= \Inprod{\sfR_{\op}\Omega_\cG}{a_{l,y}\,a_{l,x}\, \sfR_{\op}\Omega_{\cG}} = \Inprod{\Omega_\cG}{a_{l}(u_y)\,a_l(\conj{v_x})\,\Omega_{\cG}} = 0
		\end{split}
	\end{equation*}
	where we used that $[a_l(u_y),a_r^*(\overline{v_x})]_+=0$. Undoing the purification process, we can now define the reference state (mean-field approximation) $\opNop$, associated to the solution $\op$ of the Hartree--Fock equation~\eqref{eq:Hartree--Fock}, as stated in Theorem~\ref{thm:mean_field} by
	\begin{equation}\label{eq:reference_state}
		\opNop = \n{\sfI_\cG^{-1}(\Phi_{\op})}^2.
	\end{equation}

\subsection{The General Result}
	
	In this section, we state a more general result from which our main results would follow. Our result is obtained by controlling the growth of the weighted norm
	\begin{equation*}
		\Nrm{\Psi}{\cG_k} := \Nrm{\(\cN+1\)^k\Psi}{\cG}.
	\end{equation*}

	\begin{theorem}\label{thm:mean_field_2}
		Let $a\in[0,1]$ and assume Condition~\eqref{eq:scaling} is satisfied. Let $(k,n)\in\N^2$ and $\alpha\in \lt[0,1\rt]$ satisfying $n \geq 6$ and $\alpha > a - \frac{1}{2}$. Let $\op$ be a solution of the Hartree--Fock equation~\eqref{eq:Hartree--Fock} with initial condition $\op^\init\in \L^\infty(m_n)$ satisfying \eqref{eq:scaling_op} and such that
		\begin{align}\label{eq:regularity_op_2}
			\op^\init &\in \cW^{2,2}(m_n)\cap \cW^{2,4}(m_{n-2})
			\\\label{eq:regularity_sqrt_2}
			\sqrt{\op^\init} &\in \cW^{1,2}(m_n)\cap \cW^{1,q}(m_{n-2})
		\end{align}
		with $q\in [2,\infty]$ satisfying
		\begin{equation}\label{eq:scaling_q_2}
			\frac{3}{q} \in \lt[2\(\alpha - a - \frac{1}{4}\), \alpha - a + \frac{1}{2}\rt].
		\end{equation}
		Let $\Psi^\init\in \cG$. Then, there exists $T>0$ and $C>0$ such that for any $t\in[0,T]$ and any $p\in[1,\infty)$
		\begin{equation*}
			\Nrm{\op_{N:1}-\op}{\L^p} \leq \frac{C\,e^{\lambda\,h^{-\alpha}\,t}}{\min(N^\frac{1}{2}, N\,h^\frac{3}{p'})} \(\Nrm{\Psi^\init}{\cG_{\frac{3}{2}k+\frac{1}{2\,p}}}^2 + \frac{h^{2 k \(\alpha-1\)}}{N^{k-\frac{1}{p}}} t^2 \Nrm{\Psi^\init}{\cG_{\frac{3}{2}k}}^2\).
		\end{equation*}
		where $\lambda = C_{a,\alpha}\n{\kappa} C_{\op}$ for some constant $C_{\op}$ depending only on $T$ and the initial condition of the Hartree--Fock equation. 
	\end{theorem}

    In the above theorem, we assumed we know the perturbation of the vacuum, $\Psi^\init$. As done in \eqref{eq:reference_state} for the reference state $\opNop$, we can associate to $\Psi^\init$ an operator $\opN = \n{\sfI_\cG^{-1}(\sfR_{\op}\Psi)}^2$ which solves the Schrödinger equation~\eqref{eq:density_matrix_Cauchy_prblm_on_Fock}.
     
	\begin{remark}
		In particular, notice that
		\begin{equation*}
			\frac{h^{2 k\(\alpha-1\)}}{N^{k-\frac{1}{p}}} \leq 1 \iff k \geq \frac{\ln N}{p\,\ln\(Nh^{2\(1-\alpha\)}\)}.
		\end{equation*}
		More specifically, if $N = h^{-c}$, then this is equivalent to $k\geq \frac{c}{p\(c+2 (\alpha-1)\)}$. For instance, take $c=3$. Then for any $a<1/2$, we can take $\alpha = 0$ and $k = 3$, leading to
		\begin{equation*}
			\Nrm{\op_{N:1}-\op}{\L^p} \leq \frac{C\,e^{\lambda\,t}}{N^{\min(\frac{1}{2},\frac{1}{p})}} \Nrm{\Psi}{\cG_{5}}^2.
		\end{equation*}
		In the case of the Coulomb potential $a=1$, we can take $k = 2$ and any $\alpha > 1/2$, leading to
		\begin{equation*}
			\Nrm{\op_{N:1}-\op}{\L^p} \leq \frac{C}{N^{\min(\frac{1}{2},\frac{1}{p})}} \Nrm{\Psi}{\cG_4}^2 e^{\lambda\,t/h^\alpha},
		\end{equation*}
		which is small only for small times $t\ll N^{-1/6} = h^{1/2}$. This is an improvement in comparison to non semiclassical estimates which are valid only for $t\ll h$.
	\end{remark}
	
	\begin{remark}\label{remark:cutoff}
		When $a\geq 1/2$, one can also consider the potential with a $h$-dependent cut-off. For example, a way to get a potential bounded at distance $\n{x} \leq R$ is to take
		\begin{equation}\label{eq:cutoff}
			K_R(x) = \frac{\omega_a\,\kappa}{2} \int_0^{R^{-2}} s^{\frac{a}{2}-1} e^{-\pi\n{x}^2s}\d s \,\underset{R\to 0}{\longrightarrow}\, \kappa \, \frac{1}{\n{x}^a},
		\end{equation}
		which is a radial decreasing potential satisfying $K_R(x) \leq \n{\kappa}\max\!\(\frac{1}{\n{x}^a}, \frac{\omega_a}{a R^a}\)$. For the Coulomb interaction potential for example, assuming $R\leq 1$ and $N=h^{-c}$ and taking $c=3$ and $p\leq 2$, this leads to
		\begin{equation*}
			\Nrm{\op_{N:1}-\op}{\L^p} \leq \frac{C\,e^{\lambda\,t/\sqrt{R}}}{N^{1/2}} \Nrm{\Psi}{\cG_{ 5}}^2.
		\end{equation*}
		Thus, one obtains a quantitative convergence result as soon as $R > \frac{4 \lambda^2 t^2}{\(\ln N\)^2}$.
	\end{remark}
	
	\begin{remark}
		Let $\opNop$ be defined by Equation~\eqref{eq:reference_state}. Then the standard deviation of the number of particles $\sigma_\cN^2 := h^3\Tr{\cN^2\,\opNop} - \(h^3\Tr{\cN\,\opNop}\)^2$ is given by
		\begin{equation*}
			\sigma_\cN^2 = \Tr{\omega-\omega^2} = N\(1-\cC_2^2 Nh^3\).
		\end{equation*}
		In particular, $\sigma_\cN \leq \sqrt{N}$.
		
		Notice also that $\sigma_\cN = 0 \iff \omega = \omega^2 \iff \cC_2^2Nh^3 = 1$. This implies that in order for the reference state $\opNop$ to have a fixed number of particles, it has to be a pure state and the scaling has to be the critical scaling $Nh^3 = \cC_2^{-2}$. In this case, the regularity conditions \eqref{eq:regularity0_0} are not expected to hold. However, it is a good question to know whether it is possible to find a state $\opN = \n{\sfI_\cG^{-1}(\sfR_{\op}\Psi)}^2$ with a fixed number of particles but still close to $\opNop$, in the sense that the associated $\Psi$ satisfies $\Nrm{\Psi}{\cG_5} \ll N^\frac{1}{2}$.
	\end{remark}

\part{\large Propagation of Regularity}\label{part:regularity}

	This part is devoted to the proof of Theorem~\ref{thm:regularity} about the propagation of the semiclassical regularity of the solutions of the Hartree--Fock equation~\eqref{eq:Hartree--Fock}, and also of higher regularity properties needed to obtain Theorem~\ref{thm:mean_field}.

\section{The Classical Case: Regularity for the Vlasov Equation}\label{sect:regularity-vlasov}
	
	As a warm-up and an explanation of our strategy, we start by the analogue of our method in the classical case of the kinetic Vlasov equation. We define
	\begin{equation*}
		N_{p,x} := \iintd \n{\Dx f}^p m
		\quad \text{ and } \quad N_{p,\xi} := \iintd \n{\Dv f}^p m.
	\end{equation*}
	Denoting $\sfT:= \xi\cdot\Dx + E\cdot\Dv$, then we have
	\begin{equation}\label{eq:time-der-grad}
		\begin{split}
			\partial_t(\Dx f) &= -\sfT \Dx f - \nabla E \cdot \Dv f
			\\
			\partial_t(\Dv f) &= -\sfT\Dv f - \Dx f.
		\end{split}
	\end{equation}
	
	\begin{prop}
		Let $n>3$ and $f$ be a solution of Equation \eqref{eq:Vlasov}  with initial condition satisfying
		\begin{equation*}
			\nabla_{x,\xi}f^\init \in L^p(1 + \n{\xi}^n)
		\end{equation*}
		for any $p\in [1,\infty)$. Then there exists a time $T>0$ such that
		\begin{equation*}
			\nabla_{x,\xi}f \in L^\infty((0,T),L^p(1 + \n{\xi}^n)).
		\end{equation*}
	\end{prop}

	\begin{proof}
		Let $m := 1+\n{\xi}^{np}$. To simplify the computations, we observe that $\sfT^* = -\sfT$ and $\sfT(uv) = u\,\sfT(v) + \sfT(u)\,v$, so that by writing $u^p := \n{u}^{p-1}u$, it holds
		\begin{equation*}
			\iintd \sfT(u)\cdot u^{p-1}m = -\iintd u\cdot\sfT(u^{p-1})\,m + \n{u}^p\sfT(m).
		\end{equation*}
		But
		\begin{align*}
			u\cdot(\sfT(u^{p-1})) &= u^{p-1}\cdot\sfT(u) + (p-2)\(\sfT(u)\cdot u\)\n{u}^{p-2}
			\\
			&= \(p-1\)u^{p-1}\cdot\sfT(u).
		\end{align*}
		We deduce
		\begin{equation*}
			-p\iintd \sfT(u)\cdot u^{p-1}m = \iintd \n{u}^p\sfT(m).
		\end{equation*}
		Therefore, differentiating the weighted $L^p$ norms, we obtain
		\begin{align*}
			\ddt{N_{p,x}} &= \iintd \n{\Dx f}^p \sfT(m) - p\(\Dx f\)^{p-1}\cdot\nabla E\cdot\Dv f \,m\d x\d\xi
			\\
			\ddt{N_{p,\xi}} &= \iintd \n{\Dv f}^p \sfT(m) - p\(\Dv f\)^{p-1}\cdot\Dx f \,m\d x\d\xi.
		\end{align*}
		Then by Young's inequality for the product
		\begin{equation*}
			\sfT(m) = n\,p\, E\cdot \xi^{np-1} \leq np\Nrm{E}{L^\infty} m.
		\end{equation*}
		We cut $\nabla K = F_1+F_2\in L^{\fb_1} + L^{\fb_2}$. The difficult term is
		\begin{multline*}
			\iintd (\Dx f)^{p-1}\cdot\nabla E\cdot\Dv f\,m\d x\d \xi \leq \Nrm{\nabla K * \nabla\rho}{L^{\infty}} \iintd \n{\Dx f}^{p-1}\n{\Dv f}m\d x\d \xi
			\\
			\leq \(\Nrm{F_1}{L^{\fb_1}} \Nrm{\nabla\rho}{L^{{\fb_1'}}}+\Nrm{F_2}{L^{\fb_2}} \Nrm{\nabla\rho}{L^{{\fb_2'}}}\) N_{p,x}^{1-1/p} N_{p,\xi}^{1/p}
			\\
			\leq B_K \(\Nrm{\nabla\rho}{L^{{\fb_1'}}} + \Nrm{\nabla\rho}{L^{{\fb_2'}}}\) N_{p,x}^{1-1/p} N_{p,\xi}^{1/p},
		\end{multline*}
		with $B_K = \Nrm{\nabla K}{L^{\fb_1}+L^{\fb_2}}$ and where we used three times H\"older's inequality. Then, when $n\geq 3/\fb$, we get
		\begin{equation*}
			\Nrm{\nabla\rho}{L^{\fb'}} = \Nrm{\intd\Dx f\d \xi}{L^{\fb'}} \leq \Nrm{\Dx f\, (1+\n{\xi}^{n})}{L^{\fb'}_{x,\xi}} \leq C\,N_{\fb',x}^{1/\fb'}.
		\end{equation*}
		Therefore, taking respectively $p=\fb_1'$ and $p = \fb_2'$ and using the notation $u_k = u_k(t) := N_{\fb_k',x}^\frac{1}{\fb_k'}$ and $v_k = v_k(t) := N_{\fb_k',\xi}^\frac{1}{\fb_k'}$, for $k=1$ and $k=2$, it holds
		\begin{align*}
			\dt u_k &\leq n \Nrm{E}{L^\infty} u_k + B_K\(u_1+u_2\) v_k
			\\
			\dt v_k &\leq n \Nrm{E}{L^\infty} v_k + u_k v_k,
		\end{align*}
		so that defining $U := u_1 + u_2 + v_1 + v_2$ we get
		\begin{equation}\label{eq:U-gronwall}
			\dt U \leq n \Nrm{E}{L^\infty} U + \(B_K+\frac{1}{2}\) U^2,
		\end{equation}
		where we used the fact $2uv\leq u^2+v^2$, and we obtain by Gr\"onwall's inequality that $U$ remains finite as long as $t<T$ where $T$ depends on the growth of $\Nrm{E(t,\cdot)}{L^\infty}$, which we can control by
		\begin{equation*}
			\Nrm{E}{L^\infty} \leq C_K \(\Nrm{\nabla\rho}{L^{c_1'}} +\Nrm{\nabla\rho}{L^{c_2'}}\) \leq C_K \(N_{c_1',x}^{1/c_1'} + N_{c_2',x}^{1/c_2'}\),
		\end{equation*}
		with $C_K = \Nrm{K}{L^{c_1}+L^{c_2}}$ . 
		In particular, for the Coulomb interaction, one can choose $c_1=\fb_1 < 3/2$ and $c_2=\fb_2 > 3$.
	\end{proof}

\section{The Quantum Case: Propagation of Regularity for the Hartree--Fock Equation}

	In this section, we prove the semiclassical analogue of the propagation of regularity for the Vlasov equation shown in Section~\ref{sect:regularity-vlasov}. The main difficulty is to close the Gr\"onwall inequality, which we managed to do by propagating at the same time the $\L^\infty(m_n)$, the $\cW^{1,2}(m_n)$, and the $\cW^{1,q}(m_{n-2})$ norms with $q\geq 2$ and
	\begin{equation*}
		m_n = 1+\n{\opp}^n,
	\end{equation*}
	where $n\in 2\N$. This first step allows us to prove that the $\cW^{1,q}(m_n)$ norm remains bounded on some time interval for $q\in[2,q_a)$ with $q_a := \infty$ if $\fb := \frac{3}{a+1} \geq 2$ and 
	\begin{equation*}
		\frac{1}{q_a} := \frac{1}{\fb} - \frac{1}{2}
	\end{equation*}
	when $\fb<2$. It is the content of the following proposition, where we only consider $q\in[2,4]$ for simplicity.

	\begin{prop}\label{prop:regu_HF}
		Fix $a\in(0,1]$. If ${\op}^{\init}\in \cW^{1,2}(m_n)\cap \cW^{1,4}(m_{n-2}) \cap \L^\infty(m_n)$, then there exists $T>0$ such that
		\begin{align*}
			{\op} &\in L^\infty((0,T),\cW^{1,2}(m_n)\cap \cW^{1,4}(m_{n-2}) \cap \L^\infty(m_n)),
			\\
			{\rho} &\in L^\infty((0,T),H^{1} \cap W^{1,4} \cap L^1 \cap L^\infty).
		\end{align*}
	\end{prop}
	
	Now that we know that the first-order semiclassical Sobolev norms remain bounded for some finite time $T>0$ for $q\in[2,q_a)$, we can use this first result and a similar strategy to prove the propagation of higher Sobolev norms on the same time scale. This is done in the following proposition.
		
	\begin{prop}\label{prop:regu_HF_2}
		Under the hypotheses of Proposition~\ref{prop:regu_HF} and assuming moreover that $\op^\init\in \cW^{2,2}(m_n)\cap \cW^{2,4}(m_{n-2})$, then
		\begin{align*}
			\op &\in L^\infty((0,T),\cW^{2,2}(m_n)\cap \cW^{2,4}(m_{n-2})\cap\cW^{1,\infty}(m_n)),
			\\
			\rho &\in L^\infty((0,T),H^{2} \cap W^{2,4}).
		\end{align*}
	\end{prop}
	
	\begin{remark}
		The propagation of the second-order Sobolev norms will allow us to remove the constraint $q\in[2,q_a)$ and to get the boundedness of first-order Sobolev norms also for $q\geq q_a$. This is relevant when $a\geq \tfrac{1}{2}$.
	\end{remark}
	
	In order to perform the mean-field limit, we actually need to prove the propagation of these norms for $\sqrt{\op}$ instead of $\op$ (Cf. Equations \eqref{eq:def_Dqq}-\eqref{eq:def_Dqq_tilde}), which works in a similar way.
	
	\begin{prop}\label{prop:regu_HF_sqrt}
		Under the hypotheses of Proposition~\ref{prop:regu_HF_2}, if $\sqrt{\op^\init}\in \cW^{1,q}(m_n)$ for some $q\in [2,\infty]$, then
		\begin{equation*}
			\sqrt{\op} \in L^\infty((0,T),\cW^{1,q}(m_n)).
		\end{equation*}
	\end{prop}
	
	This proposition then implies also the regularity of $\op$ as indicated in next lemma.
	\begin{lem}\label{lem:regu_rho_vs_sqrt}
		Let $\op \geq 0$ be a compact operator. Then for any $q\in[1,\infty]$,
		\begin{equation}\label{eq:regu_rho_vs_sqrt}
			\Nrm{\op}{\dot{\cW}^{1,q}(m_n)} \leq 2 \Nrm{\sqrt{\op}}{\L^\infty(m_n)} \Nrm{\sqrt{\op}}{\dot{\cW}^{1,q}(m_n)}.
		\end{equation}
	\end{lem}
	
	\begin{proof}[Proof of Lemma~\ref{lem:regu_rho_vs_sqrt}]
		By the product rule for commutators and H\"older's inequality for Schatten norms, for any $\eta \in \set{x,\xi}$,
		\begin{align*}
			\Nrm{\Dh_{\!\eta} \op}{\L^q} &= \Nrm{\(\Dh_{\!\eta}(\sqrt{\op})\, \sqrt{\op} + \sqrt{\op}\, \Dh_{\!\eta}(\sqrt{\op})\)m_n}{\L^q}
			\\
			&\leq \Nrm{\Dh_{\!\eta}(\sqrt{\op})}{\L^q} \Nrm{\sqrt{\op}\, m_n}{\L^\infty} + \Nrm{\sqrt{\op}}{\L^\infty} \Nrm{\Dh_{\!\eta}\sqrt{\op}\,m_n}{\L^q},
		\end{align*}
		which implies Inequality~\eqref{eq:regu_rho_vs_sqrt}.
	\end{proof}

\subsection{The Strategy}

	Both the Hartree and the Hartree--Fock equations can be written in the form
	\begin{equation*}
		i\hbar\,\dpt\op = \com{H,\op}
	\end{equation*}
	with $H = \frac{\n{\opp}^2}{2} + V_{\op} - h^3\sfX_{\op}$ (with $\sfX_{\op} = 0$ in the case of the Hartree equation). If we look at the time derivatives of quantum gradients, since $\Dhx H = \nabla V_{\op} = -E_{\op}$ and $ \frac{1}{i\hbar}\com{\Dhv H,\op} = \frac{1}{i\hbar} \com{\opp,\op} = - \Dhx\op$ and since $\com{x,\sfX_{\op}} = \sfX_{\com{x,\op}}$ and $\com{\nabla,\sfX_{\op}} = \sfX_{\com{\nabla,\op}}$ (see Lemma~\ref{lem:X_commut_gradients}), we obtain
	\begin{align}\label{eq:time-der-grad-op}
		\partial_t \Dhx \op &= \frac{1}{i\hbar}\com{H,\Dhx\op} - \frac{1}{i\hbar} \com{h^3\sfX_{\Dhx \op},\op} - \frac{1}{i\hbar} \com{E_{\op},\op}
		\\\nonumber
		\partial_t \Dhv \op &= \frac{1}{i\hbar}\com{H,\Dhv\op} - \frac{1}{i\hbar} \com{h^3\sfX_{\Dhv \op},\op} - \Dhx\op.
	\end{align}
	These equations are of the form
	\begin{equation}\label{eq:abstract_equation}
		i\hbar\, \dpt \opmu = \com{\sfA,\opmu} + \com{\sfB,\op},
	\end{equation}
	with $\sfA$ and $\sfB$ self-adjoint. Our goal is to propagate the weighted Schatten norms for solutions of these equations, where we recall that Schatten norms were defined in Equation~\eqref{eq:Schatten}. Computing the time derivative of such quantities, we get the following result.
	\begin{lem}\label{lem:abstract_lemma}
		Let $\op$, $\sfA$ and $\sfB$ be self-adjoint operators and $\opmu = \opmu(t)$ be a family of self-adjoint operators satisfying Equation~\eqref{eq:abstract_equation}. Then, formally, for any even integer $q\geq 2$ we have
		\begin{equation*}
			 \dt \Nrm{\opmu\, m_n}{q} \leq \frac{1}{\hbar}\Nrm{\com{\sfA,m_n}\opmu}{q} + \frac{1}{\hbar} \Nrm{\com{\sfB,\op}m_n}{q}.
		\end{equation*}
	\end{lem}
	
	Applying this lemma for $\opmu = \op$ solution of the Hartree--Fock equation or for $\opmu = \Dhx{\op}$ or $\opmu = \Dhv{\op}$, and with $m_n = 1$ (for $n=0$) or $m_n = \opp_\jj^n$ for some $\jj\in\set{1,2,3}$, we obtain
	\begin{align}\label{eq:rho}
		\dt \Nrm{\op\, m_n}{q} &\leq \frac{1}{\hbar}\Nrm{\com{V_{\op},m_n}\op}{q} + \frac{1}{\hbar} \Nrm{\com{h^3\sfX_{\op},m_n}\op}{q},
		\\
		\dt \Nrm{\Dhx \op\, m_n}{q} &\leq \frac{1}{\hbar}\Nrm{\com{V_{\op},m_n}\Dhx \op}{q} + \frac{1}{\hbar} \Nrm{\com{E_{\op},\op}m_n}{q}\nonumber
		\\\label{eq:Dhx}
		&\qquad + \frac{1}{\hbar}\Nrm{\com{h^3\sfX_{\op},m_n}\Dhx \op}{q} + \frac{1}{\hbar} \Nrm{\com{h^3\sfX_{\Dhx \op},\op}m_n}{q},
		\\
		\dt \Nrm{\Dhv \op\, m_n}{q} &\leq \frac{1}{\hbar}\Nrm{\com{V_{\op},m_n}\Dhv \op}{q} + \Nrm{\Dhx\op\, m_n}{q}\nonumber
		\\\label{eq:Dhv}
		&\qquad + \frac{1}{\hbar}\Nrm{\com{h^3\sfX_{\op},m_n}\Dhv \op}{q} + \frac{1}{\hbar} \Nrm{\com{h^3\sfX_{\Dhv \op},\op}m_n}{q},
	\end{align}
	where we used the fact that $\com{H,m_n} = \com{V_{\op}-h^3\sfX_{\op},m_n}$ since $[\n{\opp}^2,m_n] = 0$. In the next sections,  we will bound all the weighted Schatten norms of the commutators appearing on the right-hand side of inequalities~\eqref{eq:rho}, \eqref{eq:Dhx} and \eqref{eq:Dhv} in order to get a Gr\"onwall-type inequality.

	\begin{proof}[Proof of Lemma~\ref{lem:abstract_lemma}]
		First notice that
		\begin{equation*}
			 \dpt \opmu^2 = \frac{1}{i\hbar}\com{\sfA,\opmu^2} + \frac{1}{i\hbar}\(\com{\sfB,\op}\opmu + \opmu\com{\sfB,\op}\).
		\end{equation*}
		Therefore, using that $2p := q$ is even and the cyclicity of the trace, we get
		\begin{align}\nonumber
			\dt \Nrm{\opmu\, m_n}{2p}^{2p} &= \dt \Tr{\(m_n\opmu^2m_n\)^p} = p\Tr{m_n\(\dt(\opmu^2)\) m_n\(m_n\opmu^2m_n\)^{p-1}}
			\\\label{eq:IA_IB}
			&= \frac{p}{i\hbar}\Tr{m_n \(\com{\sfA,\opmu^2} +\com{\sfB,\op}\opmu + \opmu\com{\sfB,\op}\) m_n\(m_n\opmu^2m_n\)^{p-1}}
			\\\nonumber
			& =: I_\sfA + I_\sfB.
		\end{align}
		For $I_\sfA$, we use again the cyclicity of the trace to write
		\begin{align*}
			I_\sfA &= \frac{p}{i\hbar} \Tr{\sfA\opmu^2m_n^2 (\opmu^2 m_n^2)^{p-2} \opmu^2m_n^2 - \opmu^2 \sfA m_n^2 (\opmu^2 m_n^2)^{p-2} \opmu^2m_n^2 }
			\\
			&= \frac{p}{i\hbar} \Tr{ m_n^2 \sfA\opmu^2m_n^2 (\opmu^2 m_n^2)^{p-2} \opmu^2 - \sfA m_n^2 (\opmu^2 m_n^2)^{p-2} \opmu^2m_n^2\opmu^2 }
			\\
			&= \frac{p}{i\hbar} \Tr{ \com{m_n^2,\sfA}\(\opmu^2 m_n^2\)^{p-1} \opmu^2 }.
		\end{align*}
		This can also be written as
		\begin{align*}
			I_\sfA &= \frac{p}{i\hbar} \Tr{\opmu \(\com{m_n,\sfA}m_n+m_n\com{m_n,\sfA}\)\opmu \n{m_n\opmu}^{2(p-1)} }
			\\
			&= \frac{2p}{i\hbar} \Im{\Tr{\opmu\, m_n\com{m_n,\sfA}\opmu \n{m_n\opmu}^{2(p-1)}}}.
		\end{align*}
		Therefore, by H\"older's inequality for the trace, we obtain
		\begin{align}\nonumber
			\n{I_\sfA} &\leq \frac{2p}{\hbar} \Nrm{\opmu\, m_n}{2p} \Nrm{\com{m_n,\sfA}\opmu}{2p} \Nrm{m_n\,\opmu}{2p}^{2(p-1)}
			\\\label{eq:IA}
			&\leq \frac{q}{\hbar} \Nrm{\opmu\, m_n}{q}^{q-1} \Nrm{\com{\sfA,m_n}\opmu}{q},
		\end{align}
		where we used the fact that since $\opmu$ and $m_n$ are self-adjoint, and since the Schatten norm is invariant by taking the adjoint, we have $\Nrm{m_n\,\opmu}{2p} = \Nrm{\opmu\,m_n}{2p}$. For the $\sfB$ term, we get more easily
		\begin{equation*}
			I_\sfB = \frac{2p}{i\hbar}\Im{\Tr{m_n\com{\sfB,\op}\opmu\, m_n \n{m_n\opmu}^{2(p-1)}}}\,.
		\end{equation*}
		By using again H\"older's inequality and the commutation in the Schatten norm, we obtain
		\begin{equation}\label{eq:IB}
			\n{I_\sfB} \leq \frac{q}{\hbar}\Nrm{m_n\com{\sfB,\op}}{q}\Nrm{\opmu\, m_n}{q}^{q-1}.
		\end{equation}
		We conclude the proof by combining inequalities~\eqref{eq:IA} and~\eqref{eq:IB} on Formula~\eqref{eq:IA_IB} and using the fact that $\dt \Nrm{\opmu\, m_n}{q} = \tfrac{1}{q} \Nrm{\opmu\, m_n}{q}^{1-q} \dt \Nrm{\opmu\, m_n }{q}^{q}$.
	\end{proof}

\subsection{Preliminary Inequalities}

	In order to simplify the computations, we will sometimes use weights of the form
	\begin{equation*}
		m_n = 1 + \n{\opp}^n
	\end{equation*}
	and
	\begin{equation*}
		\tilde{m}_n = 1 + \sum_{\jj = 1}^3 \opp_\jj^n.
	\end{equation*}
	Thanks to the following lemma, these weights define equivalent weighted Schatten norms.
	\begin{lem}
		Let $n\in\N$ be even. Then there exists $C>0$ such that for any $p\in [1,\infty]$ and any operator $\op$
		\begin{align}\label{eq:equiv_weight_1}
			C^{-1}\Nrm{\op\, \tilde{m}_n}{p} &\leq \Nrm{\op\, m_n}{p} \leq C \Nrm{\op\, \tilde{m}_n}{p}
			\\\label{eq:equiv_weight_2}
			C^{-1}\Nrm{\op\, \opp_{\jj}^n}{p} &\leq \Nrm{\op\, m_n}{p} \leq C \(\Nrm{\op}{p}+\sum_{\jj = 1}^3 \Nrm{\op\, \opp_{\jj}^n}{p}\).
		\end{align}
	\end{lem}
	
	\begin{proof}
		We observe that $\tilde{m}_n$ and $m_n$ commute, $m_n$ is invertible, and $m_n^{-1} \tilde{m}_n = g(\opp)$ with $\Nrm{g}{L^\infty} < C$ uniformly in $\hbar$. Therefore, we obtain
		\begin{equation*}
			\Nrm{\op\, \tilde{m}_n}{p} = \Nrm{\op\, m_n \,g(\opp)}{p} \leq C \Nrm{\op\, m_n}{p},
		\end{equation*}
		which proves the first inequality of Equation~\eqref{eq:equiv_weight_1}. The second one follows by reversing the role of $\tilde{m}_n$ and $m_n$, and the first inequality of Equation~\eqref{eq:equiv_weight_2} by replacing $\tilde{m}_n$ by $\opp_{\jj}^n$. The second inequality of Equation~\eqref{eq:equiv_weight_2} follows from the second inequality of Equation~\eqref{eq:equiv_weight_1} and the triangle inequality for Schatten norms.
	\end{proof}
	
	We will need the following \textit{rearrangement inequality} similar to~\cite[Equation (56)]{lafleche_strong_2021}.
	\begin{lem}\label{lem:rearrangement}
	Let $p\geq 1$ and $(n,m) \in \N^2$. Then for any self-adjoint operators $A$ and $B$, the following inequality holds
		\begin{equation}\label{eq:rearrangement}
			\Nrm{B^nAB^m}{p} \leq 2 \Nrm{AB^{n+m}}{p}.
		\end{equation}
	\end{lem} 
	
	\begin{proof}[Proof of Lemma~\eqref{lem:rearrangement}]
		Assume that $A\geq 0$. Then by H\"older's inequality, we have
		\begin{equation}\label{eq:rearrangement_0}
			\Nrm{B^nAB^m}{p} \leq \Nrm{B^n A^\frac{n}{n+m}}{\frac{n+m}{n}\, p} \Nrm{A^\frac{m}{n+m} B^m}{\frac{n+m}{m}\, p}.
		\end{equation}
		Now observe that since by definition of the absolute value $\n{BA}=\n{\n{B}A}$, we get $\Nrm{B^n A^\frac{n}{n+m}}{\frac{n+m}{n}\, p} = \Nrm{\n{B}^n A^\frac{n}{n+m}}{\frac{n+m}{n}\, p}$, and since the Schatten norm is invariant by taking the adjoint
		\begin{equation*}
			\Nrm{A^\frac{m}{n+m} B^m}{\frac{n+m}{m}\, p} = \Nrm{B^m A^\frac{m}{n+m}}{\frac{n+m}{m}\, p} = \Nrm{\n{B}^m A^\frac{m}{n+m}}{\frac{n+m}{m}\, p}\,.
		\end{equation*}
		Now, by the Araki--Lieb--Thirring inequality
		\begin{align*}
			\Nrm{\n{B}^n A^\frac{n}{n+m}}{\frac{n+m}{n}\, p} &\leq \Nrm{\n{B}^{n+m} A}{p}^\frac{n}{n+m} = \Nrm{ AB^{n+m}}{p}^\frac{n}{n+m},
			\\
			\Nrm{\n{B}^m A^\frac{m}{n+m}}{\frac{n+m}{m}\, p} &\leq \Nrm{\n{B}^{n+m} A}{p}^\frac{m}{n+m} = \Nrm{ AB^{n+m}}{p}^\frac{m}{n+m}.
		\end{align*}
		Combining these inequalities with Inequality~\eqref{eq:rearrangement_0} leads to
		\begin{equation}\label{eq:rearrangement_positive}
			\Nrm{B^nAB^m}{p} \leq \Nrm{AB^{n+m}}{p}.
		\end{equation}
		In the more general case of a self-adjoint operator $A$ possibly not nonnegative, we write $A = A_+ - A_-$ with $A_+ = \frac{\n{A}+A}{2}$ and $\frac{\n{A}-A}{2}$. Then by Inequality~\eqref{eq:rearrangement_positive} and the triangle inequality for Schatten norms, we get
		\begin{multline*}
			\Nrm{B^nAB^m}{p} \leq \Nrm{A_+B^{n+m}}{p} + \Nrm{A_-B^{n+m}}{p}
			\\
			\leq \frac{1}{2} \(\Nrm{\n{A}B^{n+m}+ AB^{n+m}}{p} + \Nrm{\n{A}B^{n+m}-AB^{n+m}}{p}\)
			\\
			\leq \Nrm{\n{A}B^{n+m}}{p} + \Nrm{AB^{n+m}}{p}
		\end{multline*}
		and we conclude by using again the fact that $\n{\n{A}B} = \n{AB}$.
	\end{proof}
	
	Let us define the adjoint representation of $A$ as
	\begin{equation*}
		\Ad{A}{B} := \com{A,B}.
	\end{equation*}
	Then, using the above Lemma, we can prove the following inequality.
	\begin{lem}\label{lem:expand_commutators}
		Let $n\in\N$. Then for any self-adjoint operators $A$ and $B$, we have 
		\begin{equation*}
			\Nrm{\ad{B}^n(A)}{p} \leq 2^{n+1} \Nrm{AB^n}{p}.
		\end{equation*}
	\end{lem}
	
	\begin{proof}
		This follows from the expansion
		\begin{equation*}
			\ad{B}^n(A) = \sum_{k=0}^n \binom{n}{k} \(-1\)^{k} B^{n-k} A B^k,
		\end{equation*}
		together with the triangle inequality for the Schatten norms and the rearrangement inequality~\eqref{eq:rearrangement}.
	\end{proof}
	
	\begin{lem}\label{lem:weighted_interpolation}
		Let $(p_0, p_1) \in [2,\infty]^2$ and $(n_0,n_1)\in \R_+^2$. Then for any $A$ self-adjoint 
		and 
		$\theta\in[0,1]$,
		\begin{equation}\label{eq:weighted_interpolation}
			\Nrm{A\,B^{n_\theta}}{p_\theta} \leq \Nrm{A\,B^{n_0}}{p_0}^{1-\theta} \Nrm{A\,B^{n_1}}{p_1}^\theta
		\end{equation} 
		where $\frac{1}{p_\theta} = \frac{1-\theta}{p_0} + \frac{\theta}{p_1}$ and $n_\theta = \(1-\theta\)n_0+\theta \,n_1$.
	\end{lem}
	
	\begin{proof}
		Let $S$ be the set of values of $\theta\in[0,1]$ such that Equation~\eqref{eq:weighted_interpolation} holds. Then $0$ and $1$ are in $S$. Moreover, if $\theta_1$ and $\theta_2$ are in $S$, then for $\theta := \(\theta_1+\theta_2\)/2$, since the Schatten norms are invariant by taking the adjoint and $A$ and $B$ are self-adjoint,
		\begin{equation*}
			\Nrm{A\,B^{n_\theta}}{p_\theta} = \Nrm{B^{n_\theta}\,A}{p_\theta} = \Nrm{A\,B^{2\,n_\theta}\,A}{p_\theta/2}^{1/2}
		\end{equation*}
		Hence, since $2\,n_\theta = n_{\theta_1} + n_{\theta_2}$, $p_\theta\geq 2$ and $2/p_\theta = 1/n_{\theta_1} + 1/n_{\theta_2}$, by H\"older's inequality,
		\begin{equation*}
			\Nrm{A\,B^{2\,n_\theta}\,A}{p_\theta/2}^{1/2} \leq \Nrm{A\,B^{n_{\theta_1}}}{p_{\theta_1}}^{1/2} \Nrm{A\,B^{n_{\theta_2}}}{p_{\theta_2}}^{1/2}
		\end{equation*}
		where we used again the invariance of Schatten norms by taking the adjoint. Hence $\theta\in S$, and so we deduce finally that $S$ is a dense subset of $[0,1]$.
	\end{proof}

	The next proposition {allows} us to control $\Nrm{\nabla\rho}{L^p}$ by $\Nrm{\Dhx{\op}\,m_n}{\L^p}$ for some weight $m_n$.
	\begin{prop}\label{prop:diag_vs_weights}
		Let $p\in[1,\infty]$ and $n>\frac{3}{p'}$. Then there exists a constant $C>0$ such that for any compact self-adjoint operator $\opmu$
		\begin{equation*}
			\Nrm{\Diag{\opmu}}{L^p} \leq C \Nrm{\opmu\,m_n}{\L^p},
		\end{equation*}
		with $m_n = 1+\n{\opp}^n$.
	\end{prop}
	
	\begin{remark}
		In particular, since for $k\in\N$, $\nabla^k\rho = \Diag{\Dhx^k{\op}}$, the above estimate implies
		\begin{equation*}
			\Nrm{\nabla^k\rho}{L^p} \leq C \Nrm{\Dhx^k{\op}\,m_n}{\L^p}.
		\end{equation*}
	\end{remark}

	\begin{proof}
		Let $\diagopmu(x) := \Diag{\opmu}\!(x) = h^3\opmu(x,x)$. Then, using the dual formulation of the $L^p$ norm and separating $\varphi$ into the sum of its positive and negative parts $\varphi := \varphi_++\varphi_-$, we have
		\begin{equation*}
			\Nrm{\diagopmu}{L^p} \leq \sup_{\Nrm{\varphi}{L^{p'}}\leq 1}\( \n{\intd \diagopmu\,\varphi_-} + \n{\intd \diagopmu\,\varphi_+}\),
		\end{equation*}
		from which we deduce that we can actually restrict ourselves to nonnegative functions $\varphi$ and identifying the function $\varphi$ with the multiplication operator by the function $\varphi$, we get
		\begin{equation}\label{eq:dual_Lp}
			\Nrm{\diagopmu}{L^p} \leq 2\,\sup_{\substack{\varphi\geq 0\\\Nrm{\varphi}{L^{p'}}\leq 1}} \n{\intd \diagopmu\,\varphi} = 2 \sup_{\substack{\varphi\geq 0\\\Nrm{\varphi}{L^{p'}}\leq 1}} \n{h^3 \Tr{\opmu\,\varphi}}.
		\end{equation}
		Taking $m_n(y) := \sqrt{1+\n{y}^n}$ and $w(y) = m_n(y)^{-1}$, we see that $\opm := m_n(\opp)$ is a positive invertible operator and its inverse $\opw := w(\opp)$ is a compact operator. By H\"older's inequality for the trace, we have
		\begin{equation}\label{eq:weighted_holder}
			h^3\Tr{\opmu\,\varphi} = h^3\Tr{\opm\,\opmu\, \opm\,\opw\,\varphi\,\opw} \leq \Nrm{\opm\,\opmu\,\opm}{\L^p} \Nrm{\opw\,\varphi\,\opw}{\L^{p'}}.
		\end{equation}
		However, since $\varphi$ is a nonnegative function, we get that it is also a positive operator. Hence
		\begin{equation*}
			\Nrm{\opw\,\varphi\,\opw}{\L^{p'}} = \Nrm{\n{\sqrt{\varphi}\,\opw}^2}{\L^{p'}} = \Nrm{\sqrt{\varphi}\,\opw}{\L^{2p'}}^2 \leq \Nrm{\varphi}{L^{p'}} \Nrm{w}{L^{2p'}}^2
		\end{equation*}
		where to get the last inequality we used the Kato--Seiler--Simon Inequality~\eqref{eq:Kato_Seiler_Simon} since $2p'\geq 2$. Combining the above inequality with inequalities~\eqref{eq:dual_Lp} and \eqref{eq:weighted_holder} yields
		\begin{equation*}
			\Nrm{\diagopmu}{L^p} \leq C_{p,n}\,\Nrm{\opm\,\opmu\,\opm}{\L^p} \leq C_{p,n}\,\Nrm{\opmu\,\opm^2}{\L^p}
		\end{equation*}
		where the second inequality is a consequence of Lemma~\ref{lem:rearrangement}, and $C_{p,n} = 2\Nrm{w}{L^{2p'}}^2$ is finite because $n>\frac{3}{p'}$ by assumption.
	\end{proof}

\subsection{Commutators Involving the Direct Term}\label{sec:commutators}

	In the semiclassical case, instead of $\nabla E_{\op}\cdot\Dv f$ (see Equation~\eqref{eq:time-der-grad}), the time derivative of the gradient let appear the term $\frac{1}{i\hbar} \com{E_{\op},\op}$ (see Equation~\eqref{eq:time-der-grad-op}). Hence we will need to get semiclassical estimates on this quantity, which is the purpose of the following proposition.

	\begin{prop}[Commutator estimates]\label{prop:estim_commutator_Lq_Besov} Let $a\in(0,1]$, $\fb = \frac{3}{a+1}$ and $(q,r)\in[2,\infty]^2$ be such that $\frac{1}{r}+\frac{1}{q} = \frac{1}{2}$. Then for any compact operator $\op_2$ it holds
		\begin{equation}\label{eq:estim_commutator_Lq_Besov}
			\frac{1}{\hbar}\Nrm{\com{E_{\op},\op_2}}{\L^q} \leq C\,\Nrm{\rho}{B^{1-3\(\frac{1}{r'}-\frac{1}{\fb}\)}_{r,1}} \Nrm{\Dhv\op_2}{\L^q}.
		\end{equation}
		When $q=2$ and $r=\infty$, we also have
		\begin{equation}\label{eq:estim_commutator_L2}
			\frac{1}{\hbar}\Nrm{\com{E_{\op},\op_2}}{\L^2} \leq C\,\Nrm{\nabla\rho}{L^{\fb',1}} \Nrm{\Dhv\op_2}{\L^2}
		\end{equation}
		for $a\in\left[\frac{1}{2},1\right]$, and 
		\begin{equation}\label{eq:estim_commutator_L2_slow-tail}
			\frac{1}{\hbar}\Nrm{\com{E_{\op},\op_2}}{\L^2} \leq C\,\Nrm{\rho}{L^{\frac{3}{1-a},1}} \Nrm{\Dhv\op_2}{\L^2}
		\end{equation}
		for $a\in\(0,\frac{1}{2}\)$.
	\end{prop}
	
	From the fact that $(L^r,W^{1,r})_{s,1} = B^s_{r,1}$ for any $r\in[1,\infty)$ and $s\in(0,1)$, we deduce the following inequality in terms of more classical Sobolev spaces.
	\begin{cor}\label{cor:estim_commutator_Lq}
		Let $(q,r)\in(2,\infty]\times (\fb',\infty)$ be such that $\frac{1}{r}+\frac{1}{q} = \frac{1}{2}$. Then it holds
		\begin{equation*}
			\frac{1}{\hbar}\Nrm{\com{E_{\op},\op_2}}{\L^q} \leq C \Nrm{\rho}{L^r}^{1-s} \Nrm{\rho}{W^{1,r}}^s \Nrm{\Dhv\op_2}{\L^q}
		\end{equation*}
		with $s = 1-3\(\frac{1}{r'}-\frac{1}{\fb}\)$.
	\end{cor}
	From the fact that $\frac{1}{r}+\frac{1}{q} = \frac{1}{2}$ and the fact that $r>\fb'$, when $a \geq 1/2$, the above results only work when $q<q_a$ with $\frac{1}{q_a} = \frac{1}{\fb} - \frac{1}{2}$.

		
	\begin{proof}[Proof of Proposition~\ref{prop:estim_commutator_Lq_Besov}]
		First observe that the integral kernel of the operator $\com{E_{\op},\op_2}$ can be written as
		\begin{align*}
			\com{E_{\op},\op_2}\!(x,y) &= \(E_{\op}(x)-E_{\op}(y)\)\op_2(x,y)
			\\
			&= \frac{\(E_{\op}(x)-E_{\op}(y)\)\otimes\(x-y\)}{\n{x-y}^2}\cdot \(x-y\)\op_2(x,y).
		\end{align*}
		Thus, we can explicitly compute its Hilbert--Schmidt norm by computing the $L^2$ norm of the kernel, and since the kernel of the operator $\Dhv{\op_2}$ is $\frac{x-y}{i\hbar}\op_2(x,y)$, we get the following estimate 
		\begin{multline}\label{eq:estim_commutator_L2_0}
			\frac{1}{\hbar}\Nrm{\com{E_{\op},\op_2}}{2} = \(\iintd\n{\frac{\(E_{\op}(x)-E_{\op}(y)\)\otimes\(x-y\)}{\n{x-y}^2}\cdot \Dhv{\op_2}(x,y)}^2 \d x\d y\)^{\frac{1}{2}}
			\\
			\leq \Nrm{\nabla E_{\op}}{L^\infty} \Nrm{\Dhv{\op_2}}{2}.
		\end{multline}
		In particular, for $a\in\left[\frac{1}{2},1\right]$, since $\nabla E_{\op} = \nabla K*\nabla \rho$ with $\nabla K\in L^{\fb,\infty}$, we deduce Inequality~\eqref{eq:estim_commutator_L2} using the fact that the dual of $L^{\fb',1}$ is $L^{\fb,\infty}$ (see e.g. \cite{hunt_lpq_1966})
		
		If $a\in\(0,\frac{1}{2}\)$, we use that $\nabla E_{\op} = \nabla^2 K* \rho$ with $\nabla^2 K\in L^{\frac{3}{a+2},\infty}$. Thus Inequality~\eqref{eq:estim_commutator_L2_slow-tail} follows from H\"older's inequality for Lorentz norms.
		
		A second possibility is to use the fundamental theorem of calculus for $E_{\op}$ and then the Fourier inversion theorem to rewrite the integral kernel of the commutator as
		\begin{align*}
			\frac{1}{i\hbar}[E_{\op},\op_2](x,y) &= \int_0^1\nabla E_{\op}((1-\theta)x+\theta y)\d\theta\cdot (\Dhv\op_2)(x,y)
			\\
			&= \int_{[0,1]\times\R^3} \widehat{\nabla E_{\op}}(z)\, e^{2i\pi z\cdot(1-\theta)x}\cdot (\Dhv\op_2)(x,y)\, e^{2i\pi z\cdot\theta y}\d\theta\d z,
		\end{align*}
		which implies that by writing $e_\omega$ the operator of multiplication by $e^{2i\pi\omega\cdot x}$, it holds
		\begin{equation*}
			\frac{1}{i\hbar}[E_{\op},\op_2] = \int_{[0,1]\times\R^3} \widehat{\nabla E_{\op}}(z)\, e_{(1-\theta)z} \(\Dhv\op_2\) e_{\theta z} \d\theta\d z.
		\end{equation*}
		Since the operators $e_\omega$ are bounded operators of norm $\Nrm{e_\omega}{\infty} = 1$, we deduce the following estimate on the operator norm of the commutator
		\begin{equation}\label{eq:estim_commutator_Linfty}
			\frac{1}{\hbar}\Nrm{\com{E_{\op},\op_2}}{\infty} \leq \int_{[0,1]\times\R^3} \n{\widehat{\nabla E_{\op}}(z)} \Nrm{\Dhv\op_2}{\infty} \d\theta\d z = \Nrm{\widehat{\nabla E_{\op}}}{L^1} \Nrm{\Dhv\op_2}{\infty}.
		\end{equation}
		In order to get a result for a general $q\in[2,\infty]$, we proceed by complex interpolation. Defining the vector-valued Hilbert--Schmidt operator $\opmu := \Dhv{\op_2}$, we observe that $\op_2(x,y) = i\hbar\, \frac{x-y}{\n{x-y}^2}\cdot\opmu(x,y)$ and the commutator can be rewritten as the bilinear operator
		\begin{equation*}
			\Lambda(E,\opmu) := \com{E,\frac{x-y}{\n{x-y}^2}\cdot\opmu} = \frac{1}{i\hbar} \com{E,\op_2}.
		\end{equation*}
		Thus, using the fact that $B^0_{\infty,1}\subset L^\infty$ and $B^\frac{3}{2}_{2,1}\subset \F{L^1}$, inequalities~\eqref{eq:estim_commutator_L2_0} and \eqref{eq:estim_commutator_Linfty} imply
		\begin{align*}
			\Nrm{\Lambda(E,\opmu)}{2} &\leq C \Nrm{E}{B^1_{\infty,1}} \Nrm{\opmu}{2},
			\\
			\Nrm{\Lambda(E,\opmu)}{\infty} &\leq C \Nrm{E}{B^{1+\frac{3}{2}}_{2,1}} \Nrm{\opmu}{\infty}.
		\end{align*}
		By the same proof, one obtains the inequality for any vector-valued Hilbert--Schmidt operator $\opmu$. Finally, we use the fact that the complex interpolation space between the involved Besov spaces is given by $[B^1_{\infty,1},B^{1+\frac{3}{2}}_{2,1}]_{\frac{2}{r}} = B^{1+\frac{3}{r}}_{r,1}$ (see for example~\cite[Theorem 6.4.5]{bergh_interpolation_1976}), while the complex interpolation of Schatten spaces $\mathfrak{S}^q$ gives $[\mathfrak{S}^2,\mathfrak{S}^\infty]_{1-\frac{2}{q}} = \mathfrak{S}^q$ (see for example \cite[Section 1.19.7]{triebel_interpolation_1978}), so that by bilinear interpolation (see \cite[Section~4.4]{bergh_interpolation_1976}), we obtain
		\begin{equation*}
			\Nrm{\Lambda(E,\opmu)}{q} \leq C \Nrm{E}{B^{1+\frac{3}{r}}_{r,1}} \Nrm{\opmu}{q} \quad \text{ with } \quad \frac{1}{r} = \frac{1}{2} - \frac{1}{q}.
		\end{equation*}
		If we take $E_{\op} = \nabla K * \rho$ with $\rho\in L^1\cap L^p$ for some $p\in(1,\infty)$ and $\nabla K\in L^{\fb,\infty}$, we know that $E_{\op}\in L^{\tilde{r}}$ for some $\tilde{r}\in(1,\infty)$. Moreover, $E_{\op}$ is proportional to $(-\Delta)^\frac{a-3}{2}\nabla\rho$, therefore we can apply \cite[Proposition 2.30]{bahouri_fourier_2011} and deduce that
		\begin{equation*}
			\Nrm{E_{\op}}{B^{1+\frac{3}{r}}_{r,1}} \leq C \Nrm{\rho}{B^{\frac{3}{r}+a-1}_{r,1}} = \Nrm{\rho}{B^{1-3\(\frac{1}{r'}-\frac{1}{\fb}\)}_{r,1}}.
		\end{equation*}
		Taking $\opmu = \Dhv{\op_2}$ yields the results.
	\end{proof}
	
	To get estimates with weights, notice that we can write $\com{E_{\op},\op}\opp_\jj^{2n}$ in the following form
	\begin{equation*}
		\frac{1}{i\hbar}\com{E_{\op},\op}\opp_\jj^{2n} = \frac{1}{i\hbar}\com{E_{\op},\op\, \opp_\jj^{2n}} - \frac{1}{i\hbar}\com{E_{\op},\opp_\jj^{2n}}\op.
	\end{equation*}
	To control the $\L^q$ norm of the first term of the right-hand side we use Proposition~\ref{prop:estim_commutator_Lq_Besov}, which gives
	\begin{equation*}
		\frac{1}{\hbar}\Nrm{\com{E_{\op},\op\,\opp_\jj^{2n}}}{\L^q} \leq C \Nrm{\rho}{B^{1-3\(\frac{1}{r'}-\frac{1}{\fb}\)}_{r,1}} \Nrm{\Dhv(\op\,\opp_\jj^{2n})}{\L^q},
	\end{equation*}
	and we can also replace $\Nrm{\rho}{B^{1-3\(\frac{1}{r'}-\frac{1}{\fb}\)}_{r,1}}$ by $\Nrm{\nabla\rho}{L^{\fb',1}}$ when $q=2$. 	
%
%


	To bound the second term, we will write the potential $K(x)$ as a sum of a singular part localized near $x=0$ and a long-range part and use Proposition~\ref{prop:weighted_com_est} and Proposition~\ref{prop:weighted_com_est_long_range} below. More precisely, for some infinitely smooth and compactly supported function $\chi$ verifying $\indic_{|x|\leq 1}\leq \chi(x) \leq \indic_{|x|\leq 2}$, we can write 
	\begin{equation}\label{eq:splitting-K}
		K = K_0 + K_\infty
	\end{equation} 
	with 
	\begin{equation*}
		K_0 = \chi\,K\quad \text{ and } \quad K_\infty = \(1-\chi\) K.
	\end{equation*}
	Now the first part of $K$ verifies $\nabla K_0\in L^\fb$ for any $\fb<3/2$ while $K_\infty \in L^\infty\cap C^\infty$. We start by the following proposition to control the singular part of the potential.

	\begin{prop}[Weighted commutator estimate]\label{prop:weighted_com_est}
		Let $E_{\op}^0 = -\nabla K_0 * \rho$ with $K_0 = \chi\,K$ as described above and let $m_n := 1+\n{\opp}^n$. Take $(q,r,r_1)\in \left[\frac{3}{2},\infty\right]\times[1,\infty]^2$, such that
		\begin{equation}\label{eq:exponents_weighted_commutator}
			\frac{1}{r} + \frac{1}{r_1} = \frac{1}{q} + \frac{1}{3}.
		\end{equation}
		Then for any $n_0>3/\fb-1$, $k' > 3/r' - 2$ and $k > 3/r-1$ there exists a constant $C>0$ such that
		\begin{multline*}
			\frac{1}{\hbar}\Nrm{\com{E_{\op}^0,\opp_\jj^{2n}}\opmu}{\L^q}
			\\
			\leq C \(\Nrm{\Dhxj\op \,m_{n+n_0}}{\L^{\fb'}} \Nrm{\opmu \,m_{2n-1}}{\L^q} + \Nrm{\Dhxj{\op}\,m_{2n+k'}}{\L^r} \Nrm{\opmu \,m_{n+k}}{\L^{r_1}}\).
		\end{multline*}
		Replacing $E_{\op}^0$ by $V_{\op}^0 = K_0 * \rho$ just amounts to replacing $\Dhxj{\op}$ by $\op$, hence by the same proof,
		\begin{equation*}
			\frac{1}{\hbar}\Nrm{\com{V_{\op}^0,\opp_\jj^{2n}}\opmu}{\L^q} \leq C \(\Nrm{\op \,m_{n+n_0}}{\L^{\fb'}} \Nrm{\opmu \,m_{2n-1}}{\L^q} + \Nrm{\op\,m_{2n+k'}}{\L^r} \Nrm{\opmu \,m_{n+k}}{\L^{r_1}}\).
		\end{equation*}
	\end{prop}

	\begin{proof}[Proof of Proposition~\ref{prop:weighted_com_est}]
		To shorten the notation, let $E := E_{\op}^0$. We notice that $\com{E,\opp_\jj} = E\,\opp_\jj-\opp_\jj\,E = i\hbar\,\partial_{\jj} E$ is the operator of multiplication by $x\mapsto i\hbar\,\partial_{\jj} E(x)$, and since $\opp_\jj^2 = \opp_\jj\opp_\jj$, we get
		\begin{equation*}
			\frac{1}{i\hbar}\com{E,\opp_\jj^2} = \frac{1}{i\hbar}\(\com{E,\opp_\jj}\opp_\jj + \opp_\jj\com{E,\opp_\jj}\) = \(\partial_{\jj} E\) \opp_\jj + \opp_\jj\( \partial_{\jj} E\),
		\end{equation*}
		and more generally, for any $n\in\N$
		\begin{equation*}
			\frac{1}{i\hbar}\com{E,\opp_\jj^{2n}} = \sum_{k=0}^{n-1} \opp_\jj^{2k} \(\partial_{\jj} E\, \opp_\jj + \opp_\jj\,\partial_{\jj} E\) \opp_\jj^{2\(n-k-1\)} = \sum_{k=0}^{2n-1} \opp_\jj^{k}\, \partial_{\jj} E\, \opp_\jj^{2n-1-k}.
		\end{equation*}
		From this formula and the triangle inequality for Schatten norms, we deduce
		\begin{equation*}
			\frac{1}{\hbar}\Nrm{\com{E,\opp_\jj^{2n}}\opmu\,}{\L^q} \leq \sum_{k=0}^{2n-1}\Nrm{\opp_\jj^{k}\, \partial_{\jj} E\, \opp_\jj^{2n-1-k}\opmu\,}{\L^q}.
		\end{equation*}
		We cannot directly  apply  H\"older's inequality here since $\opp_\jj^{k} E$ is an unbounded operator, therefore we have to make some commutations between $\opp_\jj^{k}$ and $\partial_{\jj} E$. By the Leibniz formula
		\begin{equation*}
			\opp_\jj^{k}\, \partial_{\jj} E = \sum_{\ell=0}^k \binom{k}{\ell} g_\ell\, \opp_\jj^{k-\ell},
		\end{equation*}
		where $g_\ell$ is the function defined by $g_\ell(x) = (\opp_\jj^\ell(\partial_{\jj} E))(x)$, as usual also identified with a multiplication operator. This yields
		\begin{equation*}
			\frac{1}{\hbar}\Nrm{\com{E,\opp_\jj^{2n}}\opmu\,}{\L^q} \leq \sum_{\ell=0}^{2n-1} C_\ell \Nrm{g_\ell\,\opp_\jj^{2n-1-\ell}\opmu\,}{\L^q},
		\end{equation*}
		where $C_\ell = \sum_{k=\ell}^{2n-1} \binom{k}{\ell}$. We will now distinguish two cases to bound the terms of the sum depending on the values of $\ell$.
		
		\step{1. Case $\ell$ small} Take $\ell < n$. In this case, we use H\"older's inequality for Schatten norms and the fact that the norm of the operator of multiplication by a function is the $L^\infty$ norm of this function to deduce that
		\begin{equation*}
			\Nrm{g_\ell\,\opp_\jj^{2n-1-\ell}\opmu\,}{\L^q} \leq \Nrm{g_\ell}{L^\infty}\Nrm{\opp_\jj^{2n-1-\ell}\opmu\,}{\L^q} \leq \Nrm{g_\ell}{L^\infty} \Nrm{\opmu\,m_{2n-1}}{\L^q},
		\end{equation*}
		where we used Inequality~\eqref{eq:equiv_weight_2}. Now, observe that
		\begin{equation*}
			g_\ell = -\nabla K_0 * (\opp_\jj^\ell\partial_{\jj}\rho) = -\nabla K_0 * \Diag{\ad{\opp_\jj}^\ell\!\!\(\Dhxj{\op}\)}.
		\end{equation*}
		Therefore, since $\nabla K_0\in L^\fb$ with $\fb<3/2$, by Young's inequality
		\begin{equation*}
			\Nrm{g_\ell}{L^\infty} \leq C_K \Nrm{\Diag{\ad{\opp_\jj}^\ell\!\!\(\Dhxj{\op}\)}}{L^{\fb'}},
		\end{equation*}
		where $C_K = \Nrm{\nabla K_0}{L^\fb}$. By Proposition~\ref{prop:diag_vs_weights} and Lemma~\ref{lem:expand_commutators}, we get that for any $n_0 > 3/\fb - 1 > 0$,
		\begin{equation*}
			\Nrm{g_\ell}{L^\infty} \leq C_{K,n_0} \Nrm{\ad{\opp_\jj}^\ell\!\!\(\Dhxj{\op}\) m_{2+n_0}}{\L^{\fb'}} \leq 2^{\ell+1} C_{K,n_0} \Nrm{\Dhxj{\op}\,m_{n+n_0}}{\L^{\fb'}}
		\end{equation*}
		where we used the fact that $\ell \leq n-1$.
		
		\step{2. Case $\ell$ large} Take $\ell\in\N\cap [n,2n-1]$ and define $\frac{1}{\tilde{q}} = \frac{1}{q} + \frac{1}{3}$. Then since $\tilde{q} \leq q$, 
		\begin{equation*}
			\Nrm{\cdot}{\L^q} = h^{3/q} \Nrm{\cdot}{q} \leq h^{3/q} \Nrm{\cdot}{\tilde{q}} = h^{-1} \Nrm{\cdot}{\L^{\tilde{q}}}.
		\end{equation*}
		Since $\frac{1}{r} = \frac{1}{\tilde{q}} - \frac{1}{r_1}$, and multiplying and dividing by $m_k := 1+\n{\opp}^{k}$, we deduce
		\begin{align*}
			\Nrm{g_\ell\,\opp_\jj^{2n-1-\ell}\opmu\,}{\L^q} &\leq h^{-1}\Nrm{g_\ell\, m_k^{-1} m_k\,\opp_\jj^{2n-1-\ell}\opmu\,}{\L^{\tilde{q}}}
			\\
			&\leq C_m^{1/r} \,h^{-1}\Nrm{g_\ell}{L^r} \Nrm{\opmu\,m_{n+k-1}}{\L^{r_1}},
		\end{align*}
		%
		where $C_m = \intd \frac{\d x}{(1+\n{x}^k)^r}$ is finite because $k<3/r$ and we used the fact that $\ell \geq n$. Note that since $\ell\geq 1$
		\begin{equation*}
			g_\ell = -\partial_{\jj}\nabla K_0 * \opp_\jj^\ell\rho = i\hbar\(\partial_{\jj}\nabla K_0\) *  \Diag{\ad{\opp_\jj}^{\ell-1}\!\!\(\Dhxj{\op}\)}.
		\end{equation*}
		Hence, to control $g_\ell$, we can use the fact that the convolution by $\partial_{\jj}\nabla K_0$ is continuous from $L^r$ to $L^r$ by the  Calder\'on--Zygmund Theorem (see e.g. \cite[Theorem~5.1]{duoandikoetxea_fourier_2001}) to obtain
		\begin{equation*}
			\Nrm{g_\ell}{L^r} \leq C_K \,h \Nrm{\Diag{\ad{\opp_\jj}^{\ell-1}\!\!\(\Dhxj{\op}\)}}{L^r}.
		\end{equation*}
		%
		By Proposition~\ref{prop:diag_vs_weights} and Lemma~\ref{lem:expand_commutators}, it yields for any $\eps = k'+2-3/r' > 0$,
		\begin{equation*}
			h^{-1}\,\Nrm{g_\ell}{L^r} \leq C_{K,\eps} \Nrm{\ad{\opp_\jj}^{\ell-1}\!\!\(\Dhxj{\op}\) m_{3/r'+\eps}}{\L^r} \leq 2^\ell\, C_{K,k'} \Nrm{\Dhxj{\op}\,m_{2n+k'}}{\L^r}.
		\end{equation*}
	\end{proof}
	
	Thanks to Lemma~\ref{lem:weighted_interpolation}, we can modify Proposition~\ref{prop:weighted_com_est} in a way depending only on the $\L^2$ and $\L^4$ norms.
	\begin{cor}\label{cor:direct-term_L2-L4}
		Assume $n\geq 3$, then there exists $\theta\in(0,1)$ and $C>0$ such that
		\begin{align*}
			\frac{C}{\hbar}\Nrm{\com{E_{\op}^0,\opp_\jj^{2n}}\opmu}{\L^2} &\leq \Nrm{\Dhxj{\op}\,m_{2n}}{\L^2}^{1-\theta} \Nrm{\Dhxj{\op}\,m_{2n-2}}{\L^4}^\theta \Nrm{\opmu \,m_{2n}}{\L^2}
			\\
			&\quad+ \Nrm{\Dhxj{\op}\,m_{2n}}{\L^2} \Nrm{\opmu\,m_{2n}}{\L^2}^{1/3} \Nrm{\opmu\,m_{2n-2}}{\L^4}^{2/3}
		\end{align*}
		and
		\begin{align*}
			\frac{C}{\hbar}\Nrm{\com{E_{\op}^0,\opp_\jj^{2n}}\opmu}{\L^4} &\leq \Nrm{\Dhxj{\op}\,m_{2n+2}}{\L^2}^{1-\theta} \Nrm{\Dhxj{\op}\,m_{2n}}{\L^4}^\theta \Nrm{\opmu \,m_{2n}}{\L^4}
			\\
			&\quad+ \Nrm{\Dhxj{\op}\,m_{2n+2}}{\L^2}^{1/3} \Nrm{\Dhxj{\op}\,m_{2n}}{\L^4}^{2/3} \Nrm{\opmu \,m_{2n}}{\L^4}
		\end{align*}
	\end{cor}
	
	\begin{proof}
		In the case $q=2$, use Proposition~\ref{prop:weighted_com_est} with $r=2$, $r_1 = 3$ to get for any $n_0 > 3/\fb - 1$ and $k > 1/2$,
		\begin{equation*}
			\frac{1}{\hbar}\Nrm{\com{E_{\op}^0,\opp_\jj^{2n}}\opmu}{\L^2} \leq C \(\Nrm{\Dhxj\op \,m_{n+n_0}}{\L^{\fb'}} \Nrm{\opmu \,m_{2n}}{\L^2} + \Nrm{\Dhxj{\op}\,m_{2n}}{\L^2} \Nrm{\opmu \,m_{n+k}}{\L^3}\).
		\end{equation*}
		Since $\nabla K_0\in L^\fb$ for any $\fb < 3/2$, we can in particular take $\fb\in(4/3, 3/2)$ so that $\fb'\in (3,4)$ and we can apply Lemma~\ref{lem:weighted_interpolation} with $p_0 = 2$, $p_1 = 4$ and $p_\theta = b'$ leading to
		\begin{equation*}
			\Nrm{\Dhxj{\op}\,m_{n+n_0}}{\L^{\fb'}} \leq \Nrm{\Dhxj{\op}\,m_{2n}}{\L^2}^{1-\theta} \Nrm{\Dhxj{\op}\,m_{\tilde{n}}}{\L^4}^\theta
		\end{equation*}
		where $\theta = 4\(\frac{1}{b} - \frac{1}{2}\)\in(\frac{2}{3},1)$ and $\tilde{n} \geq \frac{\(2\theta-1\)n+n_0}{\theta} \in (\frac{n+3n_0}{2},n+n_0)$. On the other hand, taking $p_\theta = 3$ yields
		\begin{equation*}
			\Nrm{\opmu\,m_{n+k}}{\L^3} \leq \Nrm{\opmu\,m_{2n}}{\L^2}^{1/3} \Nrm{\opmu\,m_{\tilde{n}}}{\L^4}^{2/3}
		\end{equation*}
		with $\tilde{n} \geq \frac{n+3k}{2}$. In particular, when $n\geq 3 $, taking $\fb$ close to $3/2$, $k$ close to $1/2$ and $n_0$ close to $1$ allows to take $\tilde{n} \leq 2n-2$.
		
		In the case $q=4$, take $r=3$ and $r_1 = 4$ in Proposition~\ref{prop:weighted_com_est} to get for any $n_0 > 3/\fb - 1$ and $k'>0$, 
		\begin{equation*}
		\begin{split}
			\frac{1}{\hbar}&\Nrm{\com{E_{\op}^0,\opp_\jj^{2n}}\opmu}{\L^4}\\ &\ \ \ \leq C \(\Nrm{\Dhxj\op \,m_{n+n_0}}{\L^{\fb'}} \Nrm{\opmu \,m_{2n}}{\L^4} + \Nrm{\Dhxj{\op}\,m_{2n+k'}}{\L^3} \Nrm{\opmu \,m_{2n}}{\L^4}\).
			\end{split}
		\end{equation*}
		As previously, we interpolate the $\L^{\fb'}$ norm between the $\L^2$ and the $\L^4$ norm leading to
		\begin{equation*}
			\Nrm{\Dhxj{\op}\,m_{n+n_0}}{\L^{\fb'}} \leq \Nrm{\Dhxj{\op}\,m_{\tilde{n}}}{\L^2}^{1-\theta} \Nrm{\Dhxj{\op}\,m_{2n}}{\L^4}^\theta
		\end{equation*}
		with again $\theta = 4\(\frac{1}{\fb} - \frac{1}{2}\)\in(\frac{2}{3},1)$ but with $\tilde{n} \geq \frac{\(1-2\theta\)n+n_0}{\theta}$ possibly negative. On the other hand, taking $p_\theta = 3$ yields
		\begin{equation*}
			\Nrm{\Dhxj{\op}\,m_{2n+k'}}{\L^3} \leq \Nrm{\Dhxj{\op}\,m_{\tilde{n}}}{\L^2}^{1/3} \Nrm{\Dhxj{\op}\,m_{2n}}{\L^4}^{2/3}
		\end{equation*}
		with $\tilde{n} \geq 2n+3k'$. In particular, taking $b$ close to $3/2$, $n_0$ close to $1$ and $k'$ sufficiently small allows to take $\tilde{n} \leq 2n+2$.
	\end{proof}
	
	Now we treat the long range part $K_\infty$ of the potential.
	\begin{prop}\label{prop:weighted_com_est_long_range}
		Let $E_{\op}^\infty = -\nabla K_\infty*\rho$ and $V_{\op}^\infty = K_\infty*\rho$ with $\rho = \Diag{\op}$ and $n\geq 1$ be an integer, then there exists a constant $C>0$ independent of $\hbar$ such that for any $q \in [1,\infty]$ and any positive compact operators $\op$ and $\opmu$,
		\begin{align*}
			\frac{1}{\hbar}\Nrm{\com{E_{\op}^\infty,\opp_\jj^{2n}}\opmu\,}{\L^q} &\leq C\(\Nrm{\op\, m_{2n}}{\L^2} + \hbar \Nrm{\Dhxj{\op}\, m_{2n}}{\L^2}\) \Nrm{\opmu\,m_{2n}}{\L^q}
			\\
			\frac{1}{\hbar}\Nrm{\com{V_{\op}^\infty,\opp_\jj^{2n}}\opmu\,}{\L^q} &\leq C\(\Nrm{\rho}{L^1} + \hbar \Nrm{\op\, m_{2n}}{\L^2}\) \Nrm{\opmu\,m_{2n}}{\L^q}.
		\end{align*}
	\end{prop}

	\begin{proof}[Proof of Proposition~\ref{prop:weighted_com_est_long_range}]
		As in the previous case, i.e. the proof of Proposition~\ref{prop:weighted_com_est}, and with the same notations, it holds
		\begin{equation*}
			\frac{1}{\hbar}\Nrm{\com{E_{\op}^\infty,\opp_\jj^{2n}}\opmu\,}{\L^q} \leq \sum_{\ell=0}^{2n-1} C_\ell \Nrm{g_\ell\,\opp_\jj^{2n-1-\ell}\opmu\,}{\L^q},
		\end{equation*}
		where $C_\ell = \sum_{k=\ell}^{2n-1} \binom{k}{\ell}$. We use H\"older's inequality for Schatten norms and the fact that the norm of the operator of multiplication by a function is the $L^\infty$ norm of this function to deduce that
		\begin{equation*}
			\Nrm{g_\ell\,\opp_\jj^{2n-1-\ell}\opmu\,}{\L^q} \leq \Nrm{g_\ell}{L^\infty}\Nrm{\opp_\jj^{2n-1-\ell}\opmu\,}{\L^q} \leq \Nrm{g_\ell}{L^\infty} \Nrm{\opmu\,m_{2n}}{\L^q},
		\end{equation*}
		where we used Inequality~\eqref{eq:equiv_weight_2}. Now, observe that
		\begin{equation*}
			g_\ell = -\partial_{\jj}\nabla K_\infty * (\opp_\jj^\ell\rho) = -\partial_{\jj}\nabla K_\infty * \Diag{\ad{\opp_\jj}^\ell\!\!\(\op\)}.
		\end{equation*}
		Therefore, since $\partial_{\jj}\nabla K_\infty\in L^2$, by Young's inequality, which is just the Cauchy--Schwarz inequality in this case
		\begin{equation*}
			\Nrm{g_\ell}{L^\infty} \leq C_K \Nrm{\Diag{\ad{\opp_\jj}^\ell\!\!\(\op\)}}{L^2},
		\end{equation*}
		where $C_K = \Nrm{\partial_{\jj}\nabla K_\infty}{L^2}$. By Proposition~\ref{prop:diag_vs_weights}, we get that
		\begin{equation*}
			\Nrm{g_\ell}{L^\infty} \leq C \Nrm{\ad{\opp_\jj}^\ell\!\!\(\op\) m_2}{\L^2}.
		\end{equation*}
		When $\ell = 0$, since $2n\geq 2$, it implies that
		\begin{equation*}
			\Nrm{g_\ell}{L^\infty} \leq C \Nrm{\op \,m_{2n}}{\L^2}.
		\end{equation*}
		When $\ell > 0$, we use the fact that $\ad{\opp_\jj}\!\!\(\op\) = -i\hbar \,\Dhxj{\op}$, $\ell \leq 2n-1$ and Lemma~\ref{lem:expand_commutators} to get
		\begin{equation*}
			\Nrm{g_\ell}{L^\infty} \leq C\,\hbar \Nrm{\ad{\opp_\jj}^{\ell-1}\!\!\(\Dhxj{\op}\) m_2}{\L^2} \leq 2^\ell\,C_{K,n_0}\,\hbar \Nrm{\Dhxj{\op}\, m_{2n}}{\L^2}.
		\end{equation*}
		
		In the case when $E_{\op}^\infty$ is replaced by $V_{\op}^\infty$, one obtains the same estimates with $\ell>0$ and $\Dhxj{\op}$ replaced by $\op$. The only remaining point is the case $\ell = 0$, that is defining $g_\ell = -\partial_{\jj} K_\infty * (\opp_\jj^\ell\rho)$, it remains to notice that since $\partial_{\jj} K_\infty \in L^{\infty}$,
		\begin{equation*}
			\Nrm{g_0}{L^\infty} = \Nrm{\partial_{\jj} K_\infty * \rho}{L^\infty} \leq C_K\,\Nrm{\rho}{L^1}
		\end{equation*}
		with $C_K = \Nrm{\nabla K_\infty}{L^\infty}$.
	\end{proof}
	

\subsection{Preliminary Properties of the Exchange Operator}

	\subsubsection{Preliminary Identities} Let $\sfX = \sfX_{\op}$ be the operator of integral kernel $\sfX(x,y) = K(x-y)\,\op(x,y)$ with $K(x) = \n{x}^{-a}$ and recall the notation of the adjoint representation of $A$, $\Ad{A}{B} = \com{A,B}$.
	\begin{lem}\label{lem:X_commut_gradients}
		Let $a\in(0,1]$. Then the following identities holds true
		\begin{align*}
			\com{x,\sfX_{\op}} &= \sfX_{\com{x,\op}}, & \com{\nabla,\sfX_{\op}} &= \sfX_{\com{\nabla,\op}},
		\end{align*}
		and more generally, with the adjoint notation, $\ad{x}^n(\sfX_{\op}) = \sfX_{\ad{x}^n(\op)}$ and $\ad{\nabla}^n(\sfX_{\op}) = \sfX_{\ad{\nabla}^n(\op)}$. In particular, since $\Dhx{\op} = \Ad{\nabla}\op$ and $\Dhv{\op} = \frac{1}{i\hbar}\Ad{x}\op$, it can be written $\Dhv^n(\sfX_{\op}) = \sfX_{\Dhv^n\op}$, and $\Dhx^n(\sfX_{\op}) = \sfX_{\Dhx^n\op}$.
	\end{lem}
	
	\begin{proof}
		The first identity follows immediately by looking at the integral kernel of the operator
		\begin{equation*}
			\com{x,\sfX_{\op}}(x,y) = \frac{\(x-y\)\r(x,y)}{\n{x-y}^a} = \sfX_{\com{x,\op}}(x,y).
		\end{equation*}
		To get the second we take $\varphi\in C^\infty_c$, integrate by parts and use the fact that $\nabla_x K(x-y) = -\nabla_y K(x-y)$ to get
		\begin{multline*}
			\com{\nabla,\sfX_{\op}}\!\varphi(x) = \nabla \intd \frac{\op(x,y)}{\n{x-y}^a}\,\varphi(y)\d y - \intd \frac{\op(x,y)}{\n{x-y}^a}\,\nabla\varphi(y)\d y
			\\
			= \intd \(\nabla_x +\nabla_y\)\(\frac{\op(x,y)}{\n{x-y}^a}\)\,\varphi(y)\d y = \intd \frac{\(\nabla_x +\nabla_y\)\!\(\op(x,y)\)}{\n{x-y}^a}\,\varphi(y)\d y,
		\end{multline*}
		and we conclude by noticing that $\(\nabla_x +\nabla_y\)\!\(\op(x,y)\)$ is nothing but the integral kernel of the operator $\com{\nabla,\op}$.
	\end{proof}

	\begin{lem}
		Let $\eta=1$, $\eta=\opp_\jj$ or $\eta=x_\jj$, then we have
		\begin{align}\label{eq:commut_X_0}
			\eta^n\,\sfX_{\op} &= \sum_{k=0}^n \binom{n}{k} \sfX_{\ad{\eta}^k(\op)}\,\eta^{n-k}
			\\\label{eq:commut_X_1}
			\com{\eta^n,\sfX_{\op}} &= \sum_{k=1}^n \binom{n}{k} \sfX_{\ad{\eta}^k(\op)}\,\eta^{n-k}.
		\end{align}
	\end{lem}

	\begin{proof}
		Since $A^k B = (A^{k-1}B)A + A^{k-1}\Ad{A}{B}$, we easily deduce the following commutator expansion
		\begin{equation*}
			A^nB = \sum_{k=0}^n \binom{n}{k} \ad{A}^k(B) \, A^{n-k}.
		\end{equation*}
		Therefore, we deduce the result by taking $B = \sfX_{\op}$ and by Lemma~\ref{lem:X_commut_gradients}.
	\end{proof}
	
	\subsubsection{Preliminary Inequalities} We know from \cite[Equation~(39a)]{lafleche_strong_2021} that if $a\in[0,3/2)$ and $q=2$, then
	\begin{equation}\label{eq:X_HS_bound}
		\Nrm{\sfX_{\op}}{q} \leq C\,h^{-a} \Nrm{\op \n{\opp}^a}{2}.
	\end{equation}
	By the fact that the Schatten norms of smaller order controls the Schatten norms of higher order, we deduce that this inequality actually holds for any $q\in[2,\infty]$. The next proposition will allow us to put the weight $\n{\opp}^a$ on an other operator $\opmu$ instead of $\op$.
	\begin{lem}\label{lem:X_commut_p} Let $\opmu$ and $\tilde{\opmu}$ be compact operators. Then for any $q\in[2,\infty]$ and any $\theta\in\{0,1\}$,
		\begin{equation}\label{eq:X_commut_p}
			\Nrm{\sfX_{\tilde{\opmu}}\,\opmu}{q} \leq C_{a}\,h^{-a}\Nrm{\tilde{\opmu}\n{\opp}^{a\(1-\theta\)}}{2} \Nrm{\opmu^*\n{\opp}^{\theta a}}{\infty},
		\end{equation}
		where $\opmu^*$ is the adjoint operator of $\opmu$.
	\end{lem}
	
	\begin{proof}
		Take $\opmu_2$ a compact but possibly not self-adjoint operator. Then
		\begin{align*}
			\Nrm{\sfX_{\opmu_2}}{2}^2 &= \Tr{\sfX_{\opmu_2}^*\,\sfX_{\opmu_2}} = \iintd \frac{\opmu_2^*(x,y)\,\opmu_2(y,x)}{\n{x-y}^{2a}}\d x\d y
			\\
			&= \iintd \frac{\n{\opmu_2(x,y)}^2}{\n{x-y}^{2a}}\d x\d y = \iintd \frac{\n{\opmu_2(x,y+x)}^2}{\n{y}^{2a}}\d x\d y,
		\end{align*}
		so that by the Hardy--Rellich inequality, 
		\begin{align*}
			\Nrm{\sfX_{\opmu_2}}{2}^2 &\leq \cC_{a}\,\iintd \n{\Delta^{a/2}_y\opmu_2(x,y+x)}^2\d x\d y = \cC_{a}\,\iintd \n{\Delta^{a/2}_y\opmu_2(x,y)}^2\d x\d y
			\\
			&\leq C_{a}\,h^{- 2a}\iintd \n{\(\opmu_2 \n{\opp}^a\)(x,y)}^2\d x\d y,
		\end{align*}
		where $\cC_{a}$ is the constant appearing in the Hardy--Rellich inequality and $C_{a} = \(2\pi\)^a\cC_{a}$. From this we deduce the generalization of \eqref{eq:X_HS_bound} for possibly not self-adjoint operators
		\begin{equation}\label{eq:X_HS_bound_2}
			\Nrm{\sfX_{\opmu_2}}{2} \leq C_{a}\,h^{- a}\Nrm{\opmu_2 \n{\opp}^a}{2}.
		\end{equation}
		By H\"older's inequality, taking $\opmu_2=\tilde{\opmu}$, it implies \eqref{eq:X_commut_p} when $\theta=0$.
		Now, noticing that we have the following integration by parts like formula
		\begin{equation*}
			\Tr{\sfX_{\tilde{\opmu}}\,\opmu} = \iintd \frac{\tilde{\opmu}(x,y)\,\opmu(y,x)}{\n{x-y}^a}\d x\d y = \Tr{\tilde{\opmu}\, \sfX_{\opmu}},
		\end{equation*}
		and using the cyclicity of the trace, H\"older's inequality and Inequality~\eqref{eq:X_HS_bound_2} with $\opmu_2 = \sfX_{\tilde{\opmu}}\,\opmu\opmu^*$, we get
		\begin{multline*}
			\Nrm{\sfX_{\tilde{\opmu}}\,\opmu}{2}^2 = \Tr{\sfX_{\tilde{\opmu}^*}\,\sfX_{\tilde{\opmu}}\,\opmu\opmu^*} = \Tr{\tilde{\opmu}^*\,\sfX_{(\sfX_{\tilde{\opmu}}\,\opmu\opmu^*)}}
			\\
			\leq \Nrm{\tilde{\opmu}}{2} \Nrm{\sfX_{(\sfX_{\tilde{\opmu}}\,\opmu\opmu^*)}}{2} \leq C_{a}\,h^{-a}\Nrm{\tilde{\opmu}}{2} \Nrm{\sfX_{\tilde{\opmu}}\,\opmu\opmu^*\n{\opp}^a}{2}.
		\end{multline*}
		By H\"older's inequality, this leads to
		\begin{equation*}
			\Nrm{\sfX_{\tilde{\opmu}}\,\opmu}{2}^2 \leq C_{a}\,h^{-a}\Nrm{\tilde{\opmu}}{2} \Nrm{\sfX_{\tilde{\opmu}}\,\opmu}{2}\Nrm{\opmu^*\n{\opp}^a}{\infty}.
		\end{equation*}
		We deduce the result by dividing both sides by $\Nrm{\sfX_{\tilde{\opmu}}\,\opmu}{2}$ and then using the fact that for $q\geq2$, $\Nrm{\sfX_{\tilde{\opmu}}\,\opmu}{q}\leq \Nrm{\sfX_{\tilde{\opmu}}\,\opmu}{2}$.
	\end{proof}
	
	The following lemma will allow us to replace the Hilbert--Schmidt norm on the right-hand side of Inequality~\eqref{eq:X_HS_bound} by another Schatten norm with higher index at the expense of using a less sharp power on $\n{\opp}$.
	\begin{lem}\label{lem:X_p_bound}
		Let $\opmu$ be a compact operator. Then for any $\alpha > a$ and any $q\in[2,\infty]$, it holds
		\begin{equation}\label{eq:X_p_bound}
			\Nrm{\sfX_{\opmu}}{q} \leq C\, h^{-\alpha} \Nrm{\opmu\(1+\n{\opp}^{\alpha}\)}{q},
		\end{equation}
		for a constant $C$ depending only on $a$ and $\alpha$.
	\end{lem}
	
	\begin{proof}
		Take $(\varphi,\phi)\in (L^2)^2$. Then one has
		\begin{equation*}
			\Inprod{\varphi}{\sfX_{\opmu}\, \phi}_{L^2} = \iintd \frac{\opmu(x,y)\,\conj{\varphi(x)}\,\phi(y)}{\n{x-y}^a} \d x \d y = \(2\pi\)^{3-a} C_{a} \Tr{\opmu\, \varphi \(-\Delta\)^\frac{a-3}{2}\phi},
		\end{equation*}
		where $\varphi$ and $\phi$ are seen as multiplication operators and $C_{a} = \frac{\omega_a}{\omega_{3-a}}$. By the definition of $\opp$, this can be written
		\begin{equation*}
			\Inprod{\varphi}{\sfX_{\opmu}\, \phi}_{L^2} = C_{a} h^{3-a} \Tr{\opmu\, \varphi \n{\opp}^{a-3}\phi} = C_{a} h^{3-a} \Tr{m_\alpha\,\opmu\,m_\alpha\,m_\alpha^{-1} \varphi\, g(\opp)\,\phi\,m_\alpha^{-1}},
		\end{equation*}
		with $g(x) = \n{x}^{a-3}$ and $m_\alpha=1+\n{\opp}^{\alpha}$. Now taking $1 \leq \frac{3}{\alpha} < p_0' < \frac{3}{a} < p_1' \leq \infty$ such that $\frac{1}{p_0'}+\frac{1}{p_1'} = \frac{a}{3}$, we have $g \in L^{p_0} + L^{p_1}$, hence we can write $g = g_0+g_1$ with $(g_0,g_1)\in L^{p_0}\times L^{p_1}$. Let $\tilde{g} = g_0$ or $\tilde{g} = g_1$, or more generally, take $\tilde{g} \in L^p$ for some $p\geq 1$ satisfying $p'>\frac{3}{\alpha}$. Then, by H\"older's inequality for Schatten norms, Lemma~\ref{lem:rearrangement} and the Kato--Seiler--Simon Inequality~\eqref{eq:Kato_Seiler_Simon}, we have
		\begin{align*}
			h^{3} &\n{\Tr{m_\alpha^\frac{1}{2}\,\opmu\,m_\alpha^\frac{1}{2}\,m_\alpha^{-\frac{1}{2}} \varphi\, \tilde{g}(\opp)\,\phi\,m_\alpha^{-\frac{1}{2}}}}
			\\
			&\leq \Nrm{m_\alpha^\frac{1}{2}\,\opmu\,m_\alpha^\frac{1}{2}}{\infty} \Nrm{m_\alpha^{-\frac{1}{2}}\varphi^\frac{1}{p'}}{\L^{2p'}} \Nrm{\varphi^\frac{1}{p} \tilde{g}(\opp)^\frac{1}{2}}{\L^{2p}} \Nrm{\tilde{g}(\opp)^\frac{1}{2}\phi^\frac{1}{p}}{\L^{2p}} \Nrm{\phi^\frac{1}{p'} m_\alpha^{-\frac{1}{2}}}{\L^{2p'}}
			\\
			&\leq C_{p}^\frac{1}{p'} \Nrm{\opmu\,m_\alpha}{\infty} \Nrm{\varphi}{L^2} \Nrm{\tilde{g}}{L^p} \Nrm{\phi}{L^2},
		\end{align*}
		where we used the notation $z^b = \n{z}^{b-1} z$ and with $C_{p} = \intd \frac{\d y}{\(1+\n{y}^\alpha\)^{p'}}$. This constant is finite since $\alpha p' > 3$. This proves Inequality~\eqref{eq:X_p_bound} when $q=\infty$. When $q=2$, the inequality follows from Formula~\eqref{eq:X_HS_bound_2}. The other cases follow by complex interpolation.
	\end{proof}

\subsection{Commutators Involving the Exchange Term}

	\begin{prop}\label{prop:exchange_moments}
		Let $a\in[0,1]$. Then there exists $C>0$ such that for any compact self-adjoint operators $\op$ and $\opmu$, any $q\in[1,\infty]$ and any integer $n\geq 2a-1$
		\begin{equation*}
			\frac{1}{\hbar} \Nrm{\com{h^3\sfX_{\op},\opp_\jj^n}\opmu\,}{\L^q} \leq 3^n\,h^{\frac{3}{2}-a}\,C \Nrm{\Dhxj{\op}\, m_n}{\L^2} \Nrm{\opmu\, m_n}{\L^q},
		\end{equation*}
		where $m_n = 1+\n{\opp}^n$.
	\end{prop}
	
	\begin{proof}
		By Formula~\eqref{eq:commut_X_1}, and the triangle inequality, we have
		\begin{equation*}
			\Nrm{\com{\sfX_{\op},\opp_\jj^n}\opmu\,}{q} \leq \sum_{k=1}^n \binom{n}{k} \Nrm{\sfX_{\ad{\opp_\jj}^k\!(\op)}\,\opp_\jj^{n-k}\,\opmu}{q}.
		\end{equation*}	
		Now, Lemma~\ref{lem:X_commut_p} gives the bound
		\begin{equation*}
			\Nrm{\sfX_{\ad{\opp_\jj}^k\!(\op)}\,\opp_\jj^{n-k}\,\opmu}{q} \leq C_{a}\,h^{-a}\Nrm{\ad{\opp_\jj}^k\!(\op)\n{\opp}^{a\(1-\theta\)}}{2} \Nrm{\opmu\, \opp_\jj^{n-k}\n{\opp}^{\theta a}}{\infty}.
		\end{equation*}
		Using the fact that $\ad{\opp_\jj}\!(\op) = -i \hbar \Dhxj{\op}$ and expanding the $k-1$ commutators in $\ad{\opp_\jj}^{k-1}$ by Lemma~\ref{lem:expand_commutators}, we get
		\begin{equation*}
			\Nrm{\ad{\opp_\jj}^k\!(\op)\n{\opp}^{a\(1-\theta\)}}{2} \leq 2^k\hbar \Nrm{\Dhxj{\op} \n{\opp}^{a\(1-\theta\)+k-1}}{2}.
		\end{equation*}
		Now when $k\geq a$, we take $\theta = 1$ so that $n-k+\theta a \leq n$ and $a\(1-\theta\)+k-1 = k-1 \leq n$. When $k < a$, we take $\theta = 0$ so that $n-k+\theta a = n-k \leq n$ and $a\(1-\theta\)+k-1 \leq 2a-1 \leq n$. In all the cases, this leads to
		\begin{equation*}
			\frac{1}{\hbar} \Nrm{\com{\sfX_{\op},\opp_\jj^n}\opmu\,}{q} \leq C\,h^{-a}\sum_{k=1}^n \binom{n}{k} 2^k \Nrm{\Dhxj{\op} \,m_n}{2} \Nrm{\opmu\, m_n}{\infty}.
		\end{equation*}
		We conclude using the fact that $\Nrm{\opmu\, m_n}{\infty} \leq \Nrm{\opmu\, m_n}{q}$ and the definition \eqref{eq:def_norm} of the $\L^2$ norm.
	\end{proof}
	
	\begin{prop}\label{prop:exchange_com}
		Let $a\in[0,1]$, $\fb = \frac{3}{a+1}$ and $n\in\N$ satisfying $n\geq 2a$. Then for any $\alpha\in(a,n-a]$ and any $q\in[2,\infty]$
		\begin{align}\label{eq:exchange_com_1}
			\frac{1}{\hbar} \Nrm{\com{h^3\sfX_{\opmu},\op}\opp_\jj^n}{\L^q} &\leq 3^n\,C\, h^{3\(\frac{1}{q}+\frac{1}{2}-\frac{1}{\fb}\)}\Nrm{\op \,m_n}{\L^\infty} \Nrm{\opmu\, m_n}{\L^{2}}
			\\\label{eq:exchange_com_2}
			\frac{1}{\hbar} \Nrm{\com{h^3\sfX_{\opmu},\op}\opp_\jj^n}{\L^q} &\leq 3^n\,C \Nrm{\op\,m_n}{\L^\infty} \(h^{\frac{3}{\beta'}}\Nrm{\opmu\,m_n}{\L^q} + h^{\frac{3}{2}-a} \Nrm{\Dhxj{\opmu}\,m_n}{\L^2}\),
		\end{align}
		where $m_n = 1+\n{\opp}^n$ and $\beta = \frac{3}{\alpha+1}$.
	\end{prop}
	
	Note that the power of $h$ in the first formula is nonnegative only for $q\leq q_a$ with $\frac{1}{q_a} = \frac{1}{\fb} - \frac{1}{2}$. In the second formula, this is true for every $q$ but involves a semiclassical derivative of $\opmu$.
	
	\begin{proof}[Proof of Proposition~\ref{prop:exchange_com}]
		Since the exchange term is vanishing when $\hbar\to 0$, we can estimate the two parts of the commutator separately by writing
		\begin{equation*}
			\Nrm{\com{\sfX_{\opmu},\op}\opp_\jj^n}{q} = \Nrm{\opp_\jj^n\com{\sfX_{\opmu},\op}}{q} \leq \Nrm{\opp_\jj^n\,\op\,\sfX_{\opmu}}{q} + \Nrm{\opp_\jj^n\,\sfX_{\opmu}\,\op}{q}.
		\end{equation*}
		The first term in the right-hand side can be bounded using H\"older's inequality for Schatten norms and Lemma~\ref{lem:X_commut_p} with $\theta = 0$, leading to
		\begin{equation}\label{eq:exchange_com_3}
			\Nrm{\opp_\jj^n\,\op\,\sfX_{\opmu}}{q} \leq C\,h^{-a}\Nrm{\opp_\jj^n\,\op}{\infty} \Nrm{\opmu \n{\opp}^a}{2} \leq C\, h^{-a} \Nrm{\op\,m_n}{\infty} \Nrm{\opmu\,m_n}{2}.
		\end{equation}
		We can also use Lemma~\ref{lem:X_p_bound} with $\alpha \in (a,n]$, to get
		\begin{equation}\label{eq:exchange_com_3_bis}
			\Nrm{\opp_\jj^n\,\op\,\sfX_{\opmu}}{q} \leq C\,h^{-\alpha} \Nrm{\opp_\jj^n\,\op}{\infty} \Nrm{\opmu\(1+\n{\opp}^\alpha\)}{q} \leq C\, h^{-\alpha} \Nrm{\op\,m_n}{\infty} \Nrm{\opmu\,m_n}{q}.
		\end{equation}
		
		To treat the second term, we want to put the first weight $m_n$ either on $\opmu$ either on $\op$. To obtain this effect, we use Formula~\eqref{eq:commut_X_0} to get
		\begin{equation}\label{eq:expansion_pn_X}
			\Nrm{\opp_\jj^n\,\sfX_{\opmu}\,\op}{q} \leq \sum_{k=0}^n \binom{n}{k} \Nrm{\sfX_{\ad{\opp_\jj}^k\!(\opmu)}\,\opp_\jj^{n-k}\,\op\,}{q}.
		\end{equation}
		Now we use Lemma~\ref{lem:X_commut_p} and then expand the commutators by Lemma~\ref{lem:expand_commutators} to get for any $\theta\in\{0,1\}$
		\begin{align*}
			\Nrm{\sfX_{\ad{\opp_\jj}^k\!(\opmu)}\,\opp_\jj^{n-k}\,\op\,}{q} &\leq C\,h^{-a}\Nrm{\ad{\opp_\jj}^k\!(\opmu)\n{\opp}^{a\(1-\theta\)}}{2} \Nrm{\op\,\opp_\jj^{n-k}\n{\opp}^{\theta a}}{\infty}
			\\
			&\leq 2^k\, C\, h^{-a} \Nrm{\opmu \(1+\n{\opp}^{k+a\(1-\theta\)}\)}{2} \Nrm{\op \(1+\n{\opp}^{n-k+\theta a}\)}{\infty},
		\end{align*}
		and similarly as in the proof of Proposition~\ref{prop:exchange_moments}, if $k\geq a$, we take $\theta = 1$ and if $k\leq a$, we take $\theta = 0$ and use the fact that $2a\leq n$. In any cases, the power on $\n{\opp}$ is smaller than $n$. Therefore, recalling Inequality~\eqref{eq:expansion_pn_X}, we obtain
		\begin{equation}\label{eq:exchange_com_4}
			\Nrm{\opp_\jj^n\,\sfX_{\opmu}\,\op}{q} \leq 3^n\,C\,h^{-a} \Nrm{\op\,m_n}{\infty} \Nrm{\opmu\,m_n}{2}.
		\end{equation}
		Combining inequalities~\eqref{eq:exchange_com_3} and \eqref{eq:exchange_com_4} and using the definition~\eqref{eq:def_norm} of $\L^q$ norms yield Formula~\eqref{eq:exchange_com_1}. 
		
		To get Formula~\eqref{eq:exchange_com_2}, we start back from Inequality~\eqref{eq:expansion_pn_X}. If $k> a$, so that in particular $k\geq 1$, we use again Lemma~\ref{lem:X_commut_p} and Lemma~\ref{lem:expand_commutators} but we use first the fact that $\ad{\opp_j}\!(\opmu) = -i\hbar\,\Dhxj{\opmu}$ to get an additional $\hbar$. This yields
		\begin{equation*}
			\Nrm{\sfX_{\ad{\opp_\jj}^k\!(\opmu)}\,\opp_\jj^{n-k}\,\op\,}{q} \leq 2^k\, C\, h^{1-a} \Nrm{\Dhxj{\opmu} \(1+\n{\opp}^{k-1}\)}{2} \Nrm{\op \(1+\n{\opp}^{n-k+ a}\)}{\infty},
		\end{equation*}
		If $k\leq a$, then we use Lemma~\ref{lem:X_p_bound} with $\alpha\in(a,n-a]$ to get
		\begin{equation*}
			\Nrm{\sfX_{\ad{\opp_\jj}^k\!(\opmu)}\,\opp_\jj^{n-k}\,\op\,}{q} \leq 2^k\, C\, h^{-\alpha} \Nrm{\opmu \(1+\n{\opp}^{k+\alpha}\)}{q} \Nrm{\op \(1+\n{\opp}^{n-k}\)}{\infty}.
		\end{equation*}
		Therefore, Inequality~\eqref{eq:expansion_pn_X} implies
		\begin{equation}\label{eq:exchange_com_5}
			\Nrm{\opp_\jj^n\,\sfX_{\opmu}\,\op}{q} \leq 3^n\,C \Nrm{\op\,m_n}{\infty} \(h^{-\alpha}\Nrm{\opmu\,m_n}{q} + h^{1-a} \Nrm{\Dhxj{\opmu}\,m_n}{2}\),
		\end{equation}
		and together with Inequality~\eqref{eq:exchange_com_3} and the definition of the $\L^q$ norm, this implies Formula~\eqref{eq:exchange_com_2}.
	\end{proof}
	
\subsection{Proof of the Propagation of Regularity}
	
	\begin{proof}[Proof of Proposition~\ref{prop:regu_HF}]
		The strategy to prove Proposition~\ref{prop:regu_HF} is to look at Equations~\eqref{eq:rho}, \eqref{eq:Dhx} and \eqref{eq:Dhv} and find a Gr\"onwall-type inequality on $\Nrm{\op}{\cW^{1,q}(m_{2n})}$ (observe that we renamed $n$ by $2n$ as we need the number of moments to be even). In particular we will see that to close the Gr\"onwall argument for $q=2$, we need to estimate $\Nrm{\op}{\cW^{1,q}(m_{2n})}$ for $q\in\set{2,4}$. We will therefore proceed by interpolation and define
		\begin{equation*}
			M_2(t):= \Nrm{\op}{\cW^{1,2}(m_{2n})},
			\quad
			M_4(t):=\Nrm{\op}{\cW^{1,4}(m_{2n-2})},
			\quad
			M_\infty(t) :=\Nrm{\op\,m_{2n}}{\L^\infty}.
		\end{equation*}
		For $a\in \lt[\frac{1}{2},1\rt]$ we will find a Gr\"onwall-type inequality on $M_2(t)+M_4(t)+M_\infty(t)$, whereas for $a\in \lt[0,\frac{1}{2}\rt)$ it suffices to apply Gr\"onwall's Lemma to $M_2(t)+M_4(t)$.
			 
		We now look at Equation~\eqref{eq:rho}. Splitting the interaction $K$ as in \eqref{eq:splitting-K}, by Proposition~\ref{prop:weighted_com_est_long_range} and Proposition~\ref{prop:weighted_com_est} we have that, for $1/r+1/r_1=1/q+1/3$, 
		\begin{align*}
			\frac{1}{\hbar}\Nrm{\com{V_{\op},m_{2n}}\op}{\L^q} &\leq C \(\Nrm{\op\,m_{n+n_0}}{\L^{\fb'}}\Nrm{\op\,m_{2n-1}}{\L^{q}}+\Nrm{\op\,m_{2n+k'}}{\L^{r}}\Nrm{\op\,m_{n+k}}{\L^{r_1}}\)\\
			&\qquad + C \(\Nrm{\rho}{L^1}+\hbar\Nrm{\op\,m_{2n}}{\L^2}\)\Nrm{\op\,m_{2n}}{\L^q}
		\end{align*}
		with
		$$n_0>\frac{3}{\fb}-1,\quad\quad k'>\frac{3}{r'}-2,\quad\quad k>\frac{3}{r}-1.$$
		The contribution given by the exchange term in the right-hand side of \eqref{eq:rho} can be bounded by Proposition~\ref{prop:exchange_moments} with $\opmu=\op$. Therefore, we obtain the following bounds on the right-hand side of Equation~\eqref{eq:rho}
		\begin{align*}
			\dt\Nrm{\op\,m_{2n}}{\L^q} &\leq C \, \big(\Nrm{\op\,m_{n+n_0}}{\L^{\fb'}} \Nrm{\op\,m_{2n-1}}{\L^{q}} + \Nrm{\op\,m_{2n+k'}}{\L^{r}} \Nrm{\op\,m_{n+k}}{\L^{r_1}}
			\\
			&\qquad +\Nrm{\rho}{L^1}\Nrm{\op\,m_{2n}}{\L^q}+\hbar\Nrm{\op\,m_{2n}}{\L^2}\Nrm{\op\,m_{2n}}{\L^q}
			\\
			&\qquad + h^{\frac{3}{2}-a}\Nrm{\Dhx\op\,m_{2n}}{\L^2}\Nrm{\op\,m_{2n}}{\L^q}\big).
		\end{align*}
		In particular, for $q=\infty$ we get
		\begin{align*}
			\dt\Nrm{\op\,m_{2n}}{\L^\infty} &\leq C \, \big(\Nrm{\op\,m_{n+n_0}}{\L^{\fb'}}\Nrm{\op\,m_{2n-1}}{\L^{\infty}}+\Nrm{\op\,m_{2n+k'}}{\L^{r}}\Nrm{\op\,m_{n+k}}{\L^{r_1}}
			\\
			&\qquad +\Nrm{\rho}{L^1}\Nrm{\op\,m_{2n}}{\L^\infty}+\hbar\Nrm{\op\,m_{2n}}{\L^2}\Nrm{\op\,m_{2n}}{\L^\infty}
			\\
			&\qquad+ h^{\frac{3}{2}-a}\Nrm{\Dhx\op\,m_{2n}}{\L^2}\Nrm{\op\,m_{2n}}{\L^\infty}\big).
		\end{align*}	
		Note that in order to close the Gr\"onwall inequality and we will need bounds on $\Dhx\op$ and $\Dhv\op$. To this end, we look at Equation~\eqref{eq:Dhx} and Equation~\eqref{eq:Dhv}. We start bounding the right-hand side of Equation~\eqref{eq:Dhx}. By Proposition~\ref{prop:weighted_com_est} and Proposition~\ref{prop:weighted_com_est_long_range} we obtain 
		\begin{equation}\label{eq:com_V_Dhx}
		\begin{split}
			\frac{1}{\hbar}&\Nrm{\com{V_{\op},m_{2n}} \Dhx\op}{\L^q}
			\\
			&\quad \leq C \,\big(\Nrm{\op\,m_{n+n_0}}{\L^{\fb'}} \Nrm{\Dhx\op\,m_{2n-1}}{\L^{q}} + \Nrm{\op\,m_{2n+k'}}{\L^{r}} \Nrm{\Dhx\op\,m_{n+k}}{\L^{r_1}}
			\\
			&\qquad\quad +\Nrm{\rho}{L^1}\Nrm{\Dhx\op\,m_{2n}}{\L^q}+\hbar\Nrm{\op\,m_{2n}}{\L^2}\Nrm{\Dhx\op\,m_{2n}}{\L^q} \big)
			\end{split}
		\end{equation}
		with the usual constraints on $r,r_1,n_0,k, k'$.\\
		By writing $\com{E_{\op},\op} m_{2n} = \com{E_{\op},\op\,m_{2n}} + \op\,\com{E_{\op},m_{2n}}$, applying Proposition~\ref{prop:estim_commutator_Lq_Besov} with $\op_2=\op\,m_{2n}$, and Proposition~\ref{prop:weighted_com_est} and Proposition~\ref{prop:weighted_com_est_long_range} with $\opmu=\op$, we get
		\begin{equation}\label{eq:E_com_est}
		\begin{split}
			\dfrac{1}{\hbar}&\Nrm{\com{E_{\op},\op} m_{2n}}{\L^q} \leq C \,\big(\Nrm{\rho}{L^r}^{1-s}\Nrm{\rho}{W^{1,r}}^s \Nrm{\Dhv\op\,m_{2n}}{\L^q}
			\\
			&\qquad + \Nrm{\Dhx\op\,m_{n+n_0}}{\L^{\fb'}} \Nrm{\op\,m_{2n-1}}{\L^q} + \Nrm{\Dhx\op\,m_{2n+k'}}{\L^r} \Nrm{\op\,m_{n+k}}{\L^{r_1}}
			\\
			&\qquad + \Nrm{\op\,m_{2n}}{\L^2}\Nrm{\op\,m_{2n}}{\L^q} + \hbar\Nrm{\Dhx\op\,m_{2n}}{\L^2}\Nrm{\op\,m_{2n}}{\L^q}\big)
		\end{split}
		\end{equation}
		for $q > 2$ and $s=1-3\(\frac{1}{r'}-\frac{1}{\fb'}\)$, where we used the interpolation of Besov spaces stated in Corollary~\ref{cor:estim_commutator_Lq}.
		
		For $q = 2$, we have
		\begin{equation}\label{eq:E_com_est>}
		\begin{split}
			\dfrac{1}{\hbar}\Nrm{\com{E_{\op},\op}m_{2n}}{\L^{2}} &\leq C \,\big(\Nrm{\nabla\rho}{L^{\fb',1}} \Nrm{\Dhv\op\,m_{2n}}{\L^{2}}\\
			&\qquad + \Nrm{\Dhx\op\,m_{2n}}{\L^2}^{1-\theta} \Nrm{\Dhx\op\,m_{2n-2}}{\L^{4}}^\theta \Nrm{\op\,m_{2n}}{\L^2}
			\\
			&\qquad +\Nrm{\Dhx\op\,m_{2n}}{\L^2}\Nrm{\op\,m_{2n}}{\L^2}^{1/3}\Nrm{\op\,m_{2n-2}}{\L^4}^{2/3}
			\\
			&\qquad +{\Nrm{\op\,m_{2n}}{\L^2}^{2}}+\hbar\Nrm{\Dhx\op\,m_{2n}}{\L^2}\Nrm{\op\,m_{2n}}{\L^2}\big)
			\end{split}
		\end{equation}
		for $a\in\lt[\frac{1}{2},1\rt]$, and
		\begin{equation}\label{eq:E_com_est<}
			\begin{split}
			\dfrac{1}{\hbar}\Nrm{\com{E_{\op},\op}m_{2n}}{\L^{2}} &\leq C \,\big(\Nrm{\rho}{L^{\frac{3}{1-a},1}} \Nrm{\Dhv\op\,m_{2n}}{\L^{2}}
			\\
			&\qquad+ \Nrm{\Dhx\op\,m_{2n}}{\L^2}^{1-\theta} \Nrm{\Dhx\op\,m_{2n-2}}{\L^{4}}^\theta\Nrm{\op\,m_{2n}}{\L^2}
			\\
			&\qquad+ \Nrm{\Dhx\op\,m_{2n}}{\L^2}\Nrm{\op\,m_{2n}}{\L^2}^{1/3}\Nrm{\op\,m_{2n-2}}{\L^4}^{2/3}
			\\
			&\qquad+ \Nrm{\op\,m_{2n}}{\L^2}^{2} + \hbar\Nrm{\Dhx\op\,m_{2n}}{\L^2}\Nrm{\op\,m_{2n}}{\L^2}\big)
			\end{split}
		\end{equation}
		for $a\in\(0,\frac{1}{2}\)$.
				
		The contributions of the exchange term can be bounded using Proposition~\ref{prop:exchange_moments} and Proposition~\ref{prop:exchange_com}. Combining them with \eqref{eq:com_V_Dhx} and \eqref{eq:E_com_est} leads to, for $q > 2$,
		\begin{align*}
			\dt\Nrm{\Dhx\op\,m_{2n}}{\L^q} &\leq C \,\big(\Nrm{\op\,m_{n+n_0}}{\L^{\fb'}}\Nrm{\Dhx\op\,m_{2n-1}}{\L^{q}}+\Nrm{\op\,m_{2n+k'}}{\L^{r}}\Nrm{\Dhx\op\,m_{n+k}}{\L^{r_1}}
			\\
			&\qquad+ \Nrm{\rho}{L^1}\Nrm{\Dhx\op\,m_{2n}}{\L^q}+\hbar\Nrm{\op\,m_{2n}}{\L^2}\Nrm{\Dhx\op\,m_{2n}}{\L^q}\big)
			\\
			&+ C \,\big(\Nrm{\rho}{L^r}^{1-s}\Nrm{\rho}{W^{1,r}}^s \Nrm{\Dhv\op\,m_{2n}}{\L^q} + \Nrm{\Dhx\op\,m_{n+n_0}}{\L^{\fb'}} \Nrm{\op\,m_{2n-1}}{\L^q}
			\\
			&\qquad+ \Nrm{\Dhx\op\,m_{2n+k'}}{\L^r}\Nrm{\op\,m_{n+k}}{\L^{r_1}}+\Nrm{\op\,m_{2n}}{\L^2} \Nrm{\op\,m_{2n}}{\L^q}
			\\
			&\qquad+ \hbar\Nrm{\Dhx\op\,m_{2n}}{\L^2} \Nrm{\op\,m_{2n}}{\L^q}\big)
			\\
			&+ C\,h^{\frac{3}{2}-a}\Nrm{\Dhx\op\,m_{2n}}{\L^2} \Nrm{\Dhx\op\,m_{2n}}{\L^q}
			\\
			&+ C\,h^{3\(\frac{1}{q}+\frac{1}{2}-\frac{1}{\fb}\)} \Nrm{\op\,m_{2n}}{\L^\infty} \Nrm{\Dhx\op\,m_{2n}}{\L^{2}}.
		\end{align*}
		To bound the right-hand side of Equation~\eqref{eq:Dhv}, we use Proposition~\ref{prop:weighted_com_est} for the contribution due to the direct term and Proposition~\ref{prop:exchange_moments} and Proposition~\ref{prop:exchange_com} to estimate the contributions of the exchange term. Hence,
		\begin{align*}
			\dt\Nrm{\Dhv\op\,m_{2n}}{\L^q} &\leq  C \,\big(\Nrm{\op\,m_{n+n_0}}{\L^{\fb'}}\Nrm{\Dhv\op\,m_{2n-1}}{\L^{q}}+\Nrm{\op\,m_{2n+k'}}{\L^{r}}\Nrm{\Dhv\op\,m_{n+k}}{\L^{r_1}}
			\\
			&\qquad + \Nrm{\rho}{L^1} \Nrm{\Dhv\op\,m_{2n}}{\L^q} + \hbar \Nrm{\op\,m_{2n}}{\L^2} \Nrm{\Dhv\op\,m_{2n}}{\L^q}\big)
			\\
			&\quad+\Nrm{\Dhx\op\,m_{2n}}{\L^q}
			\\
			&\quad+ C\,h^{\frac{3}{2}-a} \Nrm{\Dhx\op\,m_{2n}}{\L^2} \Nrm{\Dhv\op\,m_{2n}}{\L^q}
			\\
			&\quad+ C\,h^{3\(\frac{1}{q}+\frac{1}{2}-\frac{1}{\fb}\)} \Nrm{\op\,m_{2n}}{\L^\infty} \Nrm{\Dhv\op\,m_{2n}}{\L^{2}}.
		\end{align*}
		To get an estimate in $\L^2$ we need a bound on the $\L^q$ norm for $q\in(2,4)$. Therefore we look for a bound when $q=4$ using Corollary~\ref{cor:direct-term_L2-L4} and proceed by interpolation.
				
		To establish a Gr\"onwall type inequality for $a\geq\frac{1}{2}$, we consider the sum $M_2(t)+M_4(t)+M_\infty(t)$ and observe that it satisfies
		\begin{multline}\label{eq:M2-M4-gronwall}
			\dt\(M_2(t)+M_4(t)+M_\infty(t)\) \leq C \(M_2(t) + M_4(t) + M_\infty(t)\)
			\\
			+ C \(1 + h^{\frac{3}{2}-a} + h^{\frac{3}{\fb'}} + h^{3\(\frac{3}{4}-\frac{1}{\fb}\)}\) \(M_2(t)+M_4(t)+M_{\infty}(t)\)^2,
		\end{multline} 
		where we used interpolation inequality with $\theta\in(0,1)$ and Young's inequality for products to bound the norms $\L^{r},\,\L^{r'},\,L^{\mathfrak{b}'}$ with $r,\,r',\,\mathfrak{b}'\in[2,4]$. Furthermore, we used the following simple inequality: for an operator $\opmu$, $k\in(0,2n)$ and $q\geq 2$, $\Nrm{\opmu\,m_{2n-k}}{\L^q}\leq \Nrm{\opmu\,m_{2n}}{\L^q}\Nrm{m_{k}^{-1}}{\L^\infty}$.
		
		We observe that Equation~\eqref{eq:M2-M4-gronwall} is a Gr\"onwall-type inequality of the same form as Equation~\eqref{eq:U-gronwall}. Thus there exists a time $T>0$, depending only in the initial data, such that $M_2(t)+M_4(t)+M_{\infty}(t)$ is bounded for all $t\in[0,T]$.
		
		For $a<\frac{1}{2}$, we consider the quantity
		$M_2(t)+M_4(t)$ and use that 
		\begin{equation*}
			\Nrm{\op\,m_{2n}}{\L^\infty} \leq C h^{-\frac{3}{q}}\Nrm{\op\,m_{2n}}{\L^q}.
		\end{equation*}
		Hence
		\begin{multline*}
			\dt\(M_2(t)+M_4(t)\)
			\\
			\leq C\(M_2(t)+M_4(t)\) + C \(1+h^{\frac{3}{2}-a}+h^{3\(\frac{1}{\fb'}-\frac{1}{2}\)}\) \(M_2(t)+M_4(t)\)^2.
		\end{multline*}
		Therefore there exists $T>0$, depending only in the initial data, such that $M_2(t)+M_4(t)$ is bounded for all $t\in[0,T]$, thus $\op\in L^{\infty}((0,T),\cW^{1,2}(m_{2n})\cap\cW^{1,4}(m_{2n})\cap\L^\infty(m_{2n}))$. Moreover, ${\rho} \in L^\infty((0,T),H^{1} \cap W^{1,4} \cap L^1 \cap L^\infty)$ thanks to Proposition~\ref{prop:diag_vs_weights} and the bounds on $M_2(t),\,M_4(t)$ and $M_\infty(t)$.
	\end{proof}
		
	\begin{proof}[Proof of Proposition~\ref{prop:regu_HF_2}]
		Similarly to what we have done for the first-order quantum gradients, we can compute the time derivative of the second order quantum gradients of~$\op$
		\begin{equation}\label{eq:grad2-moments}
			\begin{split}
			i\hbar\,\partial_t\,\Dhx^2\op &= \com{H,\Dhx^2\op} - 2\com{E_{\op},\Dhx\op} - \com{\Dhx E_{\op},\op} - 2\com{h^3\sfX_{\Dhx\op},\Dhx\op} - \com{h^3\sfX_{\Dhx^2\op},\op}
			\\
			i\hbar\,\partial_t\,\Dhv^2\op &= \com{H,\Dhv^2\op} - i\hbar\Dhv\Dhx\op - 2\com{h^3\sfX_{\Dhv\op},\Dhv\op} - \com{h^3\sfX_{\Dhv^2\op},\op}
			\\
			i\hbar\,\partial_t\,\Dhv\Dhx\op &= \com{H,\Dhv\Dhx\op} - i\hbar\Dhx^2\op - \com{E_{\op},\Dhv\op} - \com{h^3\sfX_{\Dhv\Dhx\op},\op} - \com{h^3\sfX_{\Dhx\op},\Dhv\op},
			\end{split}
		\end{equation}	
		that are of the form
		\begin{equation}\label{eq:abstract_equation_2}
			i\hbar\, \dpt \opmu = \com{\sfA,\opmu} + \com{\sfB,\Dhx\op} + \com{\sfC,\op},
		\end{equation}
		with $\sfA,\, \sfB$ and $\sfC$ being self-adjoint operators. The proof of Lemma \ref{lem:abstract_lemma} proves also the following statement.
		\begin{lem}[{\bf Lemma \ref{lem:abstract_lemma}{ bis}}]\label{lem:abstract_lemma_bis}
			Let $\op,\, \sfA,\,\sfB,\,\sfC$ be self-adjoint operators and $\opmu = \opmu(t)$ be a family of self-adjoint operators satisfying \eqref{eq:abstract_equation_2}. Then, formally, for any even integer $q\geq 2$ we have
			\begin{equation*}
				\dt \Nrm{\opmu\, m_{2n}}{q} \leq \frac{1}{\hbar}\Nrm{\com{\sfA,m_{2n}}\opmu}{q} + \frac{1}{\hbar} \Nrm{\com{\sfB,\Dhx\op}m_{2n}}{q}+\frac{1}{\hbar} \Nrm{\com{\sfC,\op}m_{2n}}{q}.
			\end{equation*}
		\end{lem}
		\smallskip
	
		We consider the Identities~\eqref{eq:grad2-moments} and bound them by Lemma~\ref{lem:abstract_lemma_bis}. This yields
		\begin{equation}\label{eq:grad_xx}
		\begin{split}
			\hbar\,\dt\Nrm{\Dhx^2\op\,m_{2n}}{q} &\leq C \Nrm{\com{V_{\op},m_{2n}} \Dhx^2\op}{q}
			+C \Nrm{\com{E_{\op},\Dhx\op} m_{2n}}{q}\\
			& +C \Nrm{\com{\Dhx E_{\op},\op} m_{2n}}{q}
			+C \Nrm{\com{h^3\sfX_{\op},m_{2n}} \Dhx^2\op}{q}\\
			& + C\Nrm{\com{h^3\sfX_{\Dhx\op},\Dhx\op} m_{2n}}{q} + C\Nrm{\com{h^3\sfX_{\Dhx^2\op},\op} m_{2n}}{q}
			\end{split}
		\end{equation}
		\begin{equation}\label{eq:grad_vv}
		\begin{split}
			\hbar\,\dt\Nrm{\Dhv^2\op\,m_{2n}}{q} &\leq C \Nrm{\com{V_{\op},m_{2n}} \Dhv^2\op}{q} + C\Nrm{\Dhv\Dhx\op\,m_{2n}}{q}
			\\
			&+ C\Nrm{\com{h^3\sfX_{\op},m_{2n}}\Dhv^2\op}{q} + C\Nrm{\com{h^3\sfX_{\Dhv\op},\Dhv\op} m_{2n}}{q} \\ 
			&+ C\Nrm{\com{h^3\sfX_{\Dhv^2\op},\op} m_{2n}}{q}
		\end{split}
		\end{equation}
		\begin{equation}\label{eq:grad_vx}
		\begin{split}
			\hbar\,\dt&\Nrm{\Dhv\Dhx\op\,m_{2n}}{q} \leq C \Nrm{\com{V_{\op},m_{2n}} \Dhv\Dhx\op}{q} + C \Nrm{\com{E_{\op},\Dhv\op} m_{2n}}{q}
			\\
			&\qqqquad+ C \Nrm{\Dhx^2 \op\,m_{2n}}{q} + C \Nrm{\com{h^3\sfX_{\op},m_{2n}} \Dhv\Dhx\op}{q}
			\\
			&\qqqquad+ C \Nrm{\com{h^3\sfX_{\Dhx\op},\Dhv\op}m_{2n}}{q} + C \Nrm{\com{h^3\sfX_{\Dhv\Dhx\op},\op} m_{2n}}{q}.
		\end{split}
		\end{equation}
		We now estimate the right-hand side of Equation~\eqref{eq:grad_xx}. The first three contributions are related to the direct term in the Hartree equation, whereas in the others the exchange operator appears. By Proposition~\ref{prop:weighted_com_est} and Proposition~\ref{prop:weighted_com_est_long_range} we have
		\begin{equation}\label{eq:com_V_est}
		\begin{split}
			\frac{1}{\hbar} &\Nrm{\com{V_{\op},m_{2n}} \Dhx^2\op}{\L^q} \leq C \Nrm{\op\,m_{n+n_0}}{{\L}^{\mathfrak{b}'}}\Nrm{\Dhx^2\op\,m_{2n-1}}{{\L}^{q}}
			\\
			&\qquad + C \Nrm{\op\,m_{2n+k'}}{\L^r} \Nrm{\Dhx^2\op\,m_{n+k}}{\L^{r_1}} + C \Nrm{\rho}{L^1} \Nrm{\Dhx^2\op\,m_{2n}}{\L^q}
			\\
			&\qquad +C\,\hbar \Nrm{\op\,m_{2n}}{\L^2} \Nrm{\Dhx^2\op\,m_{2n}}{\L^q}
			\end{split}
		\end{equation}
		 As for the second term on the right-hand side of Equation~\eqref{eq:grad_xx}, we rewrite it as follows:
		\begin{equation*}
			\frac{1}{\hbar} \Nrm{\com{E_{\op},\Dhx\op} m_{2n}}{\L^q} = \frac{1}{\hbar} \Nrm{\com{E_{\op},\Dhx\op\, m_{2n}}}{\L^q} + \frac{1}{\hbar} \Nrm{\com{E_{\op},m_{2n}} \Dhx\op}{\L^q}.
		\end{equation*}
		By Proposition~\ref{prop:estim_commutator_Lq_Besov} and Corollary~\ref{cor:estim_commutator_Lq}, we get
		\begin{equation}\label{eq:estimate_grad_xx_m}
			\frac{1}{\hbar} \Nrm{\com{E_{\op},\Dhx\op\,m_{2n}}}{\L^q} \leq C \Nrm{\rho}{L^r}^{1-s} \Nrm{\rho}{W^{1,r}}^s \Nrm{\Dhv\Dhx\op\,m_{2n}}{\L^q},
		\end{equation}
		for $\frac{1}{r}+\frac{1}{q}=\frac{1}{2}$ and $s=1-3\(\frac{1}{r'}-\frac{1}{\fb}\)$. By Proposition~\ref{prop:weighted_com_est} and Proposition~\ref{prop:weighted_com_est_long_range} we have
		\begin{align*}
			\frac{1}{\hbar} &\Nrm{\com{E_{\op},m_{2n}} \Dhx\op}{\L^q} \leq C \Nrm{\Dhx\op\,m_{n+n_0}}{\L^{\fb'}}\Nrm{\Dhx\op\,m_{2n-1}}{\L^{q}}
			\\
			&\qquad +\Nrm{\Dhx\op\,m_{2n+k'}}{\L^r}\Nrm{\Dhx\op\,m_{n+k}}{\L^{r_1}}+\Nrm{\op\,m_{2n}}{\L^2}\Nrm{\Dhx\op\,m_{2n}}{\L^q}
			\\
			&\qquad +\hbar\Nrm{\Dhx\op\,m_{2n}}{\L^2}\Nrm{\Dhx\op\,m_{2n}}{\L^q}.
		\end{align*}
		This, together with \eqref{eq:estimate_grad_xx_m}, controls the second term in the right-hand side of Equation~\eqref{eq:grad_xx}.\\
		The third term on the right-hand side of Equation~\eqref{eq:grad_xx} can be dealt in an analogously manner as to the second term by using that $\Dhx E_{\op}=E_{\Dhx\op}$ and Proposition~\ref{prop:weighted_com_est}. This gives
		\begin{align*}
			\frac{1}{\hbar} &\Nrm{\com{\Dhx E_{\op},\op} m_{2n}}{\L^q} \leq C  \Nrm{\rho}{L^r}^{1-s}\Nrm{\rho}{W^{1,r}}^s \Nrm{\Dhv\Dhx\op\,m_{2n}}{\L^q}
			\\
			&\qquad+ \Nrm{\Dhx^2\op\,m_{n+n_0}}{\L^{\fb'}} \Nrm{\op\,m_{2n-1}}{\L^q} + \Nrm{\Dhx^2\op\,m_{2n+k'}}{\L^{r}} \Nrm{\op\,m_{n+k}}{\L^{r_1}}
			\\
			&\qquad+ \Nrm{\Dhx\op\,m_{2n}}{\L^{2}} \Nrm{\op\,m_{2n}}{\L^q} + \hbar\Nrm{\Dhx^2\op\,m_{2n}}{\L^{2}} \Nrm{\op\,m_{2n}}{\L^q}.
		\end{align*}
		We now turn to terms in which the contribution of the exchange term appears. By Proposition~\ref{prop:exchange_moments} we obtain
		\begin{equation}\label{eq:exchange_moments}
			\frac{1}{\hbar} \Nrm{\com{h^3\sfX_{\op},m_{2n}} \Dhx^2\op}{\L^q} \leq C \hbar^{\frac{3}{2}-a} \Nrm{\Dhx\op\,m_{2n}}{\L^2} \Nrm{\Dhx^2\op\,m_{2n}}{\L^q}\,.
		\end{equation} 
		By Proposition~\ref{prop:exchange_com} we get the bound
		\begin{equation}\label{eq:exchange_com}
			\frac{1}{\hbar} \Nrm{\com{h^3\sfX_{\Dhx\op},\Dhx\op} m_{2n}}{\L^q} \leq C \hbar^{3\(\frac{1}{q}+\frac{1}{2}-\frac{1}{\fb}\)} \Nrm{\Dhx\op\,m_{2n}}{\L^2} \Nrm{\Dhx\op\,m_{2n}}{\L^\infty}\,.
		\end{equation}
		Finally, by noticing that $\Dhx\sfX_{\Dhx\op}=\sfX_{\Dhx^2\op}$, we apply Proposition~\ref{prop:exchange_com} to the last term on the right-hand side in Equation~\eqref{eq:grad_xx}:
		\begin{equation}\label{eq:exchange_com_grad}
			\frac{1}{\hbar} \Nrm{\com{h^3\sfX_{\Dhx^2\op},\op} m_{2n}}{\L^q} \leq C \hbar^{3\(\frac{1}{q}+\frac{1}{2}-\frac{1}{\fb}\)} \Nrm{\op\,m_{2n}}{\L^\infty} \Nrm{\Dhx^2\op\,m_{2n}}{\L^2}\,.
		\end{equation}
		Therefore, using Proposition~\ref{prop:regu_HF} and estimates \eqref{eq:com_V_est}--\eqref{eq:exchange_com_grad}
		we obtain a bound on the time derivative of$\Nrm{\Dhx^2\op\,m_{2n}}{\L^{q}}$.
		
		We now look at the right-hand side of Equation~\eqref{eq:grad_vv}. By using Proposition~\ref{prop:regu_HF}, Proposition~\ref{prop:weighted_com_est}, Proposition~\ref{prop:weighted_com_est_long_range}, Proposition~\ref{prop:exchange_moments}, Proposition~\ref{prop:exchange_com} and by Proposition~\ref{prop:exchange_com} with $\Dhv\sfX_{\Dhv\op}=\sfX_{\Dhv^2\op}$ we obtain a bound on the time derivative of $\Nrm{\Dhv^2\op\,m_{2n}}{\L^{q}}$.
		
		As for the mixed term~\eqref{eq:grad_vx}, its right-hand side can be bounded as follows. By Proposition~\ref{prop:weighted_com_est} and Proposition~\ref{prop:weighted_com_est_long_range} we get
		\begin{equation}\label{eq:com_V_est-vx}
		\begin{split}
			\frac{1}{\hbar} \Nrm{\com{V_{\op},m_{2n}} \Dhv\Dhx\op}{\L^q}&\leq C \Nrm{\op\,m_{n+n_0}}{\L^{\fb'}} \Nrm{\Dhv\Dhx\op\,m_{2n-1}}{\L^{q}}\\ &+\Nrm{\op\,m_{2n+k'}}{\L^{r}} \Nrm{\Dhv\Dhx\op\,m_{n+k}}{\L^{r_1}}\\ &+\(\Nrm{\rho}{L^1}+\hbar\Nrm{\op\,m_{2n}}{\L^2}\)\Nrm{\Dhv\Dhx\op\,m_{2n}}{\L^q}.
			\end{split}
		\end{equation}
		As for the second term on the right-hand side of Equation~\eqref{eq:grad_vx}, we rewrite it as 
		\begin{equation*}
			\frac{1}{\hbar} \Nrm{\com{E_{\op},\Dhv\op} m_{2n}}{\L^q} = \frac{1}{\hbar} \Nrm{\com{E_{\op},\Dhv\op\,m_{2n}}}{\L^q} + \frac{1}{\hbar} \Nrm{\com{E_{\op},m_{2n}} \Dhv\op}{\L^q}.
		\end{equation*}
		and use again Proposition~\ref{prop:estim_commutator_Lq_Besov} and Corollary~\ref{cor:estim_commutator_Lq} for the first term on the right-hand side and Proposition~\ref{prop:weighted_com_est} and Proposition~\ref{prop:weighted_com_est_long_range} for the second term on the right-hand side. We turn now on the terms in which the contribution of the exchange term appears. By Proposition~\ref{prop:exchange_moments} we obtain
		\begin{equation}\label{eq:exchange_moments-vx}
			\frac{1}{\hbar} \Nrm{\com{\sfX_{\op},m_{2n}} \Dhv\Dhx\op}{\L^q} \leq C\, \hbar^{\frac{3}{2}-a} \Nrm{\Dhx\op\,m_{2n}}{\L^2} \Nrm{\Dhv\Dhx\op\,m_{2n}}{\L^q}.
		\end{equation} 
		By Proposition~\ref{prop:exchange_com} we get the bounds
		\begin{align}\label{eq:exchange_com-vx}
			\frac{1}{\hbar} \Nrm{\com{\sfX_{\Dhx\op},\Dhv\op} m_{2n}}{\L^q} &\leq C\, \hbar^{3\(\frac{1}{q}+\frac{1}{2}-\frac{1}{\fb}\)} \Nrm{\Dhv\op\,m_{2n}}{\L^\infty} \Nrm{\Dhx\op\,m_{2n}}{\L^2}.
			\\\label{eq:exchange_com_grad-vx}
			\frac{1}{\hbar} \Nrm{\com{\sfX_{\Dhv\Dhx\op},\Dhx\op} m_{2n}}{\L^q} &\leq C \hbar^{3\(\frac{1}{q}+\frac{1}{2}-\frac{1}{\fb}\)} \Nrm{\Dhx\op\,m_{2n}}{\L^\infty} \Nrm{\Dhv\Dhx\op\,m_{2n}}{\L^2}.
		\end{align}
		Therefore, using Proposition~\ref{prop:regu_HF} and estimates~\eqref{eq:com_V_est-vx}--\eqref{eq:exchange_com_grad-vx}, we obtain
		\begin{align*}
			\dt\Nrm{\Dhv\Dhx\op\,m_{2n}}{\L^q} &\leq C\Nrm{\Dhv\Dhx\op\,m_{2n}}{\L^{r_1}} + C\Nrm{\Dhv^2\op\,m_{2n}}{\L^q} + C\,\hbar^{\frac{3}{2}-a}\Nrm{\Dhv\Dhx\op\,m_{2n}}{\L^q}
			\\
			&\quad + C\,\hbar^{3\(\frac{1}{q}+\frac{1}{2}-\frac{1}{\fb}\)} \Nrm{\Dhv\op\,m_{2n}}{\L^\infty} + C\,\hbar^{3\(\frac{1}{q}+\frac{1}{2}-\frac{1}{\fb}\)} \Nrm{\Dhv\Dhx\op\,m_{2n}}{\L^2}
		\end{align*}
		for $s=1-3\(\frac{1}{r'}-\frac{1}{\fb}\)$ and with the constraints
		$\frac{1}{r}+\frac{1}{r_1} =\frac{1}{q}+\frac{1}{\fb'}$ and $\frac{1}{r}+\frac{1}{q}=\frac{1}{2}$. Now we define
		\begin{equation*}
			N_{x,q}(t) := \Nrm{\Dhx^2\op\,m_{2n}}{\L^{q}},\quad N_{v,q}(t):=\Nrm{\Dhv^2\op\,m_{2n}}{\L^{q}},\quad N_{xv,q}(t) := \Nrm{\Dhv\Dhx\op\,m_{2n}}{\L^{q}}
		\end{equation*}
		and denote by $N_{2n,q}(t)$ the quantity
		\begin{equation*}
			N_{2n,q}(t)=N_{x,q}(t)+N_{v,q}(t)+N_{xv,q}(t).
		\end{equation*}
		Then we proceed as for the first-order gradients. Using Proposition~\ref{prop:regu_HF}, we obtain a bound on the time derivative of $N_{2n,2}(t)+N_{2n-2,4}$.
		
		For $a\in\left[\frac{1}{2},1\right]$, we consider the quantity
		$F_{2n,\infty}(t) := N_{2n,2}(t)+N_{2n-2,4} + \Nrm{\op\,m_{2n}}{\dot{\cW}^{1,\infty}}$ and look for a Gr\"onwall type inequality. From Equation~\eqref{eq:Dhx} and Equation~\eqref{eq:Dhv} with $q=\infty$, we obtain an upper bound on the time derivative of $F_{2n,\infty}(t)$, using Equation~\eqref{eq:estim_commutator_Linfty} and Proposition~\ref{prop:diag_vs_weights} and standard interpolation allows to conclude by Gr\"onwall's Lemma. 
		
		For $a\in\(0,\frac{1}{2}\)$, using that $\op\in\cW^{1,4}(m_{2n-2})$ by Proposition~\ref{prop:regu_HF} and that
		\begin{align*}
			\Nrm{\Dhx\op\,m_{2n-2}}{\L^\infty} &\leq C\,h^{-\frac{3}{4}}\Nrm{\Dhx\op\,m_{2n-2}}{\L^4},
			\\
			\Nrm{\Dhv\op\,m_{2n-2}}{\L^\infty} &\leq C h^{-\frac{3}{4}}\Nrm{\Dhv\op\,m_{2n-2}}{\L^4},
		\end{align*}
		we get an estimate on the time derivative of $N_{2n,2}(t)+N_{2n-2,4}(t)$. By Gr\"onwall's inequality we conclude that $\op\in\cW^{2,2}(m_{2n})\cap\cW^{2,4}(m_{2n-2})$ for $a\in\(0,\frac{1}{2}\)$.
	\end{proof}
	
	\begin{proof}[Proof of Proposition~\ref{prop:regu_HF_sqrt}]
		We observe that analogously to \eqref{eq:rho}, \eqref{eq:Dhx} and \eqref{eq:Dhv}, the following bounds hold
		\begin{equation*}
			\hbar \,\dt\Nrm{\sqrt{\op}\,m_{2n}}{q} \leq  \Nrm{\com{V_{\op},m_{2n}} \sqrt{\op}}{q} + \Nrm{\com{h^3X_{\op},m_{2n}} \sqrt{\op}}{q},
		\end{equation*}
		\begin{equation}\label{eq:Dhx-sqrt}
		\begin{split}
			\hbar\, \dt\Nrm{\Dhx\sqrt{\op}\,m_{2n}}{q} &\leq \Nrm{\com{V_{\op},m_{2n}} \Dhx\sqrt{\op}}{q} + \Nrm{\com{E_{\op},\sqrt{\op}} m_{2n}}{q}
			\\
			&\quad + \Nrm{\com{h^3X_{\op},m_{2n}} \Dhx\sqrt{\op}}{q} + \Nrm{\com{h^3 X_{\Dhx\op},\sqrt{\op}} m_{2n}}{q},
		\end{split}
		\end{equation}
		and
		\begin{equation}\label{eq:Dhv-sqrt}
		\begin{aligned}
			\hbar\, \dt\Nrm{\Dhv\sqrt{\op}\,m_{2n}}{q} &\leq \Nrm{\com{V_{\op},m_{2n}}\Dhv \sqrt{\op}}{q} + \Nrm{\Dhx\sqrt{\op}\, m_{2n}}{q}
			\\
			&\quad + \Nrm{\com{h^3\sfX_{\op},m_{2n}}\Dhv \sqrt{\op}}{q} + \Nrm{\com{h^3\sfX_{\Dhv \op},\sqrt{\op}}m_{2n}}{q}.
		\end{aligned}
		\end{equation}
		As in Proposition~\ref{prop:regu_HF}, we look for a Gr\"onwall type inequality. To this end, we define 
		\begin{equation*}
			\widetilde{M}_{q}(t) = \Nrm{\sqrt{\op}}{\L^{2}(m_{2n})} + \Nrm{\sqrt{\op}}{\L^{q}(m_{2n})},
		\end{equation*}
		for $q\in[2,\infty]$ and notice that, because of Propositions~\ref{prop:weighted_com_est}, Proposition~\ref{prop:weighted_com_est_long_range} and \ref{prop:exchange_moments}, we have
		\begin{equation*}
			\dt\widetilde{M}_{q}(t)\leq C\,M_{r}(t)\,\widetilde{M}_{{r_1}}(t) + C\,M_2(t)\,\widetilde{M}_{q}(t)\,,
		\end{equation*}
		that implies the boundedness of $\widetilde{M}_{q}(t)$ for $q\in[2,\infty]$ thanks to Proposition~\ref{prop:regu_HF}. 
		
		We now define the quantity
		\begin{equation*}
			\widetilde{N}_{q}(t) = \Nrm{\sqrt{\op}}{\dot{\cW}^{1,2}(m_{2n})} + \Nrm{\sqrt{\op}}{\dot{\cW}^{1,q}(m_{2n})},
		\end{equation*}
		for $q\in[2,\infty]$ and using Equations~\eqref{eq:Dhx-sqrt} and \eqref{eq:Dhv-sqrt} we compute
		\begin{equation}\label{eq:gronwall-sqrt-1}
			\dt\!\(\widetilde{N}_2(t)+\widetilde{N}_{q}(t)\).
		\end{equation}
		The contributions due to the direct term in \eqref{eq:gronwall-sqrt-1} can be estimated in terms of $M_{r}$ and $\widetilde{N}_{r}$ by Proposition~\ref{prop:weighted_com_est}, in terms of $M_r^{1+\theta}$ (for $\theta\in(0,1)$) and $\widetilde{N}_{q}$ by Proposition~\ref{prop:estim_commutator_Lq_Besov}, together with $\widetilde{N}_{r}$ by Proposition~\ref{prop:weighted_com_est}. The contributions due to the exchange term in \eqref{eq:gronwall-sqrt-1} can be estimated in terms of $M_2$ and $\widetilde{N}_{q}$ by Proposition~\ref{prop:exchange_moments}, and in terms of $\widetilde{M}_{q}$, $\widetilde{N}_{q}$ and $N_2$ by Proposition~\ref{prop:exchange_com}. Hence, in the same spirit of the proofs of Propositions~\ref{prop:regu_HF} and \ref{prop:regu_HF_2}, using the results of Proposition~\ref{prop:regu_HF} and Proposition~\ref{prop:regu_HF_2}, we conclude by Gr\"onwall's Lemma obtaining boundedness of $\widetilde{N}_{q}$ for $q\in[2,\infty]$. 
	\end{proof}

\bigskip
\part{\large Mean-Field Limit}\label{part:mean-field}

\section{Scaling}

	In order to define the Bogoliubov rotation as explained in Section~\ref{sec:Bogoliubov}, we define
	\begin{equation}\label{eq:def_omega}
		\omega := \lambda\, \op \ \text{ with } \lambda = N\,h^3
	\end{equation}
	so that $\Tr{\omega} = N$ and $0\leq \omega \leq \lambda\,\cC_\infty \leq 1$. Notice that in the critical scaling $N = C\,h^{-3}$, $\lambda$ is a constant, while in the other cases when $N = h^{-c}$ with $c<3$ we have $\lambda \to 0$. We also define
	\begin{equation*}
		v = \sqrt{\omega}\quad \text{ and }\quad u = \sqrt{1-\omega},
	\end{equation*}
	which are well defined bounded positive operators since $0\leq \omega \leq 1$. With these definitions, we obtain the following behavior for the Schatten norms for $p\in[1,\infty]$
	\begin{align*}
		\Nrm{\omega}{p} &= \cC_p\, N\,h^{\frac{3}{p'}}
		\\
		\Nrm{v}{p} &= \cC_{p/2}^{1/2}\, N^{1/2}\, h^{3\(\frac{1}{2}-\frac{1}{p}\)},
	\end{align*}
	where $\cC_p = \Nrm{\op}{\L^p}$ and $p' = \frac{p}{p-1}$. The operator $u$ satisfies $\Nrm{u}{\infty} \leq 1$, but of course $u$ is not bounded in other Schatten norms. However, it is possible to prove that $0\leq 1-u \leq \omega$, hence $1-u$ is of the same order of magnitude as $\omega$. Since $\Dh_{\!\eta} u = -\Dh_{\!\eta}(1-u)$, it explains why we can expect the gradients of $u$ to be of the same order as $\Dh_{\!\eta} \omega$, as indicated more precisely in the following lemma.
	
	\begin{lem}\label{lem:nabla_u}
		Assume $\Nrm{\omega}{\infty} = \lambda\,\cC_\infty < 1$. Then
		\begin{equation*}
			C\,\Nrm{\Dh_{\!\eta} u\, m}{p} \leq \Nrm{\Dh_{\!\eta}\,\omega\, m}{p} + \Nrm{\omega\, \Dh_{\!\eta} m}{p},
		\end{equation*}
		with $C = 2\sqrt{1-\lambda\, \cC_\infty}$. In particular, it implies that
		\begin{equation*}
			C\,\Nrm{\Dhv u\, m}{p} \leq \cD_p\, N\,h^\frac{3}{p'}
		\end{equation*}
		where $\cD_p = \Nrm{\Dhv\,\op\, m}{\L^p} + \Nrm{\op\, \Dhv m}{\L^p}$ is of order $1$ in the semiclassical limit.
	\end{lem}

	\begin{proof}
		Since $\Nrm{\omega}{\infty} < 1$, we can write $u = (1-\omega)^\frac{1}{2} = \sum_{n=0}^\infty \binom{1/2}{n} \(-1\)^n\omega^n$. Therefore, for $\eta\in\set{x,\xi}$, we obtain
		\begin{equation*}
			\Nrm{\Dh_{\!\eta} u\,m}{p} = \Nrm{\Dh_{\!\eta} (u-1)\,m}{p}\leq \sum_{n=1}^\infty \n{\binom{1/2}{n}} \Nrm{\Dh_{\!\eta}\(\omega^n\,m\)}{p}.
		\end{equation*}
		Expanding the gradient with the product rule for commutators gives
		\begin{equation*}
			\Dh_{\!\eta}\(\omega^n\,m\) = \omega^n\, \Dh_{\!\eta} m + \sum_{k=1}^n \omega^{k-1} \(\Dh_{\!\eta}\omega\) \omega^{n-k}
		\end{equation*}
		which leads to
		\begin{equation*}
			\Nrm{\Dh_{\!\eta} u\,m}{p} \leq \sum_{n=1}^\infty \n{\binom{1/2}{n}} n\Nrm{\omega}{\infty}^{n-1} \(\Nrm{\Dh_{\!\eta}\,\omega\,m}{p} + \Nrm{\omega\, \Dh_{\!\eta} m}{p}\).
		\end{equation*}
		Moreover, for $n\geq 1$, $\n{\binom{1/2}{n}} = \(-1\)^{n-1}\binom{1/2}{n} $ and $\binom{1/2}{n} = \frac{1}{2n} \binom{-1/2}{n-1}$, from which we deduce
		\begin{equation*}
			\sum_{n=1}^\infty \n{\binom{1/2}{n}} n\Nrm{\omega}{\infty}^{n-1} \leq \frac{1}{2}\, \sum_{n=1}^\infty \binom{-1/2}{n-1} \(-1\)^{n-1} \Nrm{\omega}{\infty}^{n-1} = \frac{1}{2\sqrt{1-\Nrm{\omega}{\infty}}}
		\end{equation*}
		and the proof follows by combining the two last inequalities.
	\end{proof}
	
\section{Preliminary Inequalities}

	In this section, we provide estimates which are crucial for controlling the growth of the particle number operator with respect to the fluctuation dynamics in the subsequent sections. 
	
	As a preliminary, let us begin by defining some convenient notations. 
	For any pair $(\sigma,\sigma') \in \set{l,r}^2$ and a bounded operator $O:\h_{\sigma'}\rightarrow \h_{\sigma}$, we generalize the standard notation of the second quantization of the one-particle operator by
	\begin{subequations}\label{def:gen-one-part-op}
	\begin{align}
		\dG_{\sigma,\sigma'}(O) =&\ \intdd O(x,y)\,a^*_{x,\sigma}\,a_{y,\sigma'} \d x\d y\,
		\\
		\dG_{\sigma,\sigma'}^+(O) =&\ \intdd O(x,y)\,a^*_{x,\sigma}\,a^*_{y,\sigma'} \d x\d y\,
		\\
		\dG_{\sigma,\sigma'}^-(O) =&\ \intdd O(x,y)\,a_{x,\sigma}\,a_{y,\sigma'} \d x\d y\
	\end{align}
	\end{subequations}
	where the operators are expressed in terms of operator-valued distributions \eqref{eq:op-val-distr}. 
	When $\sigma = \sigma'$, we write $\dG^\circ_\sigma := \dG^\circ_{\sigma, \sigma}$ where $\circ$ denotes either $+$, $-$, or null. Moreover, we have the relations 
    \begin{equation}\label{eq:dG_adjoint_property}
        \dG_{\sigma, \sigma'}(O)^* = \dG_{\sigma', \sigma}(O^*) \quad \text{ and } \quad \dG^+_{\sigma, \sigma'}(O)^* =\dG^-_{\sigma', \sigma}(O^*). 
    \end{equation}

	We begin by extending \cite[Lemma~4.2]{benedikter_mean-field_2016} to the case of Schatten class operators between different Hilbert spaces. See \cite[Chapter~7]{weidmann_linear_1980}. 
	
	\begin{lem}\label{lem:second_quantization_op}
		Let $(\sigma',\sigma) \in \set{l,r}^2$ and $O:\h_{\sigma'}\rightarrow \h_{\sigma}$ be a compact operator. Then, for every $p\in [1,\infty]$, we have the estimate
		\begin{equation}\label{eq:second_quantization_op}
			\Nrm{\dG_{\sigma}(O)\Psi}{\cG} \leq \Nrm{O}{p} \Nrm{\cN^{\frac{1}{p'}}\Psi}{\cG}
		\end{equation}
		for every $\Psi \in \cG$ where $\cN = \DGl{1} + \DGr{1}$. Moreover, for $p\in[1,2]$ we have the estimates
		\begin{subequations}
			\begin{align}\label{eq:second_quantization_op_-}
				\Nrm{\dG^-_{\sigma,\sigma'}(O)\Psi}{\cG} &\leq \Nrm{O}{p} \Nrm{\cN^{\frac{1}{p'}}\Psi}{\cG}, \\
				\label{eq:second_quantization_op_+}
				\Nrm{\dG^+_{\sigma,\sigma'}(O)\Psi}{\cG} &\leq \Nrm{O}{p} \Nrm{\(\cN+2\)^{\frac{1}{p'}}\Psi}{\cG}
			\end{align}
			for every $\Psi \in \cG$.
		\end{subequations}
	\end{lem}
	
	\begin{proof}
		The case of inequality \eqref{eq:second_quantization_op} with $p=\infty$ and the case of inequalities \eqref{eq:second_quantization_op_-} and \eqref{eq:second_quantization_op_+} with $p=2$ are proved in \cite[Lemma~4.2]{benedikter_mean-field_2016}.
		
		For any compact $O$, we can write down a singular value decomposition of $O$, that is, $O = \sum_j \mu_j \inprod{\phi_j}{\cdot}\varphi_j$ where $(\phi_j)_{j\in\N}\subset \h_{\sigma'}$ and $(\varphi_j)_{j\in\N}\subset\h_\sigma$ are two orthonormal sets, and $\mu_j \geq 0$ are the singular values of $O$ (see e.g. \cite[Theorem~7.6]{weidmann_linear_1980}). Thus, using the notation $a^\sharp$ to denote either $a$ or $a^*$, we have
		\begin{equation*}
			\Nrm{\intdd O(x,y)\,a^\sharp_{x,\sigma}\,a^\sharp_{y,\sigma'} \d x\d y}{\infty} \leq \sum_j \mu_j \Nrm{a^\sharp_{\sigma}(\tilde{\phi}_j)\, a^\sharp_{\sigma'}(\tilde\varphi_j)}{\infty} 
		\end{equation*}
		where $\tilde{\phi}_j$ is either $\phi_j$ or $\bar\phi_j$. Since $\Nrm{a^\sharp_{\sigma}(\varphi)}{\infty} \leq \Nrm{\varphi}{L^2} = 1$, we obtain the estimate
		\begin{equation*}
			\Nrm{\intdd O(x,y)\,a^\sharp_{x,\sigma}\,a^\sharp_{y,\sigma'} \d x\d y\ }{\infty} \leq \sum_j \mu_j = \Nrm{O}{1}.
		\end{equation*}
		Hence, for any $\circ \in \set{+,-,\ }$, we have the estimate $\Nrm{\dG^\circ_{\sigma,\sigma'}(O)\Psi}{\cG} \leq \Nrm{O}{1} \Nrm{\Psi}{\cG}$. Finally, we deduce the desired result by weighted interpolation.
	\end{proof}
	
	As an immediate application, we can bound the expectation values of the operators \eqref{def:gen-one-part-op} in terms of the expectation values of powers of the number operator. 
	
	\begin{lem}\label{lem:second_quantization_op_2}
		For any $p\in[1,\infty]$, we have the estimate 
		\begin{equation}\label{eq:inner_product_bound}
			\Inprod{\Psi}{\dG_{\sigma}(O)\Psi}_{\cG} \leq \Nrm{O}{p}\Inprod{\Psi}{\cN^\frac{1}{p'}\Psi}_{\cG}
		\end{equation}
		for every $\Psi \in \cG$. Similarly, for any $p\in[1,2]$, we have the estimates
		\begin{subequations}
			\begin{align}\label{eq:inner_product_bound_2}
				\Inprod{\Psi}{\dG^+_{\sigma,\sigma'}(O)\Psi}_{\cG} &\leq 2^\frac{1}{2p'} \Nrm{O}{p}\Inprod{\Psi}{\(\cN+1\)^\frac{1}{p'}\Psi}_{\cG}
				\\
				\label{eq:inner_product_bound_3}
				\Inprod{\Psi}{\dG^-_{\sigma,\sigma'}(O)\Psi}_{\cG} &\leq 2^\frac{1}{2p'} \Nrm{O}{p}\Inprod{\Psi}{\(\cN+1\)^\frac{1}{p'}\Psi}_{\cG}
			\end{align}
		\end{subequations}
		for every $\Psi \in \cG$. 
	\end{lem}
	
	\begin{proof}
		For $\epsilon>0$, one has the equality
		\begin{align*}
			\Inprod{\Psi}{\DG{O}\Psi}_{\cG} &= \Inprod{\(\cN+\epsilon\)^{\frac{1}{2p'}}\Psi}{\(\cN+\epsilon\)^{-\frac{1}{2p'}}\DG{O}\Psi}_{\cG}
			\\
			&= \Inprod{\(\cN+\epsilon\)^{\frac{1}{2p'}}\Psi}{\DG{O}\(\cN+\epsilon\)^{-\frac{1}{2p'}}\Psi}_{\cG}.
		\end{align*}
		Applying the Cauchy--Schwarz inequality and Lemma~\ref{lem:second_quantization_op} yields
		\begin{align*}
			\Inprod{\Psi}{\DG{O}\Psi}_{\cG} &\leq \Nrm{O}{p} \Nrm{\(\cN+\epsilon\)^{\frac{1}{2p'}}\Psi}{\cG} \Nrm{\cN^\frac{1}{p'}\(\cN+\epsilon\)^{-\frac{1}{2p'}}\Psi}{\cG}
			\\
			&\leq \Nrm{O}{p} \Nrm{\(\cN+\epsilon\)^{\frac{1}{2p'}}\Psi}{\cG} \Nrm{\cN^\frac{1}{2p'}\Psi}{\cG}.
		\end{align*}
		Then inequality~\eqref{eq:inner_product_bound} follows by passing to the limit $\epsilon\to 0$. With a similar argument and the observation that for any nice function $g$, $g(\cN)\,a^* = a^*\,g(\cN+1)$, we obtain
		\begin{equation*}
			\Inprod{\Psi}{\dG^+_{\sigma,\sigma'}(O)\Psi}_{\cG} \leq \Nrm{O}{p} \Nrm{\cN^{\frac{1}{2p'}}\Psi}{\cG} \Nrm{\(\cN+2\)^\frac{1}{2p'}\Psi}{\cG}
		\end{equation*}
		from which we deduce inequality~\eqref{eq:inner_product_bound_2}. Inequality~\eqref{eq:inner_product_bound_3} follows immediately from Equation~\eqref{eq:dG_adjoint_property}.
	\end{proof}

\section{Quantum Fluctuations and the Mean-Field Limit}
	
	 In this section, we prove how the error of the mean-field approximation of the fermionic system can be controlled by the mean number of particles of the fluctuation dynamics about a quasi-free state. To this end, it suffices for us to specialize our study to the state vector 
	 \begin{equation}\label{fluctuation_state_vector}
	 \Psi_\text{fluc} = \sfR^*_{\op}\Phi_{t}=\sfR^*_{\op}e^{-i(t/\hbar)\sfL_N}\sfR_{\op_0}\Psi^\init
	 \end{equation}
	 and consider its mean number of particles
	\begin{equation*}
		\Inprod{\Psi_\text{fluc}}{\cN \Psi_\text{fluc}}_{\cG} = \Nrm{\cN^\frac{1}{2}\Psi_\text{fluc}}{\cG}^2.
	\end{equation*}
	More specifically, we control the error of the mean-field approximation by the norm
	\begin{equation}\label{mean-particle-number}
		\Nrm{\Psi_\text{fluc}}{\cG_k} := \Nrm{\(\cN+1\)^k\Psi_\text{fluc}}{\cG}
	\end{equation}
	for $k>0$, which allows us to handle additional small error terms. For the rest of this section, we drop the subscript of the fluctuation vector and the dependence on time to reduce cumbersome notations. 
	
	One can see that quantity \eqref{mean-particle-number} controls the difference of the one-particle density operators in the sense of the following proposition.
	
	\begin{prop}\label{prop:degree_evaporation_vs_Sp}
		Define $\Psi$ and $\op_{N:1}$ as in the Theorem~\ref{thm:mean_field_2}. Then, for any $p\in[1,\infty]$, we have the estimate 
		\begin{equation*}
			\Nrm{\op_{N:1} - \op}{\L^p} \leq \frac{C_p}{\min(N^{\frac{1}{2}},N\, h^{\frac{3}{p'}})} \Nrm{\Psi}{\cG_\frac{1}{2p}}^2
		\end{equation*}
		where $C_p = 2^{2+\frac{1}{2p}}$ if $p\geq 2$ and $C_p = 2+2^\frac{5}{4} \cC_{\frac{p}{2-p}}^{\frac{1}{2}}$ if $p< 2$.
	\end{prop}
	
	\begin{proof}
		Following the proof of \cite[Proof of Theorem~2.1]{benedikter_mean-field_2016-1}, we have
		\begin{align*}
			 N h^3 \opNr(x,y)-\omega(x, y) &= \Inprod{\Psi_{N}}{a_{y,l}^* \, a_{x,l}\Psi_{N}}_{\cG} = \Inprod{\Psi}{\sfR_{\op}^*\, a_{y,l}^* \,a_{x,l}\, \sfR_{\op} \Psi}_{\cG}
			\\
			&= \langle \Psi \,|\, \big(a_l^*(u_y) \, a_l(u_x) - a_l^*(u_y) \, a_r^*(\conj{v}_x)
			\\
			&\qquad\qquad - a_r(\conj{v}_y) \, a_l(u_x)- a_r^*(\conj{v}_x) \, a_r(\conj{v}_y)\big) \Psi\rangle_{\cG}.
		\end{align*}
		Since $\op = \frac{1}{Nh^3} \omega$, we deduce
		\begin{equation}\label{eq:Fock-mean-field-error}
			\begin{aligned}
				\(\op_{N:1} - \op\)\!(x,y) 
				&= \frac{1}{Nh^3} \langle\Psi \,|\, \big(a_l^*(u_y)\, a_l(u_x)- a_l^*(u_y)\, a_r^*(\conj{v}_x)
				\\
				&\qquad\qquad\qquad - a_r(\conj{v}_y)\, a_l(u_x) - a_r^*(\conj{v}_x)\, a_r(\conj{v}_y)\big) \Psi \rangle_{\cG}.
			\end{aligned}
		\end{equation}
		In particular, pairing operator \eqref{eq:Fock-mean-field-error} with an observable $O$ yields
		\begin{multline*}
			\Tr{O\(\op_{N:1} - \op\)}
			\\
			= \frac{1}{N\,h^3} \Inprod{\Psi}{\(\DGl{u\, O\, u} - \DGr{\conj{v}\, \conj{O}^* v} - \dG^+_{l,r}(v\,O\,u) - \dG^-_{r,l}(v\,O\,u)\)\Psi}_{\cG}.
		\end{multline*}
		
		In the case $p\in[2,\infty]$, we apply the fact that $\Nrm{u}{\infty} \leq 1, \Nrm{v}{\infty} \leq 1$, and Lemma~\ref{lem:second_quantization_op_2} to deduce the estimate
		\begin{equation*}
			\Tr{O\(\op_{N:1} - \op\)} \leq \frac{2^{2+\frac{1}{2p}}}{N\,h^3} \Nrm{O}{p'} \Inprod{\Psi}{\(\cN+1\)^\frac{1}{p}\Psi}_{\cG}.
		\end{equation*}
		Then, by duality and the fact that $\Nrm{\opmu}{\L^p} = h^{\frac{3}{p}} \Nrm{\opmu}{p}$, we obtain the result when $p\geq 2$.
		
		For $p\in[1,2]$, we can bound the terms with $\DGl{u\, O\, u}$ and $\DGr{\conj{v}\, \conj{O}^* v}$ as in the previous case. For the other two terms, we begin by applying H\"older's inequality to get $\Nrm{v\, O\, u}{2} \leq \Nrm{v}{r}\Nrm{O}{p'}$ where $\frac{1}{r} =\frac{1}{2}- \frac{1}{p'}$. Then, by Lemma~\ref{lem:second_quantization_op_2}, it follows that 
		\begin{align*}
			\n{\Inprod{\Psi}{ \(\dG^+_{l,r}(v\,O\,u) + \dG^-_{r,l}(v\,O\,u)\)\Psi}_{\cG}} &= 2 \n{\Inprod{\Psi}{\dG^-_{r,l}(v\,O\,u)\Psi}_{\cG}}
			\\
			&\leq 2^\frac{5}{4} \Nrm{v}{r}\Nrm{O}{p'} \Inprod{\Psi}{\(\cN+1\)^\frac{1}{2}\Psi}_{\cG}.
		\end{align*}
		Since $\Nrm{v}{r} = \cC_{\frac{r}{2}}^{\frac{1}{2}} N^\frac{1}{2} h^{3\(\frac{1}{2} - \frac{1}{r}\)}$ and $\frac{1}{2} - \frac{1}{r} = \frac{1}{p'}$, this implies
		\begin{equation*}
			\n{\Inprod{\Psi}{ \(\dG^+_{l,r}(v\,O\,u) + \dG^-_{r,l}(v\,O\,u)\)\Psi}_{\cG}} \leq 2^\frac{5}{4} \Nrm{O}{p'} \cC_{\frac{r}{2}}^{\frac{1}{2}} N^\frac{1}{2} h^{\frac{3}{p'}} \Inprod{\Psi}{\(\cN+1\)^\frac{1}{2}\Psi}_{\cG}.
		\end{equation*}
		So, we have the estimate 
		\begin{equation*}
			\Tr{O\(\op_{N:1} - \op\)} \leq \Nrm{O}{p'}\(\frac{2}{N\,h^3} + \frac{2^\frac{5}{4} \cC_{\frac{r}{2}}^{\frac{1}{2}}}{N^{\frac{1}{2}}\,h^{\frac{3}{p}}} \) \Inprod{\Psi}{\(\cN+1\)^\frac{1}{p}\Psi}_{\cG}
		\end{equation*}
		which yields the desired result.
	\end{proof}
	
	To better understand what it means to have a small number of particles after having performed the Bogoliubov transformation, it is useful to see how this latter acts on the number operator. From the definition \eqref{eq:implementation_of_nu}, we obtain the following formula for $\sigma\in\set{r,l}$
	\begin{equation}\label{eq:bogoliubov_N_lr}
		\sfR_{\op}\,\cN_\sigma \sfR_{\op}^* = \sfA_\sigma + \sfC + \sfC^*,
	\end{equation}
	where
	\begin{equation*}
		\sfA_\sigma = \cN_\sigma + N - \DG{\omega\oplus\conj{\omega}}, \quad \text{ and } \quad \sfC = \DGmrl{uv}.
	\end{equation*}
	Since changing $v$ by $-v$ changes $\sfR_{\op}$ to $\sfR_{\op}^*$, we deduce similarly that
	\begin{equation}\label{eq:bogoliubov_N_lr_2}
		\sfR_{\op}^*\,\cN_\sigma \sfR_{\op} = \sfA_\sigma - \sfC - \sfC^*.
	\end{equation}
	From these formulas, we deduce the following interesting fact: the operator $\opnu_{\op}$ acting on the single Fock space $\cF$ and corresponding to the Bogoliubov transform of the vacuum in $\cG$ commutes with the number of particles operator .
	
	\begin{lem}\label{lem:commut_opnu_N}
		Let $\opnu_{\op} := \sfI_\cG^{-1}(\sfR_{\op}\,\Omega)$. Then
		\begin{equation*}
			\com{\cN, \opnu_{\op}} = 0.
		\end{equation*}
		This also implies that $\opNop := \n{\opnu_{\op}}^2$ commutes with $\cN$.
	\end{lem}
	
	\begin{proof}
		Let $\Phi_{\op} := \sfR_{\op}\,\Omega = \sfI_\cG\opnu_{\op}$. Then  we obtain $\cN_l\Phi_{\op} = \sfI_\cG(\cN\opnu_{\op})$ and $\cN_r\Phi_{\op} = \sfI_\cG(\opnu_{\op}\cN)$, therefore
		\begin{equation*}
			\sfI_\cG^{-1}\com{\cN, \opnu_{\op}} = \(\cN_l - \cN_r\)\Phi_{\op} = \(\cN_l - \cN_r\)\sfR_{\op}\,\Omega.
		\end{equation*}
		Now we use Formula~\eqref{eq:bogoliubov_N_lr_2}, yielding
		\begin{equation*}
			\(\cN_l - \cN_r\)\sfR_{\op}\,\Omega = \sfR_{\op}\(\sfA_l-\sfA_r\)\,\Omega = \sfR_{\op}\(\cN_l-\cN_r\)\,\Omega = 0,
		\end{equation*}
		which proves the result.
	\end{proof}
	
	Since the number operator on the double Fock space $\cG$ is given by $\cN = \cN_l+\cN_r$, Equations~\eqref{eq:bogoliubov_N_lr} and \eqref{eq:bogoliubov_N_lr_2} imply
	\begin{align}\label{eq:bogoliubov_N}
		\sfR_{\op}\, \cN \sfR_{\op}^* &= \sfA + 2\sfC + 2\sfC^*
		\\\label{eq:bogoliubov_N_2}
		\sfR_{\op}^*\, \cN \sfR_{\op} &= \sfA - 2\sfC - 2\sfC^*
	\end{align}
	with $\sfA = \sfA_l + \sfA_r = \cN + 2 N - 2 \DG{\omega \oplus \conj{\omega}}$. This allows us to prove the following bounds.
	
	\begin{lem}\label{lem:bogoliubov_N_bound}
		Let $k\in\N$. Then for any $\Psi\in \cG_{k}$
		\begin{align*}
			\Nrm{\cN^k \sfR_{\op}^*\Psi}{\cG} &\leq 3^k\Nrm{\(\cN + 2N+2k\)^{k}\Psi}{\cG}
			\\
			\Nrm{\cN^k \sfR_{\op}\Psi}{\cG} &\leq 3^k\Nrm{\(\cN + 2N+2k\)^{k}\Psi}{\cG}.
		\end{align*}
	\end{lem}
	
	\begin{remark}
		With a similar proof, one obtains
		\begin{equation*}
			\Nrm{\sfR_{\op}^*\Psi}{\cG_{1/2}} \leq 3^\frac{1}{2} \Nrm{\(\cN+2N\)^\frac{1}{2}\Psi}{\cG}
		\end{equation*}
		and so by interpolation, for any $s\in[0,1/2]$,
		\begin{equation}\label{eq:bogoliubov_N_bound_small_s}
			\Nrm{\sfR_{\op}^*\Psi}{\cG_{s}} \leq 3^s \Nrm{\(\cN+2N\)^s\Psi}{\cG}.
		\end{equation}
	\end{remark}
	
	\begin{proof}[Proof of Lemma~\ref{lem:bogoliubov_N_bound}]
		Since $\sfR_{\op}^* \cN \sfR_{\op}$ and $\sfR_{\op} \cN \sfR_{\op}^*$ are two positive operators, by doing the sum of the two above identities, we deduce that $\sfA$ is also a positive operator. Therefore, since $\DG{\omega \oplus \conj{\omega}}$ is a positive operator, from the definition of $\sfA$, we obtain
		\begin{equation*}
			0 \leq \sfA \leq \cN + 2N,
		\end{equation*}
		which implies that for any $\Psi\in\cG_{1}$, $\Nrm{\sfA^{1/2}\Psi}{\cG} \leq \Nrm{\(\cN + 2N\)^{1/2}\Psi}{\cG}$. Since $\sfA$ commutes with $\cN+2N$, we deduce that
		\begin{equation*}
			\Nrm{\sfA\Psi}{\cG} \leq \Nrm{\(\cN + 2N\)\Psi}{\cG}.
		\end{equation*}
		On the other hand, by Inequalities~\eqref{eq:second_quantization_op_-} and \eqref{eq:second_quantization_op_+} and the fact that $\Nrm{u}{\infty}\leq 1$ and $\Nrm{v}{2} = N^\frac{1}{2}$, we have
		\begin{equation*}
			\Nrm{\sfC^*\Psi}{\cG} \leq \Nrm{uv}{2} \Nrm{\(\cN+2\)^{1/2}\Psi}{\cG} \leq \frac{1}{2} \Nrm{\(\cN+N+2\)\Psi}{\cG}
		\end{equation*}
		and similarly, $\Nrm{\sfC\Psi}{\cG} \leq \frac{1}{2} \Nrm{\(\cN+N\)\Psi}{\cG}$. From these inequalities, using the fact that $\sfA$ commutes with $\cN$ and the fact that $\cN \sfC = \sfC \(\cN-2\)$ and $\cN \sfC^* = \sfC^* \(\cN+2\)$, we deduce that for any $j\in\N$, by defining $c_j := 2N+2j$, we have
		\begin{multline*}
			\Nrm{\(\cN+c_j\)^j\sfR_{\op} \cN \sfR_{\op}^*\Psi}{\cG} 
			\\
			\leq \Nrm{\sfA\(\cN+c_j\)^j\Psi}{\cG} + 2\Nrm{\sfC\(\cN+c_j-2\)^j\Psi}{\cG} + 2\Nrm{\sfC^*\(\cN+c_j+2\)^j\Psi}{\cG}
			\\
			\leq 3\Nrm{\(\cN+c_{j+1}\)^{j+1}\Psi}{\cG}.
		\end{multline*}
		By induction, this implies that for any $(j,k)\in\N^2$ 
		\begin{equation*}
			\Nrm{\(\cN+c_j\)^j \(\sfR_{\op} \cN \sfR_{\op}^*\)^k\Psi}{\cG} \leq 3^k\Nrm{\(\cN + c_{j+k}\)^{j+k}\Psi}{\cG}.
		\end{equation*}
		Taking $j= 0$ and using the fact that $\sfR_{\op}$ is unitary $\(\sfR_{\op} \cN \sfR_{\op}^*\)^k = \sfR_{\op} \cN^k \sfR_{\op}^*$, we get
		\begin{equation*}
			\Nrm{\cN^k \sfR_{\op}^*\Psi}{\cG} = \Nrm{\(\sfR_{\op} \cN \sfR_{\op}^*\)^k\Psi}{\cG} \leq 3^k\Nrm{\(\cN + c_{k}\)^{k}\Psi}{\cG}.
		\end{equation*}
		The case of $\cN^k \sfR_{\op}$ can be handled in the same way.
	\end{proof}

\section{The Fluctuation Dynamics}

	With the scaling provided in \eqref{eq:def_omega}, we have that $\rho(x) = N^{-1}\, \omega(x,x)$. Let us define $X_\omega(x,y) := N^{-1}\, K(x-y)\, \omega(x,y)$, then this gives us the relation $X_\omega = h^3\, \sfX_{\op}$. Thus, the Hartree--Fock equation \eqref{eq:Hartree--Fock} can be rewritten as follows
	\begin{equation}\label{eq:Hartree--Fock_omega}
		i\hbar\,\dpt \omega = \com{H_\omega,\omega} \ \text{ with } \
		H_\omega = -\frac{\hbar^2}{2}\, \Delta + K*\rho - X_\omega.
	\end{equation}
	By \cite[Proposition~3.1]{benedikter_mean-field_2016}, we know that the dynamics of $\Psi_\text{fluc}$ satisfies 
	\begin{equation*}
		i\hbar\,\dpt \sfU_{t,s} = \sfG_{t}\, \sfU_{t,s} \ \text{ with } \ \sfU_{s,s} = 1 \ \text{ for all } s\in \RR
	\end{equation*}
	and the generator $\sfG_{t}$ is given by
	\begin{equation}\label{eq:def_G}
		\sfG = \DGl{H_\omega} - \DGr{\conj{H_\omega}} + \sfD + \sfQ + \sfQ^* + \tilde{\sfQ} + \tilde{\sfQ}^*
	\end{equation}
	where%
	\begin{small}
	\begin{align*}
		\sfD &= \frac{1}{2\,N} \intdd K(x-y) \bigg(a_l^*(u_x)\, a_l^*(u_y)\, a_l(u_y)\, a_l(u_x) - a_r^*(\conj{u_x})\, a_r^*(\conj{u_y})\, a_r(\conj{u_y})\, a_r(\conj{u_x}) 
		\\
		&\qquad\qquad\qquad\qquad\quad + 2 \,a_l^*(u_x)\, a_r^*(v_x)\, a_r(v_y)\, a_l(u_y) - 2\,a_l^*(u_x)\, a_r^*(\conj{v_y})\, a_r(\conj{v_y})\, a_l(u_x)
		\\
		&\qquad\qquad\qquad\qquad\quad + 2\,a_r^*(\conj{u_x})\, a_l^*(v_y)\, a_l(v_y)\, a_r(\conj{u_x}) - 2\, a_r^*(\conj{u_x})\, a_l^*(v_x)\, a_l(v_y)\, a_r(\conj{u_y}) 
		\\
		&\qquad\qquad\qquad\qquad\quad + a_r^*(\conj{v_y})\, a_r^*(\conj{v_x})\, a_r(\conj{v_x})\, a_r(\conj{v_y}) - a_l^*(v_y)\, a_l^*(v_x)\, a_l(v_x)\, a_l(v_y)\bigg)\d x\d y
		\\
		\sfQ^* &= \frac{1}{N} \intdd K(x-y) \,\bigg(a_l^*(u_x)\, a_l^*(u_y)\, a_r^*(\conj{v_x})\, a_l(u_y)
		- a_r^*(\conj{u_x})\, a_l^*(v_y)\, a_l^*(v_x)\, a_l(v_y)
		\\
		&\qquad\qquad\qquad\qquad\quad
		+ a_r^*(\conj{u_x})\, a_r^*(\conj{u_y})\, a_l^*(v_x)\, a_r(\conj{u_y})
		- a_l^*(u_x)\, a_r^*(\conj{v_y})\, a_r^*(\conj{v_x})\, a_r(\conj{v_y})
		\bigg)\d x\d y
		\\
		\tilde{\sfQ}^* &= \frac{1}{2\,N} \intdd K(x-y) \bigg(
		a_l^*(u_x)\, a_l^*(u_y)\, a_r^*(\conj{v_y})\, a_r^*(\conj{v_x})
		- a_r^*(\conj{u_x})\, a_r^*(\conj{u_y})\, a_l^*(v_y)\, a_l^*(v_x)\bigg)\d x\d y
	\end{align*}
	\end{small}%
	with $u_x(y) := u(y,x)$ and $v_x(y) := v(y,x)$. $\sfD$ contains quartic terms that commute with $\cN = \cN_l + \cN_r$, whereas $\sfQ^*$ and $\tilde{\sfQ}^*$ contain quartic terms that do not commute with $\cN$.

\subsection{Bounds on the Fluctuation Dynamics}\label{sec:fluctuations}

	In this section, we use the uniform in $\hbar$ regularity of the solution of the Hartree--Fock equation to estimate the growth of mean number of particles for the fluctuation dynamics. 

	We fix $p\in[1,2]$ with 
	\begin{equation}\label{eq:scaling_p_b}
		p< \fb = \frac{3}{a+1}
	\end{equation}
	and take $1\leq q_0 < q_1 \leq \infty$ such that
	\begin{equation}\label{eq:scaling_relation_1}
		\frac{1}{2} \(\frac{1}{q_1} + \frac{1}{q_0}\) = \frac{1}{p} - \frac{1}{\fb}.
	\end{equation}
	We choose $T>0$ so that the following two quantities are uniformly bounded on $[0, T]$:
	\begin{subequations}
		\begin{align}\label{eq:def_Dqq}
			\tilde{\cD}_{q_0,q_1} &:= \Nrm{\Dhv{\sqrt{\op}}\,m}{\L^{q_0}}^{\frac{1}{2}} \Nrm{\Dhv{\sqrt{\op}}\,m}{\L^{q_1}}^{\frac{1}{2}},
			\\\label{eq:def_Dqq_tilde}
			\cD_{q_0,q_1} &:= \(\cD_{q_0}\cD_{q_1}\)^\frac{1}{2},
		\end{align}
	\end{subequations}
	with $\cD_{q}$ defined in Lemma~\ref{lem:nabla_u} and $m = 1 + \n{\opp}^n$ with $n>a+1$. The main result of this section is the following inequality.

	\begin{prop}\label{prop:propagator_bound}
		Let $(k_0,k)\in[0,1/2]\times\N$. Then, for any $\Psi\in\cG$ and $t\in[0,T]$, we have
		\begin{equation*}
			\Nrm{\sfU_{t,0}\Psi}{\cG_{k_0}} \leq C_M\, e^{C_M\,\lambda_\alpha\,t} \(\Nrm{\Psi}{\cG_{k_0+\frac{3}{2}k}} + \frac{h^{\(\alpha-1\) k}}{N^{\frac{k}{2}-k_0}} \, t \Nrm{\Psi}{\cG_{\frac{3}{2}k}}\)
		\end{equation*}
		where $\alpha := \frac{3}{p} - \frac{3}{2}$, $C_M = C^{k+k_0} \(1+ N^{-\frac{1}{2}} h^{-1}\)$ for some constant $C>0$, and
		\begin{equation}\label{eq:lambda-alpha}
			\lambda_\alpha = C_{p,a,q_0}\n{\kappa}h^{-\alpha} \(1+\cC_\infty^{-\frac{1}{2}}\) \sup_{[0,T]}\Big(\Nrm{\rho(t)}{L^{p_a}}, \cD_{q_0,q_1}(t), \tilde{\cD}_{q_0,q_1}(t)\Big)
		\end{equation}
		with $p_a = \frac{3}{3-2a}$. 
	\end{prop}

	\begin{remark}
		With the cut-off given in Remark~\ref{remark:cutoff}, one obtains
		\begin{equation*}
			\Nrm{\sfU_{t,0}\Psi}{\cG_{k_0}} \leq C_M\, e^{C_M\,\lambda_R\,t} \(\Nrm{\Psi}{\cG_{k_0+\frac{3}{2}k}} + \frac{R^{3\alpha k}\,t}{N^{\frac{k}{2}-k_0}h^k} \Nrm{\Psi}{\cG_{\frac{3}{2}k}}\)
		\end{equation*}
		with
		\begin{equation*}
			\lambda_R = C_{p,a,q_0}\n{\kappa} R^{-3\alpha} \(1+\cC_\infty^{-\frac{1}{2}}\) \sup_{[0,T]}(\Nrm{\rho(t)}{L^{p_a}}, \cD_{q_0,q_1}(t), \tilde{\cD}_{q_0,q_1}(t))\,.
		\end{equation*}
	\end{remark}
	
	To prove Proposition~\ref{prop:propagator_bound}, we will first obtain uniform in $\hbar$ estimates for the generator \eqref{eq:def_G}. This is done by proving a series of lemmas. In particular, we will estimate each of the terms of the generator that do not commute with $\cN$ separately.
	 
	\subsubsection{Bounds for $\tilde{\sfQ}$}\label{sec:auxiliary_dynamics}
	For convenience, let us begin by recalling the following lemma. 
	\begin{lem}[Proposition~4.3 of \cite{lafleche_strong_2021}] \label{lem:K_x_commutator}
		Let $a \in (-1, \tfrac{3}{2})$, $p \in[1,\fb)$, and $q_0, q_1$ satisfying \eqref{eq:scaling_relation_1}. Then, for $n>a+1$, there exists a constant $C>0$ such that the estimate 
		\begin{equation*}
			\Nrm{[K_x, \op]}{p}\le C\,h^{1-\frac{3}{\fb}} \Nrm{\Dhv{\op}\, m}{q_0}^{\frac{1}{2}} \Nrm{\Dhv{\op}\,m}{q_1}^{\frac{1}{2}} 
		\end{equation*}
		holds. Here, $K_x$ denotes the multiplication operator $K_x(y) := K(x-y)$.
	\end{lem}
	Then we have the following result. 
	\begin{lem}\label{lem:bound_Q_tilde}
		Let $a\in(-1,\tfrac{3}{2})$ and $p\in[1,2]$ satisfying $p< \fb$. Then, for any $(\Psi_1,\Psi_2)\in \cG^2$, the following inequality holds
		\begin{equation}\label{eq:bound_Q_tilde}
			\frac{1}{\hbar} \Inprod{\Psi_1}{\tilde{\sfQ}^*\Psi_2}_{\cG} \leq \n{\kappa} \(\tilde{C}_1\, h^{3\(\frac{1}{2} - \frac{1}{p}\)} + \tilde{C}_2\, N^\frac{1}{2} h^\frac{3}{p'}\) \Nrm{\Psi_1}{\cG_{\frac{1}{2}}} \Nrm{\Psi_2}{\cG_{\frac{1}{p'}}},
		\end{equation}
		where $\tilde{C}_1 = C\, \tilde{\cD}_{q_0,q_1}$ and $\tilde{C}_2 = C\(Nh^3\cC_\infty\)^\frac{1}{2} \cD_{q_0,q_1}$ for some constant $C>0$ depending only on $a, p$ and $q_0$. 
	\end{lem}
	
	\begin{proof}
		Recall the definition of $\tilde{\sfQ}^*$ given in Formula~\eqref{eq:def_G}. By the anti-commutation relations \eqref{CAR}, the products of creation operators in $\tilde{\sfQ}^*$ can be written as follows
		\begin{align*}
			a_l^*(u_x)\, a_l^*(u_y)\, a_r^*(\conj{v_y})\, a_r^*(\conj{v_x}) &= a_l^*(u_x)\, a_r^*(\conj{v_x})\, a_l^*(u_y)\, a_r^*(\conj{v_y})
			\\
			a_r^*(\conj{u_x})\, a_r^*(\conj{u_y})\, a_l^*(v_y)\, a_l^*(v_x)
			&= a_l^*(v_x)\, a_r^*(\conj{u_x})\, a_l^*(v_y)\, a_r^*(\conj{u_y}).
		\end{align*}
		Moreover, using the notation \eqref{def:gen-one-part-op} and the notation $K_x(y) = K(x-y)$, we have $\DGplr{u\,K_x v} := \intd K(x-y)\, a_l^*(u_y)\,a_r^*(\conj{v_y})\d y$. Therefore, we can rewrite $\tilde{\sfQ}^*$ as
		\begin{equation}\label{eq:Q_tilde_rewrite}
			\tilde{\sfQ}^* = \frac{1}{2\,N} \intd \DGplr{u\,K_xv} \DGplr{u\,\delta_x v} - \DGplr{v\,\delta_x u} \DGplr{u\,K_x v} \d x.
		\end{equation}
		Here, $u\,\delta_x v$ denotes the operator with integral kernel $(u\,\delta_x v)(y,z) = u(y,x)\,v(x,z)$. 
		
		As in \cite[Proof of Proposition~4.3]{benedikter_mean-field_2016}, we need to exploit the hidden commutator structure in \eqref{eq:Q_tilde_rewrite} to handle the $\hbar^{-1}$ on the left-hand side of inequality \eqref{eq:bound_Q_tilde}. We begin by using the fact that $u$ commutes with $v$ to deduce the identity
		\begin{equation}\label{eq:commutator_id_1}
			u\,K_xv = v\,K_xu + u\com{K_x,v} - v\com{K_x,u} =: v\,K_x u + c_x,
		\end{equation}
		for any $x\in\Rd$. Moreover, the symmetry of $K$ allows us to write
		\begin{equation}\label{eq:symmetry_id_1}
			\intd \DGplr{v\,K_x u} \DGplr{u\,\delta_x v} \d x 
			= \intd \DGplr{v\,\delta_x u} \DGplr{u\,K_x v} \d x.
		\end{equation}
		By identities \eqref{eq:commutator_id_1}--\eqref{eq:symmetry_id_1}, we make appear more explicitly the commutator structure
		\begin{align*}
			\eqref{eq:Q_tilde_rewrite} &= \frac{1}{2N}\intd \DGplr{v\,K_x u + c_x} \DGplr{u\,\delta_x v} - \DGplr{v\,\delta_x u} \DGplr{u\,K_x v - c_x} \d x
			\\
			&= \frac{1}{2N}\intd \DGplr{c_x} \DGplr{u\,\delta_x v} + \DGplr{v\,\delta_x u} \DGplr{c_x} \d x.
		\end{align*}
		Again, using the fact that the creation operators anti-commute, we obtain
		\begin{multline*}
			\tilde{\sfQ}^* = \frac{1}{2\,N} \intd \(a_l^*(u_x)\,a_r^*(\conj{v_x}) + a_l^*(v_x)\,a_r^*(\conj{u_x})\) \DGplr{u\com{K_x,v}-v\com{K_x,u}} \d x.
		\end{multline*}
		Expanding the product inside the integral gives four terms. We define $\tilde{J}_1$ and $\tilde{J}_2$ as the terms with $\com{K_x,v}$, and $\tilde{J}_3$ and $\tilde{J}_4$ the terms with $\com{K_x,u}$. Let us look at $\tilde{J}_1$. By the Cauchy--Schwarz inequality, we obtain the following bound
		\begin{align*}
			\Inprod{\Psi_1}{\tilde{J}_1\Psi_2}_{\cG} &= \intd \Inprod{a_{l}(u_x)\,\Psi_1}{a_{r}^*(\conj{v_x})\,\DGplr{u\com{K_x ,v}} \Psi_2}_{\cG} \d x
			\\
			&\leq \(\intd \Nrm{a_{l}(u_x)\,\Psi_1}{\cG}^2\d x\)^\frac{1}{2} \(\intd\Nrm{a_{r}^*(\conj{v_x})\, \DGplr{u\com{K_x ,v}} \Psi_2}{\cG}^2 \d x\)^\frac{1}{2}.
		\end{align*}
		The first factor can be written
		\begin{equation*}
			\intd \Nrm{a_{l}(u_x)\,\Psi_1}{\cG}^2\d x = \Inprod{\Psi_1}{\intd a_l^*(u_x)\,a_l(u_x) \d x\,\Psi_1}_{\cG} = \Inprod{\Psi_1}{\DGl{1-\omega} \Psi_1}_{\cG},
		\end{equation*}
		which is smaller than $\Inprod{\Psi_1}{\cN_l\,\Psi_1}_{\cG}$. To estimate the second factor, we use the fact that
		\begin{equation*}
			\Nrm{a_r^*(\conj{v_x})}{\infty}^2 = \Nrm{v_x}{L^2}^2 = N \rho(x)
		\end{equation*}
		together with Lemma~\ref{lem:second_quantization_op} and the fact that $\Nrm{u}{\infty} \leq 1$ to get
		\begin{equation*}
			\Nrm{a_{r}^*(\conj{v_x}) \DGplr{u\com{K_x ,v}} \Psi_2}{\cG} \leq \(N\rho(x)\)^\frac{1}{2} \Nrm{\com{K_x ,v}}{p} \Nrm{\(\cN+2\)^\frac{1}{p'}\Psi_2}{\cG}.
		\end{equation*}
		Combining the above inequalities leads to
		\begin{equation*}
			\Inprod{\Psi_1}{\tilde{J}_1\Psi_2}_{\cG} \leq N^\frac{1}{2} \(\intd \Nrm{\com{K_x ,v}}{p}^2 \rho(x) \d x\)^\frac{1}{2} \Inprod{\Psi_1}{\cN\Psi_1}_{\cG}^\frac{1}{2} \Nrm{\(\cN+2\)^\frac{1}{p'}\Psi_2}{\cG}.
		\end{equation*}
		Applying Lemma \ref{lem:K_x_commutator}, since $p< \fb$, and the scaling relation \eqref{eq:scaling_relation_1}, we get that
		\begin{equation}\label{eq:commutator_K_v}
			\Nrm{\com{K_x ,v}}{p} \leq C\n{\kappa}N^\frac{1}{2}\,h^{1+3\(\frac{1}{2} - \frac{1}{p}\)} \tilde{\cD}_{q_0,q_1}.
		\end{equation}
		Therefore, we finally obtain the inequality
		\begin{equation*}
			\frac{1}{N\hbar} \Inprod{\Psi_1}{\tilde{J}_1\Psi_2}_{\cG} \leq C'\n{\kappa}\tilde{\cD}_{q_0,q_1} \,h^{3\(\frac{1}{2} - \frac{1}{p}\)} \Nrm{\Psi_1}{\cG_{\frac{1}{2}}} \Nrm{\Psi_2}{\cG_{\frac{1}{p'}}}.
		\end{equation*}
		The term $\tilde{J}_2$ is treated similarly leading to the same bound. 
		
		The terms $\tilde{J}_3$ and $\tilde{J}_4$ can also be treated in a similar manner. Except in this case, we apply Lemma~\ref{lem:nabla_u} and the fact that $\Nrm{v}{\infty} = \cC_\infty^{\frac{1}{2}} (Nh^3)^\frac{1}{2}$ to get
		\begin{equation}\label{eq:commutator_K_u}
			\Nrm{v\com{K_x,u}}{p} \leq C\n{\kappa}N^\frac{3}{2}\,h^{1+3\(\frac{3}{2} - \frac{1}{p} \)}\, \cC_\infty^{\frac{1}{2}}\, \cD_{q_0,q_1}.
		\end{equation}
		So, we obtained the claimed bound for $\tilde{\sfQ}^*$.
	\end{proof}

	\begin{remark}\label{remark:cutoff_2}
		In the case of the cut-off potential described in Remark~\ref{remark:cutoff}, we can take $q_0=q_1 = \infty$ and $p=2$ in the above inequality, with an extra factor $R^{3\(\frac{1}{2} - \frac{1}{\fb}\)}$, leading to
		\begin{equation}\label{eq:cutoff_Q}
			\frac{1}{\hbar} \Inprod{\Psi_1}{\tilde{\sfQ}^*\Psi_2}_{\cG} \leq \n{\kappa} R^{3\(\frac{1}{2} - \frac{1}{\fb}\)} \(\tilde{C}_1 + \tilde{C}_2\, N^\frac{1}{2} h^\frac{3}{2} \) \Nrm{\Psi_1}{\cG_{1/2}} \Nrm{\Psi_2}{\cG_{1/2}}.
		\end{equation}
        More precisely, Inequality \eqref{eq:cutoff_Q} is a direct consequence of the following estimate
        \begin{equation*}
            \frac{1}{\hbar} \Nrm{\com{K_{R,x},\op}}{\cL^2} \leq  C \n{\kappa} R^{3\(\frac{1}{2}-\frac{1}{\fb}\)} \Nrm{\Dhv\op\, m}{\L^\infty} 
        \end{equation*}
        which follows directly from the proof of Proposition~4.3 in \cite{lafleche_strong_2021}. 
	\end{remark}

	\subsubsection{Bounds for $\sfQ^*$}\label{sec:error_term} We label the terms of $\sfQ^*$ given in \eqref{eq:def_G} by
	\begin{equation}\label{eq:Q_form_1}
		\sfQ^* = I_1+I_2+I_3+I_4.
	\end{equation}
	Using the fact that the creation operators anti-commute, we get 
	\begin{align*}
		I_1 &= -\frac{1}{N} \intdd K(x-y) \,a_l^*(u_x)\,a_r^*(\conj{v_x})\, a_l^*(u_y)\, a_l(u_y)\d x\d y,
		\\
		I_2 &= -\frac{1}{N} \intdd K(x-y) \, a_l^*(v_x)\,a_r^*(\conj{u_x})\, a_l^*(v_y)\, a_l(v_y)\d x\d y.
	\end{align*}
	$I_3$ and $I_4$ have similar forms with the "$l$" and "$r$" labels interchanged and $(u,v)$ replaced by $(\conj{u},\conj{v})$. To reveal hidden commutator structures, which are necessary when estimating $\sfQ^*$ uniformly in $\hbar$, we need to further decompose \eqref{eq:Q_form_1}. 

	Let us start with the following decomposition lemma.
	\begin{lem}\label{lem:Q_decomposition}
		Let $\sfQ^*$ be as in \eqref{eq:Q_form_1}. Then, we have the decomposition
		\begin{equation*}
			I_1+I_2 = J_1+J_2+J_{12}+I_{12}
		\end{equation*}
		where 
		\begin{align*}
			J_1 &= \frac{1}{N} \intd a_r^*(u_x)\, a_l^*(\conj{v_x}) \DGl{u\com{u, K_x}} \d x
			\\
			J_2 &= \frac{1}{N} \intd a_l^*(v_x)\, a_r^*(\conj{u_x}) \DGl{v\com{v, K_x}} \d x
			\\
			J_{12} &= \frac{1}{N} \intd \DGplr{\com{u, K_x}v+\com{K_x, v}u} a_l^*(\omega_x)\, a_{x,l} \d x,
		\end{align*}
		and 
		\begin{equation*}
			I_{12} = -\frac{1}{N} \intd a^*_l(u_x)\, a^*_r(\conj{v_x})\DGl{K_x} \d x.
		\end{equation*}
		We have the same splitting for $I_3+I_4$, interchanging "l" by "r" and replacing $(u,v)$ by $(\conj{u},\conj{v})$. Hence, we obtain a decomposition of the form
		\begin{equation}\label{eq:def_P}
			\sfQ^* = \(J_1 + J_2 + J_3 + J_4 + J_{12} + J_{34}\) +(I_{12} + I_{34})=: \tilde{\sfP}^* + \sfP^*.
		\end{equation}
	\end{lem}

	\begin{proof}
		To simplify our computations, we use the Fefferman--de la Llave formula (cf. \cite{fefferman_relativistic_1986, hainzl_general_2002}) in its smooth version. For the potential $K$ it reads 
		\begin{equation*}
			K(x-y) = \kappa_a \int^\infty_0 \intd s^\frac{a+1}{2} \varphi_{s, z}(x)\, \varphi_{s, z}(y)\d z\d s
		\end{equation*}
		where $\varphi_{s,z}(x) = \varphi_s(x-z) = e^{-\pi\n{x-z}^2 s}$ and $\kappa_a = 2^{\frac{3-a}{2}}\frac{\pi^{a/2}}{\Gamma(a/2)}\,\kappa$.
		This allows us to rewrite $I_1$ and $I_2$ in the following forms 
		\begin{align*}
			I_1 &= -\frac{\kappa_a}{N} \int_0^\infty\!\!\!\!\intd s^{\frac{a+1}{2}} \DGplr{u\,\varphi_{s,z}\,v} \DGl{u\,\varphi_{s,z}\,u} \d z \d s,
			\\
			I_2 &= -\frac{\kappa_a}{N} \int_0^\infty\!\!\!\!\intd s^{\frac{a+1}{2}} \DGplr{v\,\varphi_{s,z}\,u} \DGl{v\,\varphi_{s,z}\,v} \d z \d s,
		\end{align*}
		where $\varphi_{s,z}$ is seen as a multiplication operator. Since $u^2=1-\omega$, then we have the identity 
		\begin{multline*}
			\DGplr{u\,\varphi\,v} \DGl{u\,\varphi \,u} =\ \DGplr{u\,\varphi\,v} \DGl{u\com{\varphi,u}} + \DGplr{\com{u,\varphi}v} \DGl{\(1-\omega\) \varphi}
			\\
			+ \DGplr{\varphi\,u\,v} \DGl{\(1-\omega\) \varphi}
		\end{multline*}
		where we used the notation $\varphi = \varphi_{s,z}$. Similarly, since $v^2=\omega$, then we have 
		\begin{multline*}
			\DGplr{v\,\varphi\,u} \DGl{v\,\varphi \,v} = \DGplr{v\,\varphi\,u} \DGl{v\com{\varphi,v}} + \DGplr{\com{v,\varphi}u} \DGl{\omega\,\varphi}
			\\
			+ \DGplr{\varphi\,u\,v} \DGl{\omega\,\varphi}.
		\end{multline*}
		Combining the two identities yields
		\begin{multline}\label{eq:commutator_id_2}
			\DGplr{u\,\varphi\,v} \DGl{u\,\varphi \,u} + \DGplr{v\,\varphi\,u} \DGl{v\,\varphi \,v}
			\\
			= \DGplr{u\,\varphi\,v} \DGl{u\com{\varphi,u}} + \DGplr{v\,\varphi\,u} \DGl{v\com{\varphi,v}}
			\\
			+ \DGplr{\com{\varphi, u} v + \com{v, \varphi} u} \DGl{\omega\, \varphi} + \DGplr{u\, \varphi\,v} \DGl{\varphi}.
		\end{multline}
		Thus, using identity \eqref{eq:commutator_id_2}, we can write $I_1+I_2 = J_1 + J_2 + J_{12} + I_{12}$ with
		\begin{align*}
			J_{1} &:= \frac{\kappa_a}{N} \int_0^\infty\!\!\!\!\intd s^{\frac{a+1}{2}} \DGplr{u\,\varphi\,v} \DGl{u\com{u, \varphi}} \d z \d s,
			\\
			J_{2} &:= \frac{\kappa_a}{N} \int_0^\infty\!\!\!\!\intd s^{\frac{a+1}{2}} \DGplr{v\,\varphi\,u} \DGl{v\com{v, \varphi}} \d z \d s,
			\\
			J_{12} &:= \frac{\kappa_a}{N} \int_0^\infty\!\!\!\!\intd s^{\frac{a+1}{2}} \DGplr{\com{u, \varphi}v+\com{\varphi, v}u} \DGl{\omega\,\varphi} \d z \d s,
		\end{align*}
		and
		\begin{equation*}
			I_{12} := -\frac{\kappa_a}{N} \int_0^\infty\!\!\!\!\intd s^{\frac{a+1}{2}} \DGplr{u\,\varphi\,v} \DGl{\varphi} \d z \d s.
		\end{equation*}
		Reversing the Fefferman--de la Llave expansion gives us 
		\begin{align*}
			J_1 &= \frac{\kappa_a}{N} \intd\int_0^\infty\!\!\!\!\intd s^{\frac{a+1}{2}} a^*_l(u_x)\, a^*_r(\conj{v_x})\,\varphi(x) \DGl{u\com{u, \varphi}} \d z \d s \d x
			\\
			&= \frac{1}{N}\intd a^*_l(u_x)\, a^*_r(\conj{v_x}) \DGl{u\com{u, K_x}} \d x.
		\end{align*}
		The same is true for $J_2$. Lastly, we have that 
		\begin{align*}
			J_{12} &= \frac{\kappa_a}{N} \int_0^\infty\!\!\!\!\intd s^{\frac{a+1}{2}} \DGplr{\com{u, \varphi}v+\com{\varphi, v}u} a_l^*(\omega_x)\,\varphi(x)\, a_{x,l} \d z \d s
			\\
			&= \frac{1}{N} \intd \DGplr{\com{u, K_x}v+\com{K_x, v}u} a_l^*(\omega_x)\, a_{x,l} \d x.
		\end{align*}
		This completes the proof of the lemma. 
	\end{proof}

	Let us first estimate the $J$ terms, which can be treated in a similar manner as in the $\tilde{\sfQ}^*$ case. One obtains the following bounds.
	\begin{lem}\label{lem:bound_P_tilde}
		Assuming the same hypotheses as in Lemma~\ref{lem:bound_Q_tilde}. Then, for any $(\Psi_1,\Psi_2)\in\cG^2$, we have the estimates
		\begin{subequations}
			\begin{align}\label{eq:bound_tilde_J_1}
				\frac{1}{\hbar} \Inprod{\Psi_1}{J_1\,\Psi_2}_{\cG} &\leq C_1 \n{\kappa} N^\frac{1}{2}\,h^{\frac{3}{p'}} \Nrm{\Psi_1}{\cG_{\frac{1}{2}}} \Nrm{\Psi_2}{\cG_{\frac{1}{p'}}}
				\\\label{eq:bound_tilde_J_2}
				\frac{1}{\hbar} \Inprod{\Psi_1}{J_2\,\Psi_2}_{\cG} &\leq C_2 \n{\kappa} N^\frac{1}{2} \,h^{\frac{3}{p'}} \Nrm{\Psi_1}{\cG_{\frac{1}{2}}} \Nrm{\Psi_2}{\cG_{\frac{1}{p'}}}
				\\\label{eq:bound_tilde_J_12}
				\frac{1}{\hbar} \Inprod{\Psi_1}{J_{12}\,\Psi_2}_{\cG} &\leq C_{12} \n{\kappa} N^\frac{1}{2}\,h^{\frac{3}{p'}}\,\Nrm{\Psi_1}{\cG_{\frac{1}{2}}} \Nrm{\Psi_2}{\cG_{\frac{1}{p'}}},
			\end{align}
		\end{subequations}
		where $C_1 = C \, \cD_{q_0,q_1}$, $C_2 = C\, \cC_\infty^\frac{1}{2}\,\tilde{\cD}_{q_0,q_1}$, and $C_{12} = C\, \cC_2\(\(Nh^3\cC_\infty\)^\frac{1}{2}\cD_{q_0,q_1}+\tilde{\cD}_{q_0,q_1}\)$ for some constant $C$ depending only on $p$ and $a$. The same inequalities hold respectively for $J_3$, $J_4$, $J_{34}$.
	\end{lem}

	\begin{proof}
	 Applying Lemma~\ref{lem:K_x_commutator} and the fact that $\Nrm{u}{\infty} \leq 1$ gives us the estimate
		\begin{equation*}
			\Nrm{u\com{K_x,u}}{p} \leq C\n{\kappa} N\,h^{1+\frac{3}{p'}} \cD_{q_0,q_1}.
		\end{equation*}
		Then following the proof of Lemma~\ref{lem:bound_Q_tilde} yields inequality~\eqref{eq:bound_tilde_J_1}. Similarly,
		by Lemma~\ref{lem:K_x_commutator} and the fact that $\Nrm{v}{\infty} = \(\cC_\infty\, N\,h^3\)^{\frac{1}{2}}$, we have that
		\begin{equation*}
			\Nrm{v\com{K_x ,v}}{p} \leq C\n{\kappa}N\,h^{1+\frac{3}{p'}}\, \cC_\infty^\frac{1}{2}\, \tilde{\cD}_{q_0,q_1}
		\end{equation*}
		from which we arrive at inequality~\eqref{eq:bound_tilde_J_2}. Finally, by direct estimation, we see that 
		\begin{equation*}
			\Inprod{\Psi_1}{J_{12}\,\Psi_2}_{\cG} \leq \frac{1}{N} \(\intd \Nrm{a_l(\omega_x)\, \DGmrl{v\com{K_x,u}-u\com{K_x,v}} \Psi}{\cG}^2 \d x\)^\frac{1}{2} \Nrm{\cN_l^{\frac{1}{2}}\Psi}{\cG}.
		\end{equation*}
		Then, Inequality~\eqref{eq:bound_tilde_J_12} follows from Lemma~\ref{lem:second_quantization_op} and Lemma~\ref{lem:K_x_commutator}.
	\end{proof}

	Lastly, let us estimate $\sfP^*=I_{12}+I_{34}$.

	\begin{lem}\label{lem:I_12}
		Let $p_a = \frac{3}{3-2a}$. Then, there exists $C>0$, depending only on $a$, such that for any $(\Psi_1, \Psi_2)\in \cG^2$ we have the estimate
		\begin{equation*}
			\frac{1}{\hbar}\n{\Inprod{\Psi_1}{\sfP^* \Psi_2}_{\cG}} \leq \frac{C\n{\kappa}}{N^\frac{1}{2}h} \Nrm{\rho}{L^{p_a}}^\frac{1}{2} \Inprod{\Psi_1}{\cN\Psi_1}_{\cG}^\frac{1}{2} \Inprod{\Psi_2}{\cN\Psi_2}_{\cG}^\frac{1}{2}.
		\end{equation*}
	\end{lem}
	\begin{proof}
		Let $(\Psi_1, \Psi_2)\in \cG^2$. By the Cauchy--Schwarz inequality and the boundedness of $a^*$, we have that 
		\begin{align*}
			\n{\Inprod{\Psi_1}{I_{12}\Psi_2}_{\cG}} &\le \frac{1}{N}\(\intd \Nrm{a_l(u_x)\Psi_1}{\cG}^2\d x\)^\frac{1}{2}\( \intdd \Nrm{a_r^*(\conj{v_x})\DGl{K_x}\Psi_2}{\cG}^2 \d x\)^\frac{1}{2}
			\\
			&\le \frac{1}{N^\frac{1}{2}} \Inprod{\Psi_1}{\cN_l\Psi_1}_{\cG}^\frac{1}{2} \(\intd \rho(x)\Nrm{\DGl{K_x} \Psi_2}{\cG}^2 \d x\)^\frac{1}{2}.
		\end{align*}
		where we used $\Nrm{v_x}{L^2}^2 = N\rho(x)$. Moreover, using the fact that 
		\begin{equation*}
			(\DGl{K_x}\Psi)^{(n, m)}(\underline{x}_{n}, \underline{y}_m) =\kappa\sum_{j=1}^n \frac{\Psi^{(n, m)}(\underline{x}_{n}, \underline{y}_m)}{\n{x-x_j}^a},
		\end{equation*}
		where $\underline{x}_{n}=(x_1,\dots,x_n)$, $\underline{y}_m=(y_1,\dots,y_m)$,
		it follows 
		\begin{equation*}
			\Nrm{(\DGl{K_x}\Psi)^{(n, m)}}{L^2(\RR^{3(n+m)})}^2 \le \n{\kappa}^2n^2\intd \frac{g(y)}{\n{x-y}^{2a}}\d y
		\end{equation*}
		where we defined $g(x) = \Nrm{\Psi^{(n,m)}(x,\underline{x}_{n-1},\underline{y}_{m})}{L^2(\d \underline{x}_{n-1}\d\underline{y}_{m})}^2$. Finally, by the Hardy--Littlewood--Sobolev inequality, we have that 
		\begin{align*}
			\intd \rho(x)\Nrm{(\DGl{K_x}\Psi)^{(n, m)}}{L^2(\RR^{3(n+m)})}^2 \d x &\le \n{\kappa}^2n^2 \intdd \frac{\rho(x)\,g(y)}{\n{x-y}^{2a}}\d x \d y
			\\
			&\le C_{p_a, a} \n{\kappa}^2 n^2 \Nrm{\rho}{L^{p_a}} \Nrm{g}{L^1}
		\end{align*}
		where $\Nrm{g}{L^1} = \Nrm{\Psi^{(n,m)}}{L^2}^2$. This yields the desired estimate. The proof for the estimate on $I_{34}$ is the same. 
	\end{proof}

\subsection{Proof of Proposition~\ref{prop:propagator_bound}}

	To control the growth of the fluctuation dynamics in the $\cG_k$ norms, the strategy consists of splitting the generator $\sfG$, defined in \eqref{eq:def_G}, into two parts then solving the problem perturbatively. More precisely, we define the splitting $\sfG = \tilde{\sfG} + \sfB$ with
	\begin{subequations}
		\begin{align}
			\tilde{\sfG} &= \DGl{H_\omega} - \DGr{\conj{H_\omega}} + \sfD + \tilde{\sfQ} + \tilde{\sfQ}^* + \tilde{\sfP} + \tilde{\sfP}^* 
			\label{eq:def_G_tilde}\\
			\sfB &= \sfP + \sfP^*
		\end{align}
	\end{subequations}
	where we recall $\sfP$ and $\tilde{\sfP}$ are defined by Formula~\eqref{eq:def_P}. The idea is to view $\sfG$ as a small perturbation of $\tilde\sfG$. This view is justifiable since, when $N^\frac{1}{2}\,h$ is large, the effect of the operator $\sfB$ is small in the following sense.
	\begin{lem}\label{lem:bound_for_B}
		Let $2\,j\in\N$ and $p_a = \frac{3}{3-2a}$. Then there exists a constant $C>0$ depending only on $a$ such that
		\begin{equation}\label{eq:bound_B}
			\frac{1}{\hbar}\Nrm{\sfB}{\cG_{j+\frac{3}{2}}\rightarrow\cG_j} \leq C\,\frac{2^j \n{\kappa}}{N^\frac{1}{2}h} \Nrm{\rho}{L^{p_a}}^\frac{1}{2}.
		\end{equation}
	\end{lem}
	
	\begin{proof}
		This follows from Lemma \ref{lem:I_12}. Notice that $\(\cN+1\)^k \sfP^* = \sfP^* \(\cN+3\)^k$, then, by Lemma \ref{lem:I_12}, we have that
		\begin{align*}
			\Nrm{\sfP^* \Psi}{\cG_j} &= \sup_{\Nrm{\Psi_1}{\cG} \le 1}\Inprod{\(\cN+1\)^{-\frac{1}{2}}\Psi_1}{\sfP^*\(\cN+3\)^{j+\frac{1}{2}} \Psi}_{\cG}\\
			&\le \frac{C_a \n{\kappa}}{N^\frac{1}{2}}\Nrm{\rho}{L^{p_a}}^\frac{1}{2}\Nrm{\(\cN+3\)^{j+\frac{1}{2}}\cN\Psi}{\cG} \le C_a\frac{2^j\n{\kappa}}{N^\frac{1}{2}}\Nrm{\rho}{L^{p_a}}^\frac{1}{2}\Nrm{\Psi}{\cG_{j+\frac{3}{2}}}.
		\end{align*}
		The estimate for $\sfP$ also follows immediately from Lemma \ref{lem:I_12}, that is,
		\begin{align*}
			\Nrm{\sfP \Psi}{\cG_k} &= \sup_{\Nrm{\Psi_1}{\cG} \le 1}\Inprod{\sfP^*\(\cN+1\)^{-1}\Psi_1}{\(\cN-1\)^{j+1} \Psi}_{\cG}\\\
			&\le \frac{C_a\n{\kappa}}{N^\frac{1}{2}}\Nrm{\rho}{L^{p_a}}^\frac{1}{2}\Nrm{(\cN-1)^{j+1}\cN^\frac{1}{2}\Psi}{\cG} \le \frac{C_a\n{\kappa}}{N^\frac{1}{2}}\Nrm{\rho}{L^{p_a}}^\frac{1}{2}\Nrm{\Psi}{\cG_{j+\frac{3}{2}}}.
		\end{align*}
		This completes the argument. 
	\end{proof}
	
	 In light of the above lemma, we define the auxiliary dynamics $\tilde\sfU_{t, s}$ to be the unitary dynamics generated by \eqref{eq:def_G_tilde}, that is, for any $(t,s)\in\R^2$, $\tilde\sfU_{t, s}$ satisfies the differential equation 
	\begin{equation}\label{eq:def_auxiliary_dynamics}
		i\hbar\,\dpt \tilde{\sfU}_{t,s}\Psi = \tilde{\sfG}_t \tilde{\sfU}_{t,s} \Psi \ \text{ with }\ \tilde\sfU_{s,s}\Psi= \Psi
	\end{equation}
	for $\Psi$ sufficiently smooth. The existence of $\tilde{\sfU}_{t,s}$ is proven in Appendix \ref{sec:auxiliary_dynamics_appendix}. Let us begin by showing the auxiliary dynamics propagates the $\cG_k$ norm under regularity assumptions on the solution of the Hartree equation.
	
	\begin{prop}
		Let $a\in(-1,\tfrac{3}{2})$, $p\in[1,2]$ satisfying $p< \fb = \frac{3}{a+1}$, and $\Psi_t=\tilde\sfU_{t,s}\Psi$ is a solution to \eqref{eq:def_auxiliary_dynamics}. Then, for any $k$ such that $2k \in\N$, the inequality
		\begin{equation*}
			\dt \Nrm{\Psi_t}{\cG_k} \leq C_k \n{\kappa} \(\tilde{\cD}_{q_0,q_1}\, h^{3\(\frac{1}{2} - \frac{1}{p}\)} + C_{\op}\, N^\frac{1}{2}h^\frac{3}{p'}\) \Nrm{\Psi_t}{\cG_{k+\(\frac{1}{2}-\frac{1}{p}\)}}
		\end{equation*}
		holds for some constants $C_k$ of the form $C_k = C_{p,a,q_0}\, C^k$ and
		\begin{equation*}
			C_{\op} = \(1+\cC_\infty^\frac{1}{2}\)\(\cD_{q_0,q_1} + \tilde{\cD}_{q_0,q_1}\)
		\end{equation*}
		where $\cD_{q_0,q_1}$ and $\tilde{\cD}_{q_0,q_1}$ are defined by formulas~\eqref{eq:def_Dqq} and \eqref{eq:def_Dqq_tilde}.
	\end{prop}
	
	\begin{remark}\label{remark:bound_for_U_tilde}
		Since $N^\frac{1}{2}\,h^\frac{3}{2} \leq \cC_\infty^{-\frac{1}{2}}$, then, by Gr\"onwall's inequality, we deduce that
		\begin{equation}\label{eq:bound_U_tilde}
			\Nrm{\tilde{\sfU}_{t,s}}{\cG_k\rightarrow \cG_k} \leq e^{C_{t,s}}
		\end{equation}
		where $C_{t,s} = C_k\n{\kappa} h^{-\alpha} \(1+\cC_\infty^{-\frac{1}{2}}\) \int_s^t \cD_{q_0,q_1}(\tau) + \tilde{\cD}_{q_0,q_1}(\tau)\d\tau$ with $\alpha := \frac{3}{p} - \frac{3}{2} \geq 0$ and is $0$ if and only if $p=2$.
	\end{remark}
	
	\begin{proof}
		Let $k\in\N$. Using the fact that $\tilde{\sfG} = \tilde{\sfG}^*$, we get 
		\begin{align*}
			i\hbar\,\dt \Nrm{\Psi_t}{\cG_\frac{k}{2}}^2 &= \Inprod{\Psi_t}{\com{\(\cN+1\)^k,\tilde{\sfG}}\Psi_t}_{\cG}
			\\
			&= \sum_{j=1}^k \Inprod{\Psi_t}{\(\cN+1\)^{j-1}\com{\cN,\tilde{\sfG}}\(\cN+1\)^{k-j}\Psi_t}_{\cG}.
		\end{align*}
		Note that the only terms in $\tilde{\sfG}$ that do not commute with $\cN$ are $\tilde{\sfQ}$, $\tilde{\sfP}$, and their adjoints. Since $\cN_{\sigma} a^*_{\sigma} = a^*_{\sigma} \(\cN_{\sigma}+1\)$ for $\sigma \in \{r, l\}$, we obtain
		\begin{equation*}
			\com{\cN,\tilde{\sfG}} = \com{\cN,\tilde{\sfQ}^*+\tilde{\sfQ}+\tilde{\sfP}^*+\tilde{\sfP}} = 4\(\tilde{\sfQ}^*-\tilde{\sfQ}\) + \tilde{\sfP}^*-\tilde{\sfP},
		\end{equation*}
		which leads to
		\begin{equation}\label{eq:dt_N_auxiliary_dynamic}
			\dt \Nrm{\Psi_t}{\cG_\frac{k}{2}}^2 = \frac{2}{\hbar}\Im{\sum_{j=1}^k \Inprod{\Psi_t}{\(\cN+1\)^{j-1}\(4\tilde{\sfQ}^*+\tilde{\sfP}^*\)\(\cN+1\)^{k-j}\Psi_t}_{\cG}}.
		\end{equation}
		Using again the commutation relation between the number operator and the creation operator, we can balance the power of the number operators appearing on the left and on the right of $\tilde{\sfQ}^*$. More precisely, if $j>\frac{k+1}{2}$, then
		\begin{align*}
			\(\cN+1\)^{j-1}\tilde{\sfQ}^*\(\cN+1\)^{k-j} &= \(\cN+1\)^\frac{k-1}{2}\tilde{\sfQ}^*\(\cN+5\)^{j-\frac{k+1}{2}} \(\cN+1\)^{k-j},
			\\
			\(\cN+1\)^{j-1}\tilde{\sfQ}^*\(\cN+1\)^{k-j} &= \(\cN+1\)^\frac{k-1}{2}\tilde{\sfQ}^*\(\cN+2\)^{j-\frac{k+1}{2}} \(\cN+1\)^{k-j},
		\end{align*}
		and similarly if $j<\frac{k+1}{2}$, using the fact that $\tilde{\sfQ}^*\(\cN+1\)^s = \(\cN-3\)^s \tilde{\sfQ}^*$. Therefore, applying Lemma~\ref{lem:bound_Q_tilde} and Lemma~\ref{lem:bound_P_tilde} on each term of the right-hand side of Equation~\eqref{eq:dt_N_auxiliary_dynamic} and the fact that $Nh^3\cC_\infty\leq 1$ and $\cC_2\leq \cC_\infty^\frac{1}{2}$, we obtain
		\begin{equation*}
			\frac{1}{2}\,\dt \Nrm{\Psi_t}{\cG_\frac{k}{2}}^2 \leq C^{k} \n{\kappa} C_{p,a,q_0} \(\tilde{\cD}_{q_0,q_1}\, h^{3\(\frac{1}{2} - \frac{1}{p}\)} + C_{\op}\, N^\frac{1}{2}h^\frac{3}{p'}\) \Nrm{\Psi_t}{\cG_\frac{k}{2}}\Nrm{\Psi_t}{\cG_{\frac{k}{2}+\(\frac{1}{2}-\frac{1}{p}\)}}
		\end{equation*}
		which leads to the desired result.
	\end{proof}

	Moreover, by Proposition~\ref{lem:bogoliubov_N_bound} and by weighted interpolation, we deduce that for any $t>0$, $\sfR_{\op_t}$ is a bounded mapping from $\cG_k$ to $\cG_k$. More precisely, for any $k\in [0,1/2]$, we have a bound of the form~\eqref{eq:bogoliubov_N_bound_small_s}, and the same bound is valid for $\sfR_{\op}^*$. Therefore, recalling that by definition $\sfU_{t,s} = \sfR_{\op_t}^*\, e^{-i\sfL (t-s)/\hbar} \, \sfR_{\op_s}$, and since $e^{-i\sfL (t-s)/\hbar}$ commutes with the number operator, we obtain for $k\in [0,1/2]$
	\begin{equation}\label{eq:bound_U_0}
		\Nrm{\sfU_{t,s}\Psi}{\cG_k} \leq 3 \Nrm{\Psi}{\cG_k} + 5\,N^k \Nrm{\Psi}{\cG}.
	\end{equation}
	
	Combining these three inequalities~\eqref{eq:bound_U_tilde}, \eqref{eq:bound_B}, \eqref{eq:bound_U_0}, and using the Duhamel formula, we obtain the main result of the section.
	\begin{proof}[Proof of Proposition~\ref{prop:propagator_bound}]
		Let $\sfB_h := \frac{1}{i\hbar} \sfB$. Then the Duhamel formula can be written
		\begin{equation*}
			\sfU_{t,0} = \tilde{\sfU}_{t,0} + (\sfU \star \sfB_h \tilde{\sfU})_{t,0}
		\end{equation*}
		where we used the notation $\star$ for the time convolution of operators 
		\begin{equation*}
			(\sfU \star \sfV)_{t,s} := \int_s^t \sfU_{t,s'}\sfV_{s',s}\d s',
		\end{equation*}
		We now define the iterated time convolution $\sfU^{(\star k)}$ by $\sfU^{(\star 1)} = \sfU$ for $k=1$ and by $\sfU^{(\star k)} = \sfU \star \sfU^{(\star (k-1))}$ for $k\geq 2$. With these notations, one can write the following iterated Duhamel's formula
		\begin{equation}\label{eq:iterated_Duhamel}
			\sfU_{t,0} = \sum_{j=0}^{k-1} (\tilde{\sfU} \star (\sfB_h \tilde{\sfU})^{(\star j)})_{t,0} + (\sfU \star(\sfB_h \tilde{\sfU})^{(\star k)})_{t,0}\,.
		\end{equation}
		and from Inequality~\eqref{eq:bound_U_0}, we deduce
		\begin{multline*}
			\Nrm{\sfU_{t,0}\Psi}{\cG_{k_0}} \leq \sum_{j=0}^{k-1} \Nrm{(\tilde{\sfU} \star (\sfB_h \tilde{\sfU})^{(\star j)})_{t,0}\Psi}{\cG_{k_0}}
			\\
			+ \int_0^t \( 3\Nrm{(\sfB_h \tilde{\sfU})^{(\star k)}_{s,0}\Psi}{\cG_{k_0}} + 5N^{k_0} \Nrm{(\sfB_h \tilde{\sfU})^{(\star k)}_{s,0}\Psi}{\cG}\)\d s.
		\end{multline*}
		Since we know from Part \ref{part:regularity} that $C_T := \sup_{[0,T]}(\Nrm{\rho}{L^{p_a}}, \cD_{q_0,q_1}, \tilde{\cD}_{q_0,q_1})$ is bounded, we deduce that $C_{t,s} \leq C_T\,C^{k_0}\, C_{p,a,q_0} \n{\kappa} h^{-\alpha} \(1+\cC_\infty^{-\frac{1}{2}}\) \(t-s\) =: \lambda_\alpha\,C^{k_0} \(t-s\)$.
		
		From inequalities~\eqref{eq:bound_B} and~\eqref{eq:bound_U_tilde}, we obtain for any $0\leq s\leq t \leq T$
		\begin{equation*}
			\Nrm{\(\sfB_h \tilde{\sfU}\)_{t,s}}{\cG_{k_0+\frac{3}{2}}\rightarrow \cG_{k_0}} \leq \frac{2^{k_0}\,C\,\lambda_0}{M}\, e^{\lambda_\alpha \,C^k \(t-s\)}
		\end{equation*}
		where $M = N^\frac{1}{2} h$, which leads to
		\begin{equation*}
			\Nrm{\(\sfB_h \tilde{\sfU}\)^{(\star j)}_{t,s}}{\cG_{k_0+\frac{3}{2}j}\rightarrow\cG_{k_0}} \leq \frac{\(C\lambda_0\)^j 2^{\(k_0+j\)\(j+1\)}}{M^j\(j-1\)!}\(t-s\)^{j-1} e^{\lambda_\alpha \, C^k \(t-s\)}\,.
		\end{equation*}
		Hence, for $\sfU$, we obtain
		\begin{align*}
			\Nrm{\sfU_{t,0}\Psi}{\cG_{k_0}} &\leq \sum_{j=0}^{k-1} \frac{\(C\lambda_0\)^j 2^{\(k_0+j\)\(j+1\)}}{M^j\,j!}\, t^j e^{\lambda_\alpha \, C^k \,t} \Nrm{\Psi}{\cG_{k_0+\frac{3}{2}j}}
			\\
			&\qquad+ \frac{\(2^k C\lambda_0\)^k}{M^k\,k!\,\lambda_\alpha}\, t^k e^{\lambda_\alpha \, C^k t} \(2^{k_0\(k+1\)}\Nrm{\Psi}{\cG_{k_0+\frac{3}{2}k}} + N^{k_0}\Nrm{\Psi}{\cG_{\frac{3}{2}k}}\)
			\\
			&\leq C_M\, e^{C_M\,\lambda_\alpha\,t} \Nrm{\Psi}{\cG_{k_0+\frac{3}{2}k}} + \frac{\(2^k C\)^k\lambda_0^{k-1}}{k!}\, t^k e^{\lambda_\alpha \, C^k t} \,\frac{N^{k_0}\,h^{3\alpha}}{M^k}\Nrm{\Psi}{\cG_{\frac{3}{2}k}}
		\end{align*}
		where $C_M = C^{k+k_0}\(1+\frac{1}{M}\)$. Observing that $\lambda_0 = h^{\alpha}\lambda_\alpha$, this implies
		\begin{multline*}
			\Nrm{\sfU_{t,0}\Psi}{\cG_{k_0}} \leq C_M\, e^{C_M\,\lambda_\alpha\,t} \Nrm{\Psi}{\cG_{k_0+\frac{3}{2}k}}
			\\
			+ \frac{2^k C}{k} \frac{\(2^k C\)^{k-1}\lambda_\alpha^{k-1}}{(k-1)!}\, t^{k-1} e^{\lambda_\alpha \, C^k t} \, \frac{N^{k_0}\,h^{\alpha k}\,t}{M^k}\Nrm{\Psi}{\cG_{\frac{3}{2}k}}
		\end{multline*}
		and using the fact that for $x>0$, $\frac{x^{k-1}}{(k-1)!} \leq e^x$, replacing the constant $C^k + 2^k C$ by $C^k$ for some other numerical constant $C$ in the second exponential and bounding $2^k C/k$ by $C^k$, we can simplify a bit the result and write
		\begin{align*}
			\Nrm{\sfU_{t,0}\Psi}{\cG_{k_0}} &\leq C_M\, e^{C_M\,\lambda_\alpha\,t} \Nrm{\Psi}{\cG_{k_0+\frac{3}{2}k}} + C^k \, e^{\lambda_\alpha \, C^k t}\, \frac{N^{k_0}\,h^{\alpha k}\,t}{M^k}\Nrm{\Psi}{\cG_{\frac{3}{2}k}}
			\\
			&\leq C_M\, e^{C_M\,\lambda_\alpha\,t} \(\Nrm{\Psi}{\cG_{k_0+\frac{3}{2}k}} + \frac{h^{\(\alpha-1\) k}}{N^{\frac{k}{2}-k_0}} \, t \Nrm{\Psi}{\cG_{\frac{3}{2}k}}\).
		\end{align*}
	\end{proof}

\section{Proofs of Theorem \ref{thm:mean_field_2} and Theorem \ref{thm:mean_field}}

	We can now prove our general theorem.
	\begin{proof}[Proof of Theorem~\ref{thm:mean_field_2}]
		In order to get the result, we want to apply Proposition~\ref{prop:propagator_bound}. Hence, we define
		\begin{equation*}
			\frac{1}{p_\alpha} := \frac{\alpha}{3} + \frac{1}{2}.
		\end{equation*}
		The assumptions $\alpha \in \lt[0,1\rt]$ and $\alpha > a - \frac{1}{2}$ are equivalent to say that $p_\alpha \in \lt[\tfrac{5}{6},2\rt] \text{ and } p_\alpha < \fb$, and imply that \eqref{eq:scaling_q_2} is a non-empty condition. Therefore, $p_\alpha$ satisfies the assumption~\eqref{eq:scaling_p_b}. Now we define
		\begin{equation}\label{eq:q0_q1_def}
			q_1 := q \quad \text{ and } \quad \frac{1}{q_0} := 2 \(\frac{1}{p_\alpha} - \frac{1}{\fb}\) - \frac{1}{q_1}
		\end{equation}
		so that Formula~\eqref{eq:scaling_relation_1} holds with $p= p_\alpha$. Assumption~\eqref{eq:scaling_q_2} can be written
		\begin{equation}\label{eq:q0_q1_prop}
			\frac{1}{q_1} \in \lt[2\(\frac{1}{p_\alpha} - \frac{1}{\fb}\)-\frac{1}{2},\frac{1}{p_\alpha} - \frac{1}{\fb}\rt].
		\end{equation}
		Definition~\eqref{eq:q0_q1_def} and Equation~\eqref{eq:q0_q1_prop} together imply that $q_0$ and $q_1$ satisfy $2\leq q_0 < q_1 \leq \infty$.
		
		Next, we have to check that we have a uniform in $h$ bound for the quantity
		\begin{equation*}
			\sup_{[0,T]}\Big(\Nrm{\rho(t)}{L^{p_a}}, \cD_{q_0,q_1}(t), \tilde{\cD}_{q_0,q_1}(t)\Big)
		\end{equation*}
		appearing in the growth rate $\lambda_\alpha$ defined in \eqref{eq:lambda-alpha}. This is done by using the propagation of regularity results for the Hartree--Fock equation of Part~\ref{part:regularity}. First, by our initial regularity assumptions and Proposition~\ref{prop:regu_HF}, we deduce that $\Nrm{\rho(t)}{L^{p_a}}$ is bounded uniformly in $h$ and in $t\in[0,T]$ for some $T = T_{\op^\init}$ depending on the initial condition of the Hartree--Fock Equation \eqref{eq:Hartree--Fock}. Then, by Proposition~\ref{prop:regu_HF_sqrt} we deduce that $\sqrt{\op}\in \cW^{1,q}(m_n)$ for any $q\in[2,q_1)$, and so
		\begin{equation*}
			\tilde{\cD}_{q_0,q_1} = \Nrm{\Dhv{\sqrt{\op}}\,m_n}{\L^{q_0}}^{\frac{1}{2}} \Nrm{\Dhv{\sqrt{\op}}\,m_n}{\L^{q_1}}^{\frac{1}{2}}
		\end{equation*}
		is uniformly bounded on $[0,T]$. Moreover, by Lemma~\ref{lem:regu_rho_vs_sqrt}, we obtain
		\begin{equation*}
			\cD_{q,q_0} = \Nrm{\Dhv{\op}\,m_n}{\L^{q_0}}^{\frac{1}{2}} \Nrm{\Dhv{\op}\,m_n}{\L^{q_1}}^{\frac{1}{2}} \leq \tilde{\cD}_{q_0,q_1}
		\end{equation*}
		so $\cD_{q,q_0}$ is also uniformly bounded on $[0,T]$. Therefore, by Proposition~\ref{prop:propagator_bound}, we deduce that
		\begin{equation}\label{eq:propagator_bound_2}
			\Nrm{\sfU_{t,0}\Psi}{\cG_\frac{1}{2p}}^2 \leq C_M^2\, e^{C_M\,\lambda \,t/h^{\alpha}} \(\Nrm{\Psi}{\cG_{\frac{3 k}{2} + \frac{1}{2p}}}^2 + \frac{h^{2k\(\alpha-1\)}}{N^{k-\frac{1}{p}}} \, t^2 \Nrm{\Psi}{\cG_\frac{3 k}{2}}^2\)
		\end{equation}
		with $\lambda$ uniformly bounded in $t\in[0,T]$ and in the Planck constant $h$. Then, by Proposition~\ref{prop:degree_evaporation_vs_Sp}, the following inequality holds
		\begin{equation*}
			\Nrm{\op_{N:1} - \op}{\L^p} \leq \frac{C_p}{\min(N^{\frac{1}{2}},N\, h^{\frac{3}{p'}})} \Nrm{\Psi}{\cG_\frac{1}{2p}}^2.
		\end{equation*}
		We conclude the proof by combining Inequality~\eqref{eq:propagator_bound_2} with the above inequality.
	\end{proof}
	
	Next, we prove that from our general Theorem~\ref{thm:mean_field_2}, we can deduce our simplified mean-field results, i.e. Theorem~\ref{thm:mean_field}. To this end, we come back to the setting of density operators over the Fock space by the following Lemma

	\begin{lem}\label{lem:bound_evaporation_vs_L1}
		Let $\opNop := \n{\sfI_\cG^{-1}(\sfR_{\op}\,\Omega)}^2$ as defined in \eqref{eq:reference_state}. Then for any $\opN \in \L^1_s(\cF)$, that commutes with $\cN$, there exists $\Psi\in\cG$ such that 
		\begin{equation}\label{eq:correspondence}
			\opN = \n{\sfI_\cG^{-1}(\sfR_{\op}\,\Psi)}^2
		\end{equation}
		and
		\begin{equation}\label{eq:bound_evaporation_vs_L1}
			\Nrm{\cN^s\Psi}{\cG}^2 \leq C_s \Nrm{\(\cN+2N\)^s\(\opN - \opNop\)}{\L^1(\cF)}
		\end{equation}
		with $C_s = 12^s\(s+1\)^s$.
	\end{lem}
	
	\begin{proof}
		Let $\Phi_{\op} = \sfR_{\op}\Omega = \sfI_\cG(\opupNop)$. Then $\n{\opupNop} = \sqrt{\opNop}$, and by the polar decomposition of $\opupNop$, there exists a unique operator $U_{N,\op}$ such that 
		\begin{equation*}
			\opupNop = U_{N,\op} \n{\opupNop},
		\end{equation*}
		with $\Nrm{U_{N,\op}\psi}{\cF} = \Nrm{\psi}{\cF}$ if $\psi \in \(\ker{\opupNop}\)^{\perp}$ and $\Nrm{U_{N,\op}\psi}{\cF} = 0$ if $\psi \in \ker{\opupNop}$ (see e.g. \cite[Theorem~VI.10]{reed_functional_1980}). Then we define
		\begin{equation*}
			\opupN := U_{N,\op} \n{\sqrt{\opN}},
		\end{equation*}
		and $\Phi = \sfI_\cG(\opupN) $, $\Psi := \sfR_{\op}^*\Phi$. In particular, $\opN = \n{\opupN}^2$ so Formula~\eqref{eq:correspondence} is satisfied. Now from Lemma~\ref{lem:bogoliubov_N_bound}, we have
		\begin{equation*}
			\Nrm{\cN^s\Psi}{\cG} = \Nrm{\cN^s(\Psi-\Omega)}{\cG} \leq 3^s \Nrm{\(\cN+2N+2s+2\)^s \(\Phi-\Phi_{\op}\)}{\cG}.
		\end{equation*}
		Using the fact that $\sfI_\cG$ is an isometry, $\cN_l\Phi = \sfI_\cG(\cN\opupN)$, $\cN_r\Phi = \sfI_\cG(\opupN\cN)$ and $\opupN$ commutes with $\cN$, we deduce that
		\begin{equation*}
			\Nrm{\cN^s\Psi}{\cG} \leq C_s \Nrm{\(\cN+N\)^s \(\opupN-\opupNop\)}{\L^2(\cF)}.
		\end{equation*}
		By our choice of $U_{N,\op}$, it holds
		\begin{equation*}
			\(\cN+N\)^s \(\opupN-\opupNop\) = U_{N,\op}\(\n{\(\cN+N\)^s\opupN}-\n{\(\cN+N\)^s\opupNop}\)
		\end{equation*}
		with $\Nrm{U_{N,\op}}{\infty} \leq 1$. Now recall the Powers--St{\o}rmer inequality \cite[Lemma~4.1]{powers_free_1970} which tells that if $A$ and $B$ are nonnegative operators, then $\tr\!\big(\n{A-B}^2\big)\leq \Tr{\n{A^2-B^2}}$. Hence,
		\begin{align*}
			\Nrm{\cN^s\Psi}{\cG}^2 &\leq C_s^2 \Nrm{\n{\opupN \(\cN+N\)^s}^2-\n{\opupNop\(\cN+N\)^s}^2}{\L^1(\cF)}
			\\
			&\leq C_s^2 \Nrm{\(\cN+N\)^{s}\(\opN-\opNop\)\(\cN+N\)^{s}}{\L^1(\cF)}.
		\end{align*}
	\end{proof}

	\begin{proof}[Proof of Theorem~\ref{thm:mean_field}]
		In the setting of Theorem~\ref{thm:mean_field}, since $a<1/2$, we can take $\alpha = 0$ in Theorem~\ref{thm:mean_field_2}, and the hypothesis for $q$ implies that condition~\eqref{eq:scaling_q_2} is satisfied. With this choice, Theorem~\ref{thm:mean_field_2} yields for any $k_1\in\N$
		\begin{equation*}
			\Nrm{\op_{N:1}-\op}{\L^p} \leq \frac{C\,e^{\lambda\,t}}{\min(N^\frac{1}{2}, N\,h^\frac{3}{p'})} \Nrm{\Psi}{\cG_{\frac{3}{2} k_1+\frac{1}{2p}}}^2 \(1+ \frac{h^{-2 k_1}}{N^{k_1-\frac{1}{p}}}\).
		\end{equation*}
		Taking $k = \frac{3}{2} k_1+\frac{1}{2p}$, the hypothesis on $k$ implies that $\frac{h^{-2 k_1}}{N^{k_1-\frac{1}{p}}}\leq C$. Finally, we use Lemma~\ref{lem:bound_evaporation_vs_L1} to get
		\begin{equation*}
			\Nrm{\Psi}{\cG_k}^2 \leq 2^{k+1}\(\Nrm{\Psi}{\cG}^2+\Nrm{\cN^k\Psi}{\cG}^2\) \leq C_k \(1+\Nrm{\opN - \opNop}{\L^1_k(\cF)}\)
		\end{equation*}
		for some $k$ dependent constant $C_k>0$.
	\end{proof}

\appendix

\section{Existence of the Auxiliary Dynamics}\label{sec:auxiliary_dynamics_appendix}

	The purpose of this appendix is to extend the result on the existence of the auxiliary dynamics for smooth potentials in the interaction picture given in the appendix of~\cite{benedikter_mean-field_2016} to the case of singular interaction potentials of the form $K(x)=\n{x}^{-a}$ for $0 \le a \le 1$. 
	
	In this section, $\hbar$ will not play any role in our analysis. Therefore, to simplify the presentation, we set $\hbar\equiv 1$. By Formula~\eqref{eq:def_P}, the time-dependent operator $\tilde{\sfG}$ defined in Equation~\eqref{eq:def_G} can be written
	\begin{equation}\label{eq:aux_gen}
		\tilde \sfG = 
		\DGl{H_\omega} - \DGr{\overline{H_\omega}} + \tilde \sfQ + \tilde \sfQ^* + \sfD + 
		\tilde \sfP + \tilde \sfP^* 
	\end{equation}
	where $\tilde{\sfQ}^*$ and $\sfD$ are already defined after Equation~\eqref{eq:def_G} and
	\begin{align*}
		\tilde \sfP^* &= \frac{1}{N} \intdd \bigg( a^*_r(\bar v_x) \, a_l^*(u_x)\DGl{u\com{K_x, u}} + a^*_r(\bar u_x) \, a_l^*(v_x)\DGl{v\com{K_x, v}}
		\\
		&\qquad\qquad + \DGplr{\com{u, K_x}v +\com{K_x, v}u} \DGl{\omega_x}
		\\
		&\qquad\qquad + a^*_l(v_x) \, a_r^*(\bar u_x) \DGr{\bar u\com{K_x,\bar u}} + a^*_l(u_x) \, a_r^*(\bar v_x) \DGr{\bar v \com{K_x, \bar v}}
		\\
		&\qquad\qquad + \DGprl{\com{\bar u, K_x}\bar v + \com{K_x, \bar v}\bar u} \DGr{\bar \omega_x}\bigg)\d x.
	\end{align*}
	The goal is to show that the operator $\tilde{\sfG}$ generates a unitary dynamics $\tilde \sfU_{t, s}$ in Fock space that satisfies the differential equation
	\begin{equation}\label{aux_eq}
		i\,\dpt \tilde \sfU_{t, s}\Psi = \tilde \sfG_t\tilde \sfU_{t, s}\Psi
		\ \text{ with } \ \tilde \sfU_{s, s}\Psi = \Psi,
	\end{equation}
	for sufficiently smooth $\Psi \in \cG$. To this end, it is convenient to consider the dynamics in the interaction picture. More precisely, define the operator
	\begin{equation*}
		\widehat \sfG_t = -\sfL_0 + \sfU^{(0)*}_t \tilde\sfG_t\sfU^{(0)}_t
	\end{equation*}
	where $\sfL_0 = \DGr{\lapl} - \DGl{\lapl}$ and $\sfU^{(0)}_t = \sfU^{(0)}_{t,0}$ is the free evolution, i.e. $\sfU^{(0)}_{t,s}$ solves
	\begin{equation*}
			i\,\dpt \sfU^{(0)}_{t, s} \Psi = \sfL_0 \sfU^{(0)}_{t, s} \Psi, 
	\end{equation*}
	with $\sfU_{s, s}^{(0)}\Psi = \Psi$. We will show that $\widehat\sfG_t$ generates a unitary operator $\widehat \sfU_{t, s}$ in Fock space which in turn allows us to define the auxiliary dynamics by
	\begin{equation*}
		\tilde \sfU_{t, s} := \sfU^{(0)}_t\widehat \sfU_{t, s} 
		\sfU^{(0)*}_s.
	\end{equation*}
	which formally satisfies Equation~\eqref{aux_eq}. 
	
	Since much of the result in this appendix is similar to that of the appendix of \cite{benedikter_mean-field_2016}, we will only focus on the part of the result that relies explicitly on the regularity of the potential and refer the reader to \cite{benedikter_mean-field_2016} for a more complete proof of the result. Hence, the rest of this section will be devoted to prove that the mapping $t\mapsto \widehat \sfG_t\Psi$ is H\"older continuous when $\Psi$ is sufficiently smooth. More precisely, we define the homogeneous Sobolev-type double Fock space by the norm
	\begin{equation}\label{eq:Fock_Sobolev_norm}
		\Lp{\Psi}{\dot{\cH}^s_k} := \Nrm{\cN^{k-1/2} \dG((-\Delta)^s)^{1/2}\Psi}{\cG}.
	\end{equation}
	In particular, $\Lp{\Psi}{\dot{\cH}^0_k} = \Lp{\cN^k\Psi}{\cG}$.
	The main proposition of this section is the following result.
	\begin{prop}\label{prop:time_regularity}
		Let $\op$ be a solution to the Hartree--Fock equation with initial condition $\op^\init$ satisfying \eqref{eq:regularity_op}, \eqref{eq:regularity_sqrt}, and $\intd \rho^\init(x)\,(1+\n{x}^3)\d x \leq C$. Then there exists $T>0$ and a constant $C_T$ depending on $\op^\init$ such that for any $\(t,s\)\in[0,T]^2$,
		\begin{equation*}
			\Nrm{\(\widehat\sfG_t-\widehat\sfG_s\)\Psi}{\cG} \leq C_T \n{t-s}^\frac{3-2a}{7} \(\Nrm{\Psi}{\cG_2} + \Nrm{\Psi}{\dot{\cH}^{3/2}_2}\).
		\end{equation*}
	\end{prop}

	\begin{remark}
		For a fixed $\hbar$, the global-in-time well-posedness of solutions to the Hartree--Fock equation is a standard result (see for instance \cite{chadam_time-dependent_1976}). However, the bounds of the propagated quantity may depend on $\hbar$. In particular, for a general fixed $\hbar$, the constant $C_T$ in the above proposition may depend on $\hbar$.
	\end{remark}
	
	\begin{remark}
		We know from Part~\ref{part:regularity} that the conditions \eqref{eq:regularity_op} and \eqref{eq:regularity_sqrt} remain satisfied on $[0,T]$. In particular, $\Nrm{\sqrt{\op}}{\L^2(\n{\opp}^n)}^2 = \Tr{\op \n{\opp}^{2n}}$ is uniformly bounded on $[0,T]$. To see that the third-order spatial moment $\intd \rho^\init(x) \n{x}^3\d x = \Tr{\op \n{x}^3}$ remains bounded, one can notice that
		\begin{equation*}
			i\,\dt \Tr{\op \n{x}^3} = \Tr{\com{\tfrac{\n{\opp}^2}{2},\n{x}^3}\op} + \Tr{\com{\sfX_{\op},\n{x}^3}\op}.
		\end{equation*}
		The first term is controlled, using \cite[Formula~(42)]{lafleche_global_2021}, by a term proportional to
		\begin{equation*}
			\Tr{\op \(\n{x}^3+\n{\opp}^3+1\)}.
		\end{equation*}
		The second term is zero since
		\begin{equation*}
			\Tr{\com{\sfX_{\op},\n{x}^3}\op} = \iintd \n{\op(x,y)}^2 \frac{\n{y}^3-\n{x}^3}{\n{x-y}^a}\d x\d y
		\end{equation*}
		is the integral of an anti-symmetric function of $x$ and $y$. Then, by the standard Gr\"onwall argument, one obtains the desired result. 
	\end{remark}

	It will be convenient to use the fact that the above defined norm~\eqref{eq:Fock_Sobolev_norm} controls quantities of the form $\Nrm{\dG(A\nabla)\Psi}{\cG}$ as stated in the following lemma.
	\begin{lem}
		Let $A\in \L^\infty$ and $\Psi\in\dot{\cH}_1^1$. Then $\Nrm{\dG(A\nabla)\Psi}{\cG} \leq \Nrm{A}{\infty} \Lp{\Psi}{\dot{\cH}^1_{1}}$.
	\end{lem}
	
	\begin{proof}
		Using the fact that $A$ is a bounded operator, we obtain
		\begin{equation*}
			\Nrm{\dG(A\nabla)\Psi}{\cG}^2 \leq \Nrm{A}{\infty}^2 \sum_{n=1}^\infty \bigg(\sum_{j\leq n} \Nrm{\nabla_{x_j} \Psi^{(n)}}{L^2}\bigg)^2.
		\end{equation*}
		By the Cauchy--Schwarz inequality and integration by parts, the last factor satisfies
		\begin{multline*}
			\sum_{n=1}^\infty \bigg(\sum_{j\leq n} \Nrm{\nabla_{x_j} \Psi^{(n)}}{L^2}\bigg)^2 \leq \Inprod{\cN_l\Psi^{(n)}}{\DGl{-\Delta} \Psi^{(n)}}_{\cG} = \Nrm{\DGl{-\Delta}^{1/2}\cN_l^{1/2}\Psi}{\cG}^2
		\end{multline*}
		which is bounded above by $\Lp{\Psi}{\dot{\cH}^1_{1}}$.
	\end{proof}
	
	To simplify some of the calculation, it will also be convenient to employ the following lemma.
	\begin{lem}\label{lem:conju}
		For any self-adjoint integral operator $A$ on $\h = L^2(\Rd)$, we have the identities
		\begin{subequations}
			\begin{align}\label{conju1}
				\sfU^{(0)*}_t \DGl{A} \sfU^{(0)}_t &= \DGl{A_I}
				\\\label{conju2}
				\sfU^{(0)*}_t \DGplr{A} \sfU^{(0)}_t &= \DGplr{A_I}
			\end{align}
		\end{subequations}
		where $A_I := e^{-it\lapl}\,A \,e^{it\lapl}$ denotes the operator $A$ in the interaction picture.
	\end{lem}
	
	\begin{proof}
		By a direct computation, we see that
		\begin{equation}\label{commutator-id}
				\com{\sfL_0, a_{x, l}} = \com{\DGl{-\lapl}, a_{x, l}} = \lapl_x a_{x, l}.
		\end{equation}
		Therefore, using the Baker--Campbell--Hausdorff formula
		\begin{equation*}
			e^X Y e^{-X} = Y + \com{X, Y} + \frac{1}{2!}\com{X, \com{X, Y}} + \ldots
		\end{equation*}
		and Equation~\eqref{commutator-id}, one can show the conjugation formula
		\begin{equation*}
			\sfU^{(0)*}_t a_{x, l}\sfU^{(0)}_t = e^{it\lapl_x} a_{x, l}.
		\end{equation*}
		Hence, we arrive at the desired identity
		\begin{align*}
			\sfU^{(0)* }_t\DGl{A} \sfU^{(0)}_t &= \intdd A(z_1, z_2)\, 
			\sfU^{(0)*}_t\, a_{z_1,l}^* \, a_{z_2,r}\, \sfU^{(0)}_t \d z_1 \d z_2
			\\
			&= \intdd e^{-it\lapl_{z_1}}A(z_1, z_2)\,e^{it\lapl_{z_2}}\, a^*_{z_1, l} \,
			a_{z_2, r}\d z_1 \d z_2.
		\end{align*}
		This establishes \eqref{conju1}. The proof of Identity~\eqref{conju2} is similar. 
	\end{proof}
	
	\begin{proof}[Proof of Proposition~\ref{prop:time_regularity}]
		To prove Proposition~\ref{prop:time_regularity}, first notice that $\com{\sfU^{(0)}_t, \DG{-\lapl}} = 0$, which using Formula~\eqref{eq:aux_gen} allows us to write
		\begin{align*}
			\widehat\sfG_t &= \sfU^{(0)*}_t \Big(\DGl{V_{\op}-\sfX_{\op}} - \DGr{V_{\op}-\sfX_{\op}}\Big) \sfU^{(0)}_t
			\\
			&\qquad + \sfU^{(0)*}_t \(\tilde \sfQ + \tilde\sfQ^*\) \sfU^{(0)}_t + \sfU^{(0)*}_t\,\sfD\, \sfU^{(0)}_t + \sfU^{(0)*}_t\(\tilde \sfP+\tilde\sfP^*\)\sfU^{(0)}_t
			\\
			&=: \mathrm{I}_t + \mathrm{II}_t + \mathrm{III}_t + \mathrm{IV}_t.
		\end{align*}
		We shall prove the H\"older continuity of $t\mapsto \widehat\sfG_t\Psi$ by proving the property for each term $\mathrm{I}_t$, $\mathrm{II}_t$, $\mathrm{III}_t$, and $\mathrm{IV}_t$. This is the content of the following lemmas~\ref{lem:appendix_I}, \ref{lem:appendix_II}, \ref{lem:appendix_III}, \ref{lem:appendix_IV}. Combining these lemmas leads to the result.
	\end{proof}
	
	\begin{lem}\label{lem:appendix_I}
		Under the conditions of Proposition~\ref{prop:time_regularity}, there exists a constant $C_T$ depending on the initial conditions such that
		\begin{equation*}
			\Nrm{\dpt\mathrm{I}_t\Psi}{L^\infty((0,T),\cG)} \leq C_T \(\Nrm{\Psi}{\cG_1} + \Lp{\Psi}{\dot{\cH}^1_{1}}\).
		\end{equation*}
	\end{lem}
	
	\begin{proof}[Proof of Lemma~\ref{lem:appendix_I}]
		It suffices to consider the left contribution since the proof for the right contribution is exactly the same. Let us first handle the term with $V_{\op}$. Using Identity~\eqref{conju1}, we see that
		\begin{align*}
			i\dt\(\sfU^{(0)*}_t \DGl{V_{\op}} \sfU^{(0)}_t\) &= \sfU^{(0)*}_t \(\DGl{\com{V_{\op}, -\lapl}}
			+ \DGl{i\dpt\rho* K}\) \sfU^{(0)}_t
			\\
			&=: \sfJ_1 + \sfJ_2.
		\end{align*}
		We start by estimating $\sfJ_1\Psi$. We rewrite the commutator in $\sfJ_1$ by using the fact that
		\begin{equation*}
			\DGl{\com{V_{\op}, -\lapl}} = 2 \DGl{\grad V_{\op}\cdot \grad} + \DGl{\lapl V_{\op}}.
		\end{equation*}
		Then, since $\sfU^{(0)}_t$ is unitary and commutes with $\nabla$, we obtain
		\begin{equation*}
			\Lp{\sfU^{(0)*}_t \DGl{\grad V_{\op}\cdot \grad} \sfU^{(0)}_t\Psi}{\cG} \leq \Lp{\nabla V_{\op}}{L^\infty} \Lp{\Psi}{\dot{\cH}^1_{1}},
		\end{equation*}
		where since $\nabla K\in L^{\fb,\infty}$, we have
		\begin{equation*}
			\Lp{\nabla V_{\op}}{L^\infty} \leq \sup_{x\in\Rd} \intd \n{\nabla K(x-y)}\rho(y)\d y \leq \Nrm{\nabla K}{L^{\fb,\infty}} \Nrm{\rho}{L^{\fb',1}}.
		\end{equation*}
		Similarly, for the second term, by Lemma~\ref{lem:second_quantization_op} and the fact that $\nabla K\in L^{\fb,\infty}$, we have that
		\begin{equation*}
			\Lp{\sfU^{(0)*}_t \DGl{\lapl V_{\op}} \sfU^{(0)}_t\Psi}{\cG}
			\leq \Nrm{\lapl V_{\op}}{L^\infty} \Nrm{\Psi}{\cG_1} \leq \Nrm{\nabla K}{L^{\fb,\infty}} \Nrm{\nabla\rho}{L^{\fb',1}} \Nrm{\Psi}{\cG_1}.
		\end{equation*}
		By Proposition~\ref{prop:regu_HF}, the norm of $\rho$ in $L^{\fb',1}$ remains bounded for $t\in[0,T]$. When $\fb'\geq 2$, the same holds for $\nabla \rho$. Moreover, since $\hbar=1$, $\Nrm{\nabla \rho}{L^1} \leq C \Tr{\(1-\Delta\)\op}$ is also bounded on $[0,T]$ by Proposition~\ref{prop:regu_HF}, and so $\nabla \rho$ is in $L^\infty([0,T],L^p)$ for any $p\in[1,4]$. Hence, it follows that
		\begin{equation}\label{eq:appendix_bound_J1}
			\Lp{\sfJ_1\Psi}{\cG}\leq C_T \(\Lp{\Psi}{\cG_{1}} + \Lp{\Psi}{\dot{\cH}^1_{1}}\).
		\end{equation}
		For the $\sfJ_2$ term, let us begin by recalling the fact that $\rho$ satisfies the equation
		\begin{equation*}
			\dpt \rho + \nabla\cdot j_{\op} = 0
		\end{equation*}
		where $j_{\op} = \frac{1}{2} \Diag{\op\,\opp+\opp\,\op}$ is known as the probability current. Similarly as for $\sfJ_1$, we have the estimate
		\begin{equation}\label{eq:appendix_bound_J2}
			\Lp{\sfJ_2\Psi}{\cG} = \Lp{\sfU^{(0)*}_t\DGl{\grad\cdot\(j_{\op} * K\)} \sfU^{(0)}_t \Psi}{\cG} \leq \Nrm{j_{\op}}{L^{\fb',1}} \Lp{\nabla K}{L^{\fb,\infty}} 
			\Lp{\Psi}{\cG_1}.
		\end{equation}
		The term $\Nrm{j_{\op}}{L^{\fb',1}}$ is bounded as for $\rho$ by Proposition~\eqref{prop:diag_vs_weights} and the kinetic energy of $\op$.
		
		Now let us handle the exchange term $\sfX_{\op}$ in term $\mathrm{I}_t$. Note that 
		\begin{equation*}
			i\dt\(\sfU^{(0)*}_t \DGl{\sfX_{\op}} \sfU^{(0)}_t\) = \sfU^{(0)*}_t\!\(\DGl{\com{\sfX_{\op}, -\lapl}} + \DGl{i\bd_t \sfX_{\op}}\) \sfU^{(0)}_t
			=: \sfJ_3 + \sfJ_4.
		\end{equation*}
		We start by rewriting the $\sfJ_3$ term. Observe we have that
		\begin{equation*}
			\DGl{\com{\sfX_{\op}, -\lapl}} = 2 \DGl{(\sfX_{\Dh_x \op})\cdot \grad} + \DGl{\sfX_{\DDh_x\op}}.
		\end{equation*}
		The two terms are handled in the same exact manner as before. We will only deal with the second term. By Lemma~\ref{lem:second_quantization_op} and Inequality~\eqref{eq:X_HS_bound_2}, we have that
		\begin{equation}\label{eq:appendix_bound_J3_0}
			\Lp{\sfU^{(0)*}_t \DGl{\sfX_{\DDh_x \op}}\sfU^{(0)}_t\Psi}{\cG}
			\leq \Lp{\sfX_{\DDh_x \op}}{2} \Lp{\Psi}{\cG_{1}} \leq \Lp{\DDh_x\op \n{\opp}^a}{2} \Lp{\Psi}{\cG_{1}},
		\end{equation}
		and since $\hbar=1$, we have $\DDh_x\op = -\sum_{\jj = 1}^3 \com{\opp_\jj,\com{\opp_\jj,\op}}$ and so by Lemma~\ref{lem:expand_commutators}, $\Lp{\DDh_x\op \n{\opp}^a}{2}\leq C \Nrm{\op \n{\opp}^{a+2}}{2}$ which remains bounded on $[0,T]$ by Proposition~\ref{prop:regu_HF}.
		Hence we have the estimate
		\begin{equation}\label{eq:appendix_bound_J3}
			\Lp{\sfJ_3\Psi}{\cG}\leq C_T\(\Lp{\Psi}{\cG_{1}} + \Lp{\Psi}{\dot{\cH}^1_{1}}\).
		\end{equation}
		For the $\sfJ_4$ term, we have that
		\begin{equation*}
			\sfJ_4 = \sfU^{(0)*}_t \(\DGl{\sfX_{\com{-\lapl,\op}}} + \DGl{\sfX_{\com{V_{\op},\op}}} - \DGl{\sfX_{\com{\sfX_{\op},\op}}}\) \sfU^{(0)}_t.
		\end{equation*}
		To estimate the term with the Laplacian, we proceed as in Inequality~\eqref{eq:appendix_bound_J3_0} and use the fact that since $\hbar = 1$, $\com{-\lapl,\op} = \n{\opp}^2\op - \op \n{\opp}^2$. To estimate the second term, we use Inequality~\eqref{eq:X_HS_bound_2} to get
		\begin{equation*}
			 \Nrm{\sfX_{\com{V_{\op},\op}}}{\infty} \leq \Nrm{\com{V_{\op},\op}\n{\opp}^a}{\infty} \leq \Nrm{\com{V_{\op},\op}\(1+\n{\opp}^2\)}{\infty}.
		\end{equation*}
		Then similarly as in Section~\ref{sec:commutators}, we write $\com{V_{\op},\op} m = \com{V_{\op},\op\,m} - \com{V_{\op},m} \op$ and use Proposition~\ref{prop:estim_commutator_Lq_Besov} and Proposition~\ref{prop:weighted_com_est} with $V_{\op}$ instead of $E_{\op}$. Similarly, to bound the last term, we use Inequality~\eqref{eq:X_HS_bound_2} and then Proposition~\ref{prop:exchange_com}. Hence we have the estimate
		\begin{equation}\label{eq:appendix_bound_J4}
			\Lp{\sfJ_4\Psi}{\cG} \leq C_T \Nrm{\Psi}{\cG_1}.
		\end{equation}
		The bound on $\dpt \mathrm{I}_t$ now follows by combining the inequalities for each part, i.e. \eqref{eq:appendix_bound_J1}, \eqref{eq:appendix_bound_J2}, \eqref{eq:appendix_bound_J3} and \eqref{eq:appendix_bound_J4}.
		\end{proof}
		
	\begin{lem}\label{lem:appendix_II}
		Under the conditions of Proposition~\ref{prop:time_regularity}, there exists a constant $C_T$ depending on the initial conditions such that for any $\(t,s\)\in[0,T]^2$,
		\begin{equation*}
			\Nrm{\(\mathrm{II}_t-\mathrm{II}_s\)\Psi}{\cG} \leq C_T \n{t-s}^\frac{3-2a}{7} \Nrm{\Psi}{\cG_{3/2}}.
		\end{equation*}
	\end{lem}
	
	\begin{proof}[Proof of Lemma~\ref{lem:appendix_II}]
		To estimate term $\mathrm{II}$, it suffices to focus on the first term of $\tilde\sfQ^*$, which we will denote by $\tilde \sfQ_1^*$. Furthermore, we decompose the singular potential into a long-range part and a singular part as follows
		\begin{equation}\label{pot-decomp}
			K = K_R^L + K_R^S := C_a \(\int^{R^{-2}}_0 s^{\frac{a}{2}-1}\varphi_{s}\d s + \int^\infty_{R^{-2}} s^{\frac{a}{2}-1}\varphi_{s}\d s\) .
		\end{equation}
		for some $R$ which we will determine shortly, and with $\varphi_s(x) = e^{-\pi \n{x}^2 s}$. Consequently, we have the decomposition
		\begin{equation*}
			N\, \tilde \sfQ_1^* = \intdd \(K^L_R + K^S_R\)\!(x-y) \DGplr{u\,\delta_x\,v} \DGplr{u\,\delta_y\,v} \d x\d y =: \tilde \sfQ_{1, R}^{L *}+\tilde \sfQ_{1, R}^{S *}.
		\end{equation*}
		
		For the long-range part, we follow the proof of the bounded potential case as in \cite{benedikter_mean-field_2016} and show that $\tilde\sfQ_{1, R}^{L*}$ is time differentiable. Applying Lemma~\ref{lem:conju} and the operator identity $e^{-it\lapl}A(x)\,e^{it\lapl} = A(x-2it\grad)$, we can now rewrite $\tilde\sfQ^L_{1,R}$ as follows
		\begin{multline*}
			\sfU^{(0)*}_t \tilde\sfQ^{L*}_{1,R} \sfU^{(0)}_t = \intd \widehat{K^L_R}(y) \DGplr{u_I\, e^{iy\cdot\(x-2it\grad\)}\,v_I} \DGplr{u_I\,e^{-iy\cdot\(x-2it\grad\)}\,v_I} \d y,
		\end{multline*}
		where $A_I := e^{-it\lapl}A\,e^{it\lapl}$ denotes the operator $A$ in the interaction picture. To estimate the time derivative of $\tilde\sfQ^L_{1,R}$, we make the observation that
		\begin{multline*}
			i\,\dpt(u_I\,e^{iy\cdot(x-2it\grad)}\,v_I) = e^{-it\lapl} \(u \com{e^{i y\cdot x}, -\lapl} v\) e^{it\lapl}
			\\
			+ e^{-it\lapl} \(\com{V_{\op}-\sfX_{\op}, u} e^{i y\cdot x}\, v + u\,e^{i y\cdot x} \com{V_{\op}-\sfX_{\op},v}\) e^{it\lapl}.
		\end{multline*}
		Applying Lemma~\ref{lem:second_quantization_op}, we have the estimates
		\begin{equation}\label{time-deriv-uv}
		\begin{split}
			\Lp{\DGplr{u_I e^{iy\cdot(x-2it\grad)}v_I}\Psi}{\cG} &\leq 2 \Lp{u}{\infty} \Lp{v}{2} \Lp{\Psi}{\cG_1},
			\\
			\Lp{\DGplr{\dpt(u_I\, e^{iy\cdot\(x-2it\grad\)}\,v_I)}\Psi}{\cG} &\leq C_T \weight{y}^2 \Lp{u}{\infty} \Lp{\bangle{\grad}v}{2}\Lp{\Psi}{\cG_1},
		\end{split}
		\end{equation}
		where $C_T = C\sup_{t\in[0,T]}\(1 + \Lp{V_{\op}}{L^\infty} + \Lp{\sfX_{\op}}{\infty}\)$ is finite, and $\weight{y}^2 = 1+\n{y}^2$. In particular, it follows from~\eqref{time-deriv-uv} that we have the inequality
		\begin{equation*}
			\Lp{\dt \sfU^{(0)*}_t\tilde\sfQ^{L*}_{1, R}\sfU^{(0)}_t\Psi}{\cG} \leq C_T \intd \n{\widehat{K^L_R}(y)} \weight{y}^2 \d y \Lp{\Psi}{\cG_1}.
		\end{equation*}
		To complete the estimate, we need to compute the $L^1$-norm of $\widehat{K^L_R}$ to get the explicit dependence of the constant on $R$. Using the fact that $\widehat\varphi_s = s^{-\frac{3}{2}}\varphi_{1/s}$, we have 
		\begin{equation}\label{eq:fourier_cutoff}
			\intd \n{\widehat{K^L_R}} \weight{y}^2 \d y = \int_0^{R^{-2}}\!\!\!\!\!\intd s^{\frac{a-5}{2}} e^{-\frac{\pi}{s} \n{y}^2} \weight{y}^2 \d y \d s = \frac{3}{\pi}\(\tfrac{R^{-\(a+2\)}}{a+2} + \tfrac{R^{-a}}{a}\).
		\end{equation}
		Therefore, provided $R<1$, we obtain the estimate
		\begin{equation*}
			\Lp{\dt\(\sfU^{(0)*}_t \tilde\sfQ^{L*}_{1,R} \sfU^{(0)}_t\Psi\)}{\cG} \leq \frac{6}{\pi a} R^{-(a+2)} \Lp{\Psi}{\cG_1}
		\end{equation*}
		which implies for any $(t,s)\in[0,T]^2$,
		\begin{equation}\label{eq:holder1}
			\Lp{\(\sfU^{(0)*}_t\tilde\sfQ^{L*}_{1, R}\sfU^{(0)}_t-\sfU^{(0)*}_s\tilde\sfQ^{L*}_{1, R}\sfU^{(0)}_s\)\Psi}{\cG} \leq \frac{6}{\pi a}\, R^{-(a+2)} \n{t-s} \Lp{\Psi}{\cG_1}.
		\end{equation}
		For the singular part, by the Cauchy--Schwarz inequality, we have
		\begin{multline*}
			\n{\Inprod{\Psi_1}{\tilde\sfQ^{S*}_{1, R}\Psi_2}} =  \n{\intd \Inprod{a_l(u_x)\Psi_1}{a^*_r(\conj{v_x}) \DGplr{u\, K^S_{R, x}v}\Psi_2} \d x}
			\\
			\leq \(\intd \Lp{a_l(u_x)\Psi_1}{\cG}^2 \d x \)^\frac{1}{2} \(\intd \Lp{a^*_r(\conj{v_x}) \DGplr{u\, K^S_{R,x}v}\Psi_2}{\cG}^2 \d x\)^\frac{1}{2}.
		\end{multline*}
		Applying Lemma~\ref{lem:second_quantization_op} and the fact that $\Lp{u}{\infty}\le 1$ yields 
		\begin{equation*}
			\Lp{a^*_r(\conj{v_x}) \DGplr{u\, K^S_{R, x}v} \Psi_2}{\cG} \leq (N\rho(x))^\frac{1}{2} \Lp{K_{R,x}^S v}{2} \Lp{\Psi_2}{\cG_{1/2}},
		\end{equation*}
		which gives us
		\begin{equation*}
			\n{\Inprod{\Psi_1}{\tilde\sfQ^{S*}_{1, R}\Psi_2}} \leq C\,N^\frac{1}{2} \Inprod{\Psi_1}{\cN\Psi_1}^\frac{1}{2} \(\intd \rho(x) \Lp{K^S_{R, x}v}{2}^2\d x\)^\frac{1}{2} \Lp{\Psi_2}{\cG_{1/2}}.
		\end{equation*}
		Since $\Diag{v^2} = N\,\rho$, we see that
		\begin{align*}
			\Lp{K^S_{R, x}v}{2} &\leq \int^\infty_{R^{-2}} s^{\frac{a}{2}-1} \Lp{\varphi_{s,x}\,v}{2} \d s
			\\
			&= N^\frac{1}{2}\int^\infty_{R^{-2}} s^{\frac{a}{2}-1} \(\n{\varphi_{s}}^2* \rho\)^\frac{1}{2}\!(x)\, \d s
			\\
			&\le N^\frac{1}{2} \Lp{\rho}{L^\infty}^\frac{1}{2} \int_{R^{-2}}^\infty s^{\frac{a}{2}-1} \Lp{\varphi_{s}}{L^2} \d s \leq C_T\, N^\frac{1}{2}\,R^{\frac{3}{2}-a}.
		\end{align*}
		Hence, by duality, it follows that $\Lp{\tilde\sfQ^{S}_{1, R}\Psi}{\cG} \leq R^{\frac{3}{2}-a} \Lp{\Psi}{\cG_{3/2}}$. By a similar argument, one can also show the same inequality for the dual operator $\tilde\sfQ^{S*}_{1,R}$. Therefore,
		\begin{equation}\label{eq:holder2}
			\Lp{\tilde\sfQ^{S*}_{1,R}\Psi}{\cG} \leq C_{T}\,N\, R^{\frac{3}{2}-a} \Lp{\Psi}{\cG_{3/2}}.
		\end{equation}
		Combining \eqref{eq:holder1} and \eqref{eq:holder2}, we obtain that for any $(t,s)\in [0,T]^2$ and any $R\in (0,1)$, the following inequality holds
		\begin{equation*}
				\Lp{\(\sfU^{(0)*}_t\tilde\sfQ^{*}_{1}\sfU^{(0)}_t-\sfU^{(0)*}_s\tilde\sfQ^{*}_{1}\sfU^{(0)}_s\)\Psi}{}
				\leq C_T \(R^{-(a+2)}\n{t-s} + R^{\frac{3}{2}-a}\) \Lp{\Psi}{\cG_{3/2}}.
		\end{equation*}
		In particular, if $t\neq s$, one can take $R^{\frac{7}{2}} = \frac{\n{t-s}}{T} \leq 1$, leading to
		\begin{equation*}
			\Lp{\(\sfU^{(0)*}_t\tilde\sfQ^{*}_{1}\sfU^{(0)}_t-\sfU^{(0)*}_s\tilde\sfQ^{*}_{1}\sfU^{(0)}_s\)\Psi}{\cG} \leq C_T \n{t-s}^\frac{3-2a}{7}\Lp{\Psi}{\cG_{3/2}}.
		\end{equation*}
		If $t=s$, we can make $R\to 0$ to obtain the same inequality.
	\end{proof}
	
	Next, let us consider the type $\mathrm{III}$ terms.
	\begin{lem}\label{lem:appendix_III}
		Under the conditions of Proposition~\ref{prop:time_regularity}, there exists a constant $C_T$ depending on the initial conditions such that for any $\(t,s\)\in[0,T]^2$,
		\begin{equation*}
			\Nrm{\(\mathrm{III}_t-\mathrm{III}_s\)\Psi}{\cG} \leq C_T \n{t-s}^\frac{3-2a}{7} \(\Nrm{\Psi}{\cG_2} + \Nrm{\Psi}{\dot{\cH}^{3/2}_2}\).
		\end{equation*}
	\end{lem}
	
	\begin{proof}[Proof of Lemma~\ref{lem:appendix_III}]
		Let us focus on the first term of $\sfD$ which we denote by $\sfD_1$. The proof of H\"older
		continuity of $\sfD_1$ is similar to that of $\tilde\sfQ_1$. Using \eqref{pot-decomp}, we decompose $\sfD_1$ into two parts 
		\begin{equation*}
			2\,N\,\sfD_1 = \sfD_{1,R}^L+\sfD_{1,R}^S.
		\end{equation*}
		For the long-range part, we begin by writing 
		\begin{multline*}
			\sfU^{(0)*}_t \sfD_{1,R}^L \sfU^{(0)}_t = \intd \widehat{K^L_R}(y) \DGl{u_I \, e^{iy\cdot\(x-2it\grad\)} u_I} \DGl{u_I \, e^{-iy\cdot\(x-2it\grad\)} \, u_I} \d y. 
		\end{multline*}
		Using the identity
		\begin{equation*}
			i\dpt(u_I \, e^{iy\cdot\(x-2it\grad\)} \, u_I) = e^{-it\lapl} \(\eta \, e^{iy\cdot x} \, u + u \, e^{iy\cdot x} \, \eta + u \com{e^{iy\cdot x}, -\lapl} u\) e^{it\lapl},
		\end{equation*} 
		where $\eta = [V_{\op}-\sfX_{\op}, u]$, and Lemma~\ref{lem:second_quantization_op}, we deduce the following estimates
		\begin{equation}\label{time-deriv-uu}
		\begin{split}
			\Lp{\DGl{u_I \, e^{iy\cdot(x-2it\grad)} u_I}\Psi}{\cG} &\leq \Nrm{\Psi}{\cG_1}
			\\
			\Lp{\DGl{\dpt (u_I \, e^{iy\cdot(x-2it\grad)} \, u_I)}\Psi}{\cG}
			&\leq \(\Lp{\eta}{\infty}+\n{y}^2\) \Nrm{\Psi}{\cG_1} + \n{y} \Lp{\Psi}{\dot{\cH}^1_{1}}. 
		\end{split}
		\end{equation}
		By the above inequalities~\eqref{time-deriv-uu} and Formula~\eqref{eq:fourier_cutoff}, provided $R\in(0,1)$, we get an estimate of the form
		\begin{equation*}
			\Lp{\dt\(\sfU^{(0)*}_t \sfD^{L}_{1,R} \sfU^{(0)}_t\)\Psi}{\cG} \leq C_a \(1+\Nrm{\eta}{\infty}\) R^{-(a+2)} \(\Lp{\Psi}{\cG_{2}} + \Lp{\Psi}{\dot{\cH}^1_{2}}\).
		\end{equation*}
		To handle the singular part, we begin by writing $u = 1 - w$. Then it follows that
		 \begin{align*}
			 \sfD^S_{1,R} &= \intdd K^S_R(x-y) \,a^*_{x, l} \, a^*_{y,l} \, a_l(u_y) \, a_l(u_x) \d x \d y
			 \\
			 &\qquad + \intdd K^S_R(x-y) \, a^*_{l}(w_x) \, a^*_{l}(w_y) \, a_l(u_y) \, a_l(u_x) \d x \d y
			 \\
			 &\qquad - \intdd K^S_R(x-y) \, a^*_{x,l} \, a^*_{l}(w_y) \, a_l(u_y) \, a_l(u_x) \d x \d y
			 \\
			 &\qquad - \intdd K^S_R(x-y) \, a^*_{l}(w_x) \, a^*_{y,l} \, a_l(u_y) \, a_l(u_x) \d x \d y
			 \\\nonumber
			 &=: \sfI_1+ \sfI_2+\sfI_3+\sfI_4.
		\end{align*}
		To estimate $\sfI_1$, we begin by observing that
		\begin{equation*}
			(\sfI_1\Psi)^{(n,m)}\!\(\underline{x}_{n},\underline{y}_{m}\) =  \sum_{1\leq j < k \leq n} K^S_R(x_i-x_j) \(\bar u^{(x_j)} \bar u^{(x_k)} \Psi^{(n,m)}\)\!\(\underline{x}_{n},\underline{y}_{m}\)
		\end{equation*}
		where $u^{(x_j)}$ is the operator acting on the variable $x_j$ and $\underline{x}_{n} = (x_1,\dots,x_n)$. Defining $g(x,y) := \Nrm{\bar u^{(x)} \bar u^{(y)} \Psi^{(n,m)}\!\(x,y,\dots\)}{L^2(\R^{3\(n+m-2\)})}$, it follows from the triangle inequality and the anti-symmetry of $\Psi$ that
		\begin{align*}
			\Lp{(\sfI_1\Psi)^{(n,m)}}{L^2} &\leq n\(n-1\) \(\iint_{\n{z}\leq R} \frac{\n{g(x+z,x)}^2}{\n{z}^{2a}} \d z \d x\)^\frac{1}{2}
			\\
			&\leq n^2 R^{\frac{3}{2}-a} \(\intd\Nrm{\frac{g_{R,x}(z)}{\n{z}^{a}}}{L^2_z(B_1)}^2 \d x\)^\frac{1}{2}
		\end{align*}
		where $g_{R,x}(z) = g(x+zR,x)$ and $B_1$ is the unit ball of $\R^3$. Now let $T_{B_1}$ be the bounded extension operator 
		\begin{align*}
			T_{B_1} &: H^a({B_1}) \to H^a(\R^3), & \forall x\in B_1, (T_{B_1} g)(x) &= g(x)
		\end{align*}
		which exists as proved for example in \cite[Theorem~IX.7]{brezis_analyse_2005} when $a\in\N$. One can proceed by interpolation when $a\in\R$. Then one has
		\begin{equation*}
			\Nrm{\frac{g_{R,x}(z)}{\n{z}^{a}}}{L^2_z(B_1)} = \Nrm{\frac{T_{B_1} g_{R,x}(z)}{\n{z}^{a}}}{L^2_z(B_1)} \leq \Nrm{\frac{T_{B_1} g_{R,x}(z)}{\n{z}^{a}}}{L^2_z(\R^3)}
		\end{equation*}
		and by Hardy--Rellich's inequality
		\begin{equation*}
			\Nrm{\frac{T_{B_1} g_{R,x}(z)}{\n{z}^{a}}}{L^2_z(\R^3)} \leq C \Nrm{\(-\Delta\)^\frac{a}{2} T_{B_1}g_{R,x}}{L^2(\R^3)} \leq C_{T_{B_1}} \Nrm{g_{R,x}}{H^a(B_1)}.
		\end{equation*}
		For $\sfI_1$, this leads to
		\begin{align*}
			\Lp{(\sfI_1\Psi)^{(n,m)}}{L^2} &\leq C\,n^2 R^{\frac{3}{2}-a} \(\iint_{\n{z}\leq 1} \n{(-\Delta)^\frac{a}{2}g_{R,x}(z)}^2 + \n{g_{R,x}(z)}^2 \d x \d z\)^\frac{1}{2}
			\\
			&\leq C\,n^2 \(\iint_{\n{z}\leq R} \n{\(-\Delta\)_z^\frac{a}{2} g}^2 + \n{R^{-a} g}^2 \d z \d x\)^\frac{1}{2}.
		\end{align*}
		Now by H\"older's inequality and by Sobolev's embedding, for any $\alpha > 0$ and any $f\in H^\alpha$,
		\begin{equation*}
			\int_{\n{z}\leq R} \n{f(z)}^2\d z \leq \Nrm{f}{L^{2p'}}^2 \Nrm{\indic_{B_R}}{L^{p}} \leq C\,R^{2\alpha} \Nrm{(-\Delta)^{\alpha/2} f}{L^2}^2,
		\end{equation*}
		with $p = \frac{3}{2\alpha}$. In particular, taking $f= \(-\Delta\)_z^\frac{a}{2} g$ and $\alpha = \frac{3}{2} - a$ and taking $f= R^{-a} g$ and $\alpha = \frac{3}{2}$ we obtain for $\sfI_1$
		\begin{equation*}
			\Lp{(\sfI_1\Psi)^{(n,m)}}{L^2} \leq C\,n^2\, R^{\frac{3}{2} - a} \Nrm{(-\Delta)^{3/4}_x g}{L^2}.
		\end{equation*}
		Using the fact that $\Nrm{u}{\infty} \leq 1$, we can control the $L^2$ norm on the right-hand side of the above inequality by
		\begin{align*}
			\Nrm{\(-\Delta\)_x^{3/4} g}{L^2} &= \Nrm{\(\bar u^{(y)} \(-\Delta\)_x^{3/4} \bar u^{(x)} \Psi^{(n,m)}\)\!\(x,y,\dots\)}{L^2(\R^{3\(n+m\)})}
			\\
			&\leq \Nrm{\(\(-\Delta\)^{3/4}\bar u\)^{(x)} \Psi^{(n,m)}\(x,\dots\)}{L^2(\R^{3\(n+m\)})}
		\end{align*}
		and using the fact that $\bar{u} = 1 - \bar{w}$, we finally obtain
		\begin{equation*}
			\Lp{\sfI_1\Psi}{\cG} \leq \frac{C}{N}\,R^{\frac{3}{2}-a} \(\Lp{\Psi}{\dot{\cH}^{\frac{3}{2}}_2} + \Lp{\n{\opp}^{\frac{3}{2}}w}{2} \Lp{\Psi}{\cG_2}\).
		\end{equation*}
		The other $\sfI_i$ terms are less singular and treated in the same way, leading to
		\begin{equation*}
			\Lp{\sfU^{(0)*}_t\sfD^{S}_{1, R}\sfU^{(0)}_t\Psi}{\cG} \leq C_T\, R^{\frac{3}{2}-a} \(\Lp{\cN\Psi}{\dot\cH^{3/2}_2} + \Nrm{\Psi}{\cG_2}\).
		\end{equation*}
		By the same argument as in the case of $\tilde\sfQ_1$, we see that $\sfU^{(0)*}_t\sfD_{1}\sfU^{(0)}_t\Psi$ is also H\"older continuous in time.
	\end{proof}
	
	Finally, let us handle type $\mathrm{IV}$ terms.
		
	\begin{lem}\label{lem:appendix_IV}
		Under the conditions of Proposition~\ref{prop:time_regularity}, there exists a constant $C_T$ depending on the initial conditions such that for any $\(t,s\)\in[0,T]^2$,
		\begin{equation*}
			\Nrm{\(\mathrm{IV}_t-\mathrm{IV}_s\)\Psi}{\cG} \leq C_T \n{t-s}^\frac{3-2a}{7} \(\Nrm{\Psi}{\cH^1_2} + \Nrm{\Psi}{\cG_2}\).
		\end{equation*}
	\end{lem}
	
	\begin{proof}[Proof of Lemma~\ref{lem:appendix_IV}]
		For this case, it suffices to consider 
		\begin{align*}
		 J_1 &= -\intd \DGplr{u\,\delta_x\,v} \DGl{u \com{K_x, u}} \d x
		 \\
		 J_{12} &= - \intd \DGplr{\com{K_x, u} v + \com{v, K_x} u} \DGl{\omega_x} \d x.
		\end{align*}
		Following the same routine as before, we decompose the operators into a long-range part and a singular part using Formula~\eqref{pot-decomp}. Again, we will denote the decomposition by $J_{1, R}^L+J_{1, R}^S$ and likewise for $J_{12}$. Applying Lemma~\ref{lem:conju}, we can now rewrite $J^L_{1,R}$ as follows
		\begin{equation*}
			\sfU^{(0)*}_tJ^{L}_{1, R}\sfU^{(0)}_t = \intd \widehat{K^L_R}(y) \DGplr{u_I \, e^{-iy\cdot (x-2it\grad)} v_I} \DGl{u_I \com{e^{-iy\cdot(x-2it\grad)}, w_I}} \d y.
		\end{equation*}
		Since we have that
		\begin{multline*}
		i\,\dpt\!\(u_I \com{e^{-iy\cdot(x-2it\grad)}, w_I}\) = e^{-it\lapl} \(\com{V_{\op}-\sfX_{\op}, u}\com{e^{-iy\cdot x}, w}\) e^{it\lapl}
		\\
		+ e^{-it\lapl} \(u \com{\com{e^{i y \cdot x}, -\lapl}, w} + u \com{e^{-iy\cdot x}, \com{V_{\op}-\sfX_{\op}, w}}\) e^{it\lapl}
		\end{multline*}
		then by Lemma~\ref{lem:second_quantization_op}, since $\Nrm{u}{\infty}\leq 1$ and $\Nrm{w}{\infty}\leq 1$, we have the estimate 
		\begin{subequations}\label{time-deriv-up}
		\begin{align}
			\Lp{\DGl{u_I \com{e^{-iy\cdot(x-2it\grad)},w_I}}\Psi}{\cG} &\leq 2 \Nrm{\Psi}{\cG_1}
			\\
			\Lp{\DGl{\dpt\!\(u_I \com{e^{-iy\cdot(x-2it\grad)}, w_I}\)}\Psi}{\cG} &\leq \cC_T \weight{y}^2 \Lp{\weight{\opp}^2 w}{2} \Nrm{\Psi}{\cG_1}.
		\end{align}
		\end{subequations}
		where $\cC_T = C\sup_{[0,T]}\!\(1 + \Lp{V_{\op}}{\infty} + \Lp{\sfX_{\op}}{\infty}\)$. In particular, by inequalities~\eqref{time-deriv-uv}, \eqref{time-deriv-up}, and~\eqref{eq:fourier_cutoff}, we have that
		\begin{equation*}
		 \Lp{\dt \sfU^{(0)*}_t J^{L}_{1,R} \sfU^{(0)}_t \Psi}{\cG} \leq C_T\, R^{-(a+2)} \Nrm{\Psi}{\cG_1}.
		\end{equation*}
		The singular part follows from Lemma~\ref{lem:bound_Q_tilde} and Remark~\ref{remark:cutoff_2}. More precisely, we have that
		\begin{equation}
			\Lp{\sfU^{(0)*}_tJ^{S}_{1, R}\sfU^{(0)}_t\Psi}{\cG} \leq C_T\, R^{\frac{3}{2}-a} \Nrm{\Psi}{\cG_1}.
		\end{equation}
		Repeating the argument of $\tilde \sfQ_1$ shows that $\sfU^{(0)*}_tJ_{1}\sfU^{(0)}_t\Psi$ is H\"older continuous in time. 
		
		Lastly, let us estimate the operator $J_{12}$. We begin by writing 
		\begin{subequations}
			\begin{align}\label{eq:def_J12_term1}
				\sfU^{(0)*}_tJ_{12}\sfU^{(0)}_t &= \intd \DGplr{\com{K_{x, I}, w_I} v_I } \DGl{\omega_{x, I}} \d x
				\\\label{eq:def_J12_term2}
				&\qquad +\intd \DGplr{ \com{K_{x, I}, v_I} u_I} \DGl{\omega_{x, I}} \d x.
			\end{align}
		\end{subequations}
		It suffices to handle Term~\eqref{eq:def_J12_term2} since Term~\eqref{eq:def_J12_term1} can be treated in a similar manner. Taking its time-derivative yields
		\begin{align*}
			i\,\dpt \eqref{eq:def_J12_term2} &= \intd \DGplr{ i\bd_t(\com{K_{x, I}, v_I} u_I)} \DGl{\omega_{x, I}} \d x
			\\
			&\qquad + \intd \DGplr{ \com{K_{x, I}, v_I} u_I} \DGl{i\bd_t\omega_{x, I}} \d x =: \sfI_5 + \sfI_6.
		\end{align*}
		Let us first consider $\sfI_6$. Notice, we have the identity
		\begin{equation*}
			i\,\dpt \omega_{x, I} = e^{-it\lapl} \(2\,\grad_1\omega_x\cdot \grad_1 + \lapl_2\omega_x + \com{V_{\op}-\sfX_{\op}, \omega}_x\)e^{it\lapl}.
		\end{equation*}
		In particular, we can write 
		\begin{align*}
			\sfI_6 &= \sfU^{(0)*}_t \intd \DGplr{ \com{K_{x}, v} u} \DGl{2\,\grad_1\omega_x\cdot\grad} \d x\, \sfU^{(0)}_t
			\\
			&\quad + \sfU^{(0)*}_t \intd \DGplr{ \com{K_{x}, v} u} \DGl{\lapl_x\omega_x} \d x\, \sfU^{(0)}_t
			\\
			&\quad + \sfU^{(0)*}_t \intd \DGplr{ \com{K_{x}, v} u} \DGl{\com{V_{\op}-\sfX_{\op}, \omega}_x} \d x\, \sfU^{(0)}_t
			\\
			&=: \sfJ_1+\sfJ_2+\sfJ_3.
		\end{align*}
		To bound $\sfJ_1$, it suffices to estimate the following quantity
		\begin{equation}\label{eq:term_J1_of_J12}
			\Nrm{\intd (\com{K_x, v}u)(z_1, z_2)\, \grad_1\omega(x_n, x) \cdot\grad\Psi^{(n,m)}(\underline{x}_{n-1}, x_n,\underline y_n)\d x}{L^2(\d\underline z_2 \d \underline x_n\d \underline y_n)}
		\end{equation}
		where $\d\underline x_n = \d x_1\ldots \d x_n$ and $\d\underline z_2 = \d z_1 \d z_2$. Let us also break the commutator, that is, 
		\begin{align*}
			\eqref{eq:term_J1_of_J12} &\leq \Nrm{\intdd \frac{v(z_1, z) \, u(z, z_2)}{\n{x-z_1}^a} \, \grad_1\omega(x_n, x) \cdot \grad\Psi^{(n,m)}(\underline{x}_{n-1}, x_n,\underline y_n)\d x\d z}{L^2(\d\underline z_2 \d \underline x_n\d \underline y_n)}
			\\
			&\quad + \Nrm{\intdd \frac{v(z_1, z)\, u(z, z_2)}{\n{x-z}^a} \, \grad_1\omega(x_n, x)\cdot \grad\Psi^{(n,m)}(\underline{x}_{n-1}, x_n,\underline y_n) \d x\d z}{L^2(\d\underline z_2 \d \underline x_n\d \underline y_n)}.
		\end{align*}
		 Since $u = 1-w$ where $w$ is a Hilbert--Schmidt operator, we will focus on the identity part. Using the fact that $\omega = v^2$ and the Cauchy--Schwarz inequality, we see that
		\begin{multline*}
			\Nrm{\intd \frac{v(z_1, z_2)}{\n{x-z_1}^a} \, \grad_1\omega(x_n, x) \cdot \grad_{x_n}\Psi^{(n,m)}(\underline{x}_{n-1}, x_n,\underline y_n)\d x}{L^2(\d\underline z_2 \d \underline x_n\d \underline y_n)}.
			\\
			= \Nrm{\intdd \frac{v(z, x)}{\n{x-z_1}^a} \, v(z_1, z_2) \grad_1v(x_n, z) \cdot \grad_{x_n}\Psi^{(n,m)}(\underline{x}_{n-1}, x_n,\underline y_n) \d x\d z}{L^2(\d\underline z_2 \d \underline x_n\d \underline y_n)}
			\\
			\le \sup_{z_1} \(\intd \frac{\rho(x)^\frac{1}{2}}{\n{x-z_1}^a} \d x\) \Nrm{v}{2} \Nrm{\grad_1v}{L^\infty_xL^2_z} \Nrm{\grad\Psi^{(n)}}{L^2(\d \underline x_n\d \underline y_n)},
		\end{multline*}
		since $\Nrm{v_x}{L^2} = \rho(x)^{1/2}$. Note that by Young's and H\"older's inequalities,
		\begin{equation*}
			\intd \frac{\rho(x)^\frac{1}{2}}{\n{x- z_1}^a}\d x \leq C\Nrm{\rho^{1/2}}{L^\frac{3}{3-a}} \leq C \intd \rho(x)\weight{x}^k\d x
		\end{equation*}
		provided $k>3-2a$. Hence, we have the estimate 
		\begin{equation*}
			\Nrm{\sfJ_1\Psi}{\cG}\le C_T\(\Nrm{\Psi}{\cH_2^1}+\Nrm{\Psi}{\cG_2}\).
		\end{equation*}
		The other two terms $\sfJ_2$ and $\sfJ_3$ can be handled in the same manner since $v$ is sufficiently smooth and $\Nrm{V_{\op}}{L^\infty}$ and $\Nrm{\sfX_{\op}}{L^\infty_xL^2_y} \le C \Nrm{\op\n{\opp}^{2+a}}{2} $ are bounded. Thus, it follows that 
		\begin{equation*}
			\Nrm{\sfI_6\Psi}{\cG}\le C_T\(\Nrm{\Psi}{\cH_2^1}+\Nrm{\Psi}{\cG_2}\).
		\end{equation*}

		Lastly, we handle the $\sfI_5$ term. Since we have that
		\begin{align*}
			i\,\dpt\!\(\com{K_{x, I}, v_I}u_I\) &= e^{-it\lapl} \(\com{K_{x}, v} \com{V_{\op}-\sfX_{\op}, u}\) e^{it\lapl}
			\\
			&\qquad + e^{-it\lapl} \(\com{\com{\lapl, K_x}, v}u + \com{K_x,\com{V_{\op}-\sfX_{\op}, v}}u\) e^{it\lapl}
		\end{align*}
		then we can write 
		\begin{align*}
			\sfI_5 &= \sfU^{(0)*}_t \intd \DGplr{ \com{K_{x}, v} \com{-V_{\op}+\sfX_{\op}, w}} \DGl{\omega_x} \d x\, \sfU^{(0)}_t
			\\
			&\qquad + \sfU^{(0)*}_t \intd \DGplr{ \com{K_x, \com{V_{\op}-\sfX_{\op}, v}}u} \DGl{\omega_x} \d x\, \sfU^{(0)}_t
			\\
			&\qquad + \sfU^{(0)*}_t \intd \DGplr{ \com{\com{\lapl, K_x}, v}u} \DGl{\omega_x} \d x\, \sfU^{(0)}_t
			\\
			&=: \sfJ_4+\sfJ_5+\sfJ_6.
		\end{align*}
		Term $\sfJ_4$ and $\sfJ_5$ can be estimated in the same manner as in the previous case, since $\com{-V_{\op}+\sfX_{\op}, w}$ is a bounded operator and $\com{V_{\op}-\sfX_{\op}, v}$ is a Hilbert--Schmidt operator. Hence, it suffices to estimate $\sfJ_6$. 

		To bound $\sfJ_6$, it suffices to estimate the following quantity
		\begin{equation}\label{eq:term_J6_of_J12}
			\Nrm{\intd \com{\lapl K_x+2\grad K_x\cdot\grad, v}\!(z_1, z_2)\,\omega(x_n, x) \, \Psi^{(n,m)}(\underline{x}_{n-1}, x_n,\underline y_n) \d x}{L^2(\d\underline z_2 \d \underline x_n\d \underline y_n)}.
		\end{equation}
		In the case $a=1$, we have that 
		\begin{align*}
			\eqref{eq:term_J6_of_J12} &\le C\Nrm{\intd (v\,\delta_x )(z_1, z_2) \, \omega(x_n, x) \, \Psi^{(n,m)}(\underline{x}_{n-1}, x_n,\underline y_n) \d x}{L^2(\d\underline z_2 \d \underline x_n\d \underline y_n)}
			\\
			&\quad +C \Nrm{\intd \frac{x-z_1}{\n{x-z_1}^{3}}\cdot\grad_1v(z_1, z_2) \, \omega(x_n, x)\, \Psi^{(n,m)}(\underline{x}_{n-1}, x_n,\underline y_n) \d x}{L^2(\d\underline z_2 \d \underline x_n\d \underline y_n)}.
		\end{align*}
		For the first term, we have 
		\begin{multline*}
			\Nrm{v(z_1, z_2)\, \omega(x_n, z_2) \, \Psi^{(n,m)}(\underline{x}_{n-1}, x_n,\underline y_n)}{L^2(\d\underline z_2 \d \underline x_n\d \underline y_n)}
			\\
			\le \Nrm{v}{L^\infty_xL^2_y} \Nrm{\omega}{L^\infty_xL^2_y} \Nrm{\Psi^{(n)}}{L^2(\d \underline x_n\d \underline y_n)}
			\le C \Nrm{\rho}{L^\infty}^\frac{1}{2} \Nrm{\omega\n{\opp}^2}{2} \Nrm{\Psi^{(n)}}{L^2(\d \underline x_n\d \underline y_n)}.
		\end{multline*}
		For the second term, we have 
		\begin{align*}
			&\Nrm{\intd \frac{x-z_1}{\n{x-z_1}^{3}}\cdot\grad_1v(z_1, z_2) \, \omega(x_n, x) \, \Psi^{(n,m)}(\underline{x}_{n-1}, x_n,\underline y_n) \d x}{L^2(\d\underline z_2 \d \underline x_n\d \underline y_n)}
			\\
			&\le C \sup_{z_1}\!\(\intd \frac{\rho(x)^\frac{1}{2}}{\n{x-z_1}^2} \d x \) \Nrm{\rho}{L^\infty}^\frac{1}{2}\Nrm{\grad_1 v}{2}\Nrm{\Psi^{(n)}}{L^2(\d \underline x_n\d \underline y_n)}.
		\end{align*}
		where the first integral term is controlled by $\Nrm{\rho}{L^\infty}^\frac{1}{2}+\Nrm{\rho}{L^1}^\frac{1}{2}$. The case when $0<a<1$ is similar, except that when $a\leq 1/2$, we need to estimate the last quantity with moments in $x$. Thus, it follows that 
		\begin{equation*}
			\Nrm{\sfI_5\Psi}{\cG}\le C_T \Nrm{\Psi}{\cG_2}
		\end{equation*}
		which completes the proof. 
	\end{proof}
	
{\bf Acknowledgments.} J.C. was supported by the NSF through the RTG grant DMS-RTG 184031. L.L. has received funding from the European Research Council (ERC) under the European Union’s Horizon 2020 research and innovation program (grant agreement No 865711). C.S. acknowledges the support of the Swiss National Science Foundation through the Eccellenza project PCEFP2\_181153 and of the NCCR SwissMAP.


\bibliographystyle{abbrv} 
\bibliography{Vlasov}

\end{document}